\documentclass[12pt]{amsart}
\usepackage[utf8]{inputenc}
\usepackage{amsmath,amsthm,amsfonts,amssymb,amsxtra}
\usepackage{mathrsfs}

\usepackage{ amsmath, amsthm, amsfonts, amssymb, color}
\usepackage{mathrsfs}
\usepackage{amsfonts, amsmath}
\usepackage{amsmath,amstext,amsthm,amssymb,amsxtra}
\usepackage{txfonts} %also pxfonts
\usepackage[colorlinks, citecolor=blue,pagebackref,hypertexnames=false]{hyperref}
\allowdisplaybreaks
\usepackage{pgf,tikz}
\usetikzlibrary{calc}
\usepackage{caption} % for \captionof
\usepackage{enumitem}

\newcommand\R{\mathbb{R}}

\newcommand\dd{\,{\rm d}}
\newcommand{\cdo}{\, \cdot \,}

\newcommand{\supp}{{\rm supp}{\hspace{.05cm}}}

\newcommand{\dist}{\operatorname{dist}}
\newcommand{\diver}{\operatorname{div}}
\newcommand{\cA}{\mathcal A}

\newcommand{\ip}[2]{\left\langle #1,#2\right\rangle}
\newcommand{\norm}[2]{\left\lVert #1\right\rVert_{#2}}
\newcommand{\abs}[1]{\left\lvert #1\right\rvert}

\newtheorem{theorem}{Theorem}[section]
\newtheorem{proposition}[theorem]{Proposition}
\newtheorem{corollary}[theorem]{Corollary}
\newtheorem{lemma}[theorem]{Lemma}

\newtheorem{problem}[theorem]{Problem}

\theoremstyle{remark}
\newtheorem{remark}[theorem]{Remark}

\begin{document}

\title{Stability of the Riesz-type inequalities under gluing}
\author{Dangyang He}  
\address{Dangyang He, Department of Mathematics, Sun Yat-sen (Zhongshan) University, Guangzhou, 510275, P.R. China\\
Formerly Department of Mathematics, Macquarie University, Sydney, Australia}
\email{hedy28@mail.sysu.edu.cn, dangyang.he@hdr.mq.edu.au}



\begin{abstract}
We establish stability principles under compact gluing for the reverse
square-root inequality and the Riesz transform.  Below a common Sobolev
dimension, the reverse estimate passes from finitely many model manifolds
to their connected sum. A boundary-compatible Poisson parametrix also proves the Neumann reverse inequality for every $1<p<\infty$ for real symmetric uniformly elliptic divergence-form operators on exterior Lipschitz domains. For the forward transform, we glue a bounded-core model and an independent whole-space end model that agree with the coefficients on overlapping regions.  A homogeneous Fredholm construction then resolves the form-valued error.  The resulting range is controlled by the two model exponents, with the additional restriction $p<n$ for Dirichlet data. More broadly, our method illuminates the interplay between the bounded interior and exterior problems, and provides a flexible framework for studying other boundary value problems.

\end{abstract}

\maketitle

\hypersetup{linkcolor=red}
\tableofcontents
\hypersetup{linkcolor=red}

\section{Introduction}

\subsection*{Background}

For a nonnegative Laplace--Beltrami operator $\Delta$ on a complete
Riemannian manifold, the representation theorem for closed forms yields
\[
 \|\Delta^{1/2}f\|_2=\|df\|_2,
 \qquad
 \|d\Delta^{-1/2}u\|_2=\|u\|_2
\]
for $u$ in the closure of the range of $\Delta$; the kernel is treated in
the usual orthogonal quotient. For a real symmetric divergence-form operator $L=-\diver(A\nabla)$, the corresponding statement is
\[
 \|L^{1/2}f\|_2^2
   =\int A\nabla f\mathbin{\cdo}\nabla f,
\]
and the exact first-order transform is $A^{1/2}\nabla L^{-1/2}$.  Uniform
ellipticity implies equivalence, for the
Euclidean-gradient expressions $\|L^{1/2}f\|_2$ and $\|\nabla f\|_2$.
This distinction is minor in the present self-adjoint setting but important
historically.  For
complex, possibly nonsymmetric, bounded coefficients the assertion that
$D(L^{1/2})=W^{1,2}$ is the Kato square-root problem.  Its solution on
$\R^n$ is due to Auscher, Hofmann, Lacey, McIntosh and Tchamitchian
\cite{AHLMT02}; the corresponding $L^2$ theory for strongly Lipschitz
domains was proved by Auscher and Tchamitchian
\cite{AuscherTchamitchian03}.  The reverse Riesz estimate studied here may therefore be viewed as one $L^p$ side of a square-root norm equivalence, but our operators remain real and symmetric.

For $1<p<\infty$, the two estimates are
\begin{align}
 \|\nabla L^{-1/2}u\|_{L^p}&\leq C_p\|u\|_{L^p},
 \tag{$\mathrm R_p$}\label{eq:R-p}\\
 \|L^{1/2}f\|_{L^p}&\leq C_p\|\nabla f\|_{L^p}.
 \tag{$\mathrm{RR}_p$}\label{eq:RR-p}
\end{align}
The first is the Riesz-transform estimate and the second its reverse.  In
the geometric case one replaces $\nabla$ by $d$.  In the divergence-form
case ellipticity is absorbed into the constants.  By duality,
\eqref{eq:R-p} implies \((\mathrm{RR}_{p'})\), but
there is no converse in general.  In particular, the reverse problem below
two cannot be recovered from the universally available forward estimates
at exponents at most two.

The Euclidean Riesz transforms are classical Calder\'on--Zygmund operators.
On manifolds and domains, however, their boundedness reflects heat-kernel
regularity, Poincar\'e inequalities, and the global geometry.  Coulhon and
Duong established the robust low-exponent theory \cite{CoulhonDuong}; the
higher-exponent problem and its relation to reverse H\"older estimates were
developed in \cite{ACDH04,AuscherCoulhon05}.  The reverse
inequality is governed by a different mechanism.  In particular, at low
frequency the forward transform and the square root may respond differently
to the same geometry; see \cite{AuscherCoulhon05,KillipVisanZhang16}.  A
typical obstruction to the forward transform is carried by a leading
harmonic term in its kernel.  After the square root is written in bilinear
form and paired with the gradient, this contribution can instead be removed
by a \emph{harmonic annihilation} argument; see \cite{He26}.

That difference is particularly sharp on spaces with several ends.
Nonconstant harmonic functions connecting distinct ends may create the
leading low-frequency obstruction to the forward transform, whereas the
same contribution disappears in the bilinear argument for the reverse
inequality.  This accounts for the greater stability of the reverse estimate
under compact gluing.  Exterior domains exhibit the boundary counterpart of
this geometry: after separating a neighborhood of the obstacle from the
Euclidean region at infinity, the exterior problem is assembled from a
bounded interior problem and a whole-space problem.  Compact gluing therefore
provides a natural link between analysis on manifolds with ends and
boundary-value problems on Euclidean domains.

Carron, Coulhon and Hassell determined the obstruction created by harmonic
functions connecting Euclidean ends \cite{CarronCoulhonHassell06}; Carron
then proved a general forward stability theorem by a Poisson-semigroup
parametrix \cite{Carron07}.  Exterior domains exhibit a related boundary
version of the phenomenon.  Hassell and Sikora found the dimensional
Dirichlet obstruction in a radial model \cite{HassellSikora09}, and Killip,
Visan and Zhang obtained sharp Sobolev and Riesz-transform results
outside smooth convex obstacles \cite{KillipVisanZhang16}.  For real
symmetric divergence-form operators on exterior Lipschitz domains, Jiang
and Lin obtained reverse-H\"older characterizations and several VMO and
boundary-value consequences \cite{JiangLin24}.  Under additional
coefficient and boundary hypotheses, Jiang and Yang later described the
high-exponent Dirichlet defect space \cite{JiangYang25}.

This paper presents three main results.
\begin{enumerate}[label=\textup{(\roman*)}]
\item The reverse estimate \eqref{eq:RR-p} is stable under a finite
      connected-sum construction below a common Sobolev dimension greater than $2$. A real-interpolation argument also settles the critical case at ``dimension'' $2$.
\item For the exterior Neumann realization, we        establish \eqref{eq:RR-p} for every             $1<p<\infty$, thereby confirming a              conjecture of Auscher and Tchamitchian
      \cite{AuscherTchamitchian01}. The proof uses an explicit,direct Poisson argument compatible with the boundary condition.
\item Armed with a homogeneous Fredholm gluing        theorem that resolves the form error of an
      exterior Poisson parametrix of the Riesz transform on exterior Lipschitz domains, we establish a criterion which transfers the whole-space and bounded-core forward estimates to all exterior-domain ranges in
      Theorem~\ref{thm:intro-forward}.  The two models need only agree with the exterior coefficients on their respective regions of use.
\end{enumerate}

\subsection*{Reverse stability on connected sums}

Let $(M_j,g_j)$, $1\leq j\leq \ell$, be complete Riemannian manifolds of
the same topological dimension.  A complete manifold $M$ is a connected
sum of the $M_j$ if there are compact sets $K\subset M$ and
$K_j\subset M_j$ such that
\[
 M\setminus K=\bigsqcup_{j=1}^{\ell}E_j,
 \qquad E_j\simeq M_j\setminus K_j
\]
isometrically.

\begin{theorem}
\label{thm:intro-connected}
Let $(M_j,g_j)$, $1\leq j\leq\ell$, be complete Riemannian manifolds
whose Ricci curvature is bounded below, and let $\nu\geq2$.
If $\nu>2$, assume that every $M_j$ satisfies
\[
 \|h\|_{L^{2\nu/(\nu-2)}(M_j)}
 \leq C_j\|dh\|_{L^2(M_j)},
 \qquad h\in C_c^\infty(M_j).
\]
If $\nu=2$, assume instead that every $M_j$ satisfies the Nash inequality
\[
 \|h\|_{L^2(M_j)}^4
 \leq C_j\|dh\|_{L^2(M_j)}^2\|h\|_{L^1(M_j)}^2,
 \qquad h\in C_c^\infty(M_j).
\]
Let $M$ be a connected sum of the $M_j$.
If \eqref{eq:RR-p} holds on every $M_j$ for some $1<p<\nu$,
then it holds on $M$.
\end{theorem}

The Nash inequality is a natural replacement for the homogeneous
$L^2$ Sobolev inequality when $\nu\leq2$.
For $\nu>2$, its dimensional form
\[
 \|h\|_2^{2+4/\nu}
 \leq C\|dh\|_2^2\|h\|_1^{4/\nu}
\]
is equivalent to the homogeneous $L^2$ Sobolev inequality with exponent
$2\nu/(\nu-2)$; see \cite{BCLS95}.

% \begin{theorem}
% \label{thm:intro-connected}
% Let $(M_j,g_j)$, $1\leq j\leq\ell$, be complete Riemannian manifolds whose
% Ricci curvature is bounded below.  Assume that there is a common
% $\nu>2$ such that, for every $j$,
% \[
%  \|h\|_{L^{2\nu/(\nu-2)}(M_j)}
%  \leq C_j\|dh\|_{L^2(M_j)},
%  \qquad h\in C_c^\infty(M_j).
% \]
% Let $M$ be a connected sum of the $M_j$.  If \eqref{eq:RR-p} holds on
% every $M_j$ for some $1<p<\nu$, then it holds on $M$.
% \end{theorem}

Theorem~\ref{thm:intro-connected} is an abstract model-to-sum implication.
The closest forward stability result can be stated, in the present
notation, as follows.

\begin{theorem}\cite[Theorem~1.3]{Carron07}
Let $M_0$ have Ricci curvature bounded below and satisfy the $L^2$ Sobolev
inequality with dimension $\nu>3$.
Fix $p$ with $\nu/(\nu-1)<p<\nu$.  If \eqref{eq:R-p} holds on $M_0$,
then it also holds on every complete manifold
isometric at infinity to a finite disjoint union of copies of $M_0$.
\end{theorem}

A related compact-perturbation theorem
for one-ended manifolds is due to Devyver \cite{Devyver15}.  Carron's theorem
assumes $\nu>3$ and treats the interval $\nu/(\nu-1)<p<\nu$.  By contrast,
Theorem~\ref{thm:intro-connected} requires only $\nu\ge 2$, with the Nash assumption when $\nu=2$, and transfers the
reverse estimate throughout $1<p<\nu$.  This illustrates that the reverse
inequality is more stable under compact gluing than the forward transform.
The phenomenon is consistent with \cite{He26}, where the reverse inequality
is proved for every $1<p<\infty$ on finite connected sums of product ends
$\R^{n_j}\times\mathcal M_j$, including ends with different volume-growth
dimensions.  In the same setting, the forward results presently available
in \cite{CarronCoulhonHassell06,HassellSikora09} generally occupy only a finite
interval of exponents.

Theorem~\ref{thm:intro-connected} applies beyond the product class.
For example, Corollary~\ref{cor:nilpotent-ends} gives the reverse inequality
for all $1<p<\infty$ on $\R^3\#\mathrm{Heis}_3$, where
$\mathrm{Heis}_3$ carries a left-invariant Riemannian metric.  The two
models have volume growth of degrees three and four at infinity.  The
Heisenberg end is not a Euclidean--compact product end, and the connected
sum is not globally doubling. 

The familiar Euclidean example is best read in this comparative sense.

\begin{corollary}\label{cor:intro-euclidean-sum}
Let $M=\R^n\#\cdots\#\R^n$, $n\geq2$, with an arbitrary smooth compact
gluing.  Then \eqref{eq:RR-p} holds on $M$ for every $1<p<\infty$.
\end{corollary}

Below $n$ this follows from Theorem~\ref{thm:intro-connected}; at and above
$n$ it follows by duality from the known low-exponent forward transform; see \cite{CarronCoulhonHassell06}. The conclusion also follows from the product-end theorem \cite{He26}; the corollary records the abstract transfer theorem in its simplest
model and is not asserted as an independent novelty.

\subsection*{Exterior Lipschitz domains}

Let $\Omega\subset\R^n$, $n\geq2$, be connected and Lipschitz, with
nonempty compact complement.  Let
\[
 L=-\diver(A\nabla)
\]
on $\R^n$, where $A$ is real, symmetric, bounded and uniformly elliptic,
and denote its Dirichlet and Neumann realizations on $\Omega$ by
$L_D$ and $L_N$.

It is well-known that the Riesz-transform estimates also provide a direct connection with
regularity for elliptic boundary-value problems. For $B\in\{D,N\}$,
boundedness of $\nabla L_B^{-1/2}$ on both $L^p(\Omega)$ and
$L^{p'}(\Omega)$ yields, by factorization and duality, the estimate
\[
 \|\nabla u\|_{L^p(\Omega)}
 \leq C_p\|F\|_{L^p(\Omega)}
\]
for the variational energy solution of boundary problems:
\begin{equation*}
    \begin{cases}
        L_D u = - \diver F & \textrm{in} \quad \Omega,\\
        u = 0 & \textrm{on}\quad \partial \Omega.
    \end{cases} \qquad 
    \begin{cases}
        L_N u = - \diver F & \textrm{in} \quad \Omega,\\
        \mathbf n \cdot A \nabla u = \mathbf n \cdot F & \textrm{on}\quad \partial \Omega,
    \end{cases}
\end{equation*}
where $\mathbf n$ is the outward unit normal and the Neumaan condition
is understood weakly; Neumann solutions are considered modulo constants; see \cite{Byun05, Geng12} and references therein. From this viewpoint, our gluing construction separates the
contribution of the boundary, encoded by the bounded-core problem,
from the behavior of the operator at infinity.

Our next result, inspired by the method used to prove Theorem~\ref{thm:intro-connected}, confirms a conjecture of Auscher and Tchamitchian \cite[Remark 12]{AuscherTchamitchian01} for the Neumann operator.

\begin{theorem}
\label{thm:intro-neumann-reverse}
For every $1<p<\infty$,
\[
 \|L_N^{1/2}f\|_{L^p(\Omega)}
 \leq C_p\|\nabla f\|_{L^p(\Omega)},
 \qquad f\in C_c^\infty(\overline\Omega).
\]
Equivalently, $L_N^{1/2}$ extends continuously from the completion of
$C_c^\infty(\overline\Omega)$ in the gradient norm, modulo constants, to
$L^p(\Omega)$.
\end{theorem}

The corresponding Dirichlet reverse inequality was obtained by Jiang--Lin \cite[Theorem~1.3]{JiangLin24} and is not reproved here. We mention that the above Neumann conclusion was also recorded by Jiang and Yang without proof; see \cite[Remark~1.3]{JiangYang25}.  Theorem
\ref{thm:intro-neumann-reverse} supplies an independent direct proof at the
level of the Neumann form.  The argument glues the whole-space and bounded
interior problems through a smooth core and transfers the model information
to the exterior domain; the same principle will also be used for the forward
problem later.  The same boundary condition is imposed on the physical and
artificial components of the core boundary.  A central feature of the
construction is the explicit removal of the constant Neumann mode.  

For the forward problem, choose a bounded Lipschitz core $\Omega_c$ whose
boundary is the union of the physical boundary and an artificial sphere,
and impose the same boundary condition $B\in\{D,N\}$ on both pieces.  Let
$L_{c,B}$ be this pure core realization and set
\begin{align*}
 p_L&=\sup\bigl\{r>2:\nabla L^{-1/2}\text{ is bounded on every }
                     L^q(\R^n),\ 2\leq q<r\bigr\},\\
 p_{c,B}&=\sup\bigl\{r>2:\nabla(1+L_{c,B})^{-1/2}\text{ is bounded on every }
                     L^q(\Omega_c),\ 2\leq q<r\bigr\},\\
 P_B&=\min\{p_L,p_{c,B}\}.
\end{align*}

\begin{theorem}
\label{thm:intro-forward}
With $P_B$ as above, the following hold.
\begin{enumerate}[label=\textup{(\roman*)}]
\item If $n\geq2$, $\nabla L_N^{-1/2}$ is bounded on $L^p(\Omega)$ for
      every $1<p<P_N$.
\item If $n\geq3$, $\nabla L_D^{-1/2}$ is bounded on $L^p(\Omega)$ for
      every $1<p<\min\{n,P_D\}$.
\end{enumerate}
More generally, let $0<R_-<R_c<R_+$, with
$\R^n\setminus\Omega\subset B(0,R_-)$, and suppose that
\begin{equation}\label{eq:overlapping-models}
 A=A_0\quad\text{on }\Omega\cap B(0,R_+),
 \qquad
 A=A_\infty\quad\text{on }\Omega\setminus\overline{B(0,R_-)}.
\end{equation}
Here $A_0$ on $B(0,R_+)$ and $A_\infty$ on $\R^n$ are real,
symmetric, bounded and uniformly elliptic.  Let $L_\infty=-\diver(A_\infty\nabla)$ on $\R^n$, and
let $L_{0,c,B}$ be the pure $B$ realization of $-\diver(A_0\nabla)$ on
$\Omega_c=\Omega\cap B(0,R_c)$.  Define
\begin{align*}
 p_\infty&=\sup\bigl\{r>2:\nabla L_\infty^{-1/2}\text{ is bounded on every }
                  L^q(\R^n),\ 2\leq q<r\bigr\},\\
 p_{0,B}&=\sup\bigl\{r>2:\nabla(1+L_{0,c,B})^{-1/2}\text{ is bounded on every }
                  L^q(\Omega_c),\ 2\leq q<r\bigr\},\\
 \widehat P_B&=\min\{p_\infty,p_{0,B}\}.
\end{align*}
Then \textup{(i)} and \textup{(ii)} hold with $P_B$ replaced by
$\widehat P_B$.
\end{theorem}

The exponents $p_L$ and $p_{c,B}$ reflect limitations of the whole-space
and bounded-core problems, respectively. Kenig's counterexample
\cite[Section~4.2.2, Theorem~7]{AuscherTchamitchian98} shows that
$p_L$ can be arbitrarily close to $2$, while the estimates of Shen
and Geng \cite{Shen05,Geng12} give bounded-core ranges depending
on coefficient regularity and boundary geometry.

\begin{remark}\label{rem:end-model}
The relevant overlap is between the region where the core model is used
and the region where the end model is used.  In
\eqref{eq:overlapping-models}, it contains the open annulus
$B(0,R_+)\setminus\overline{B(0,R_-)}$, on which
$A_0=A_\infty$ almost everywhere.  The positive gaps between the radii
allow the cutoff supports to be separated.  
The balls are a convenient choice: the same proof works for a bounded
Lipschitz core with a smooth artificial boundary whenever the matching
regions admit the separated cutoffs and connected transition sets used
in the proof.  No whole-space estimate for an extension of $A_0$ is required.  
\end{remark}

The Fredholm construction is the main step in the forward argument.
The model inverses produce left and right parametrices on homogeneous
Sobolev spaces.  Local Rellich compactness controls their form-valued
remainders even for measurable coefficients.  In the Neumann case,
balanced localized data and fixed representatives remove the constant
modes.  The index is transported from the energy exponent along the
interpolation interval, and compatibility of the model inverses promotes
the remaining kernel to the energy space, where uniqueness applies.
This constructs the exterior inverse from the two model bounds before
the exterior Riesz estimate is proved.  Sections~\ref{sec:forward-parametrix}--\ref{sec:proof-forward} then use that
inverse to resolve the integrated Poisson error.

The quantities $p_L$ and $p_{c,B}$ come from two independent model problems.  In the whole space, the upper range of \eqref{eq:R-p} is usually described through reverse H\"older estimates for $L$-harmonic functions;
see, for example, \cite{Shen05,CJKS,Jiang21}.  On bounded Lipschitz domains, the corresponding Dirichlet and Neumann ranges follow
from the boundary regularity theories of \cite{Shen05, JerisonKenig95} and \cite{Geng12,Zanger2000}, respectively.  Thus $p_L$ records the information at infinity, whereas
$p_{c,B}$ records the interior boundary-value problem.  The content of
Theorem~\ref{thm:intro-forward} is that these two inputs alone control the
exterior transform, subject only to the natural Dirichlet restriction
$p<n$.  Under \eqref{eq:overlapping-models}, these inputs are
$p_\infty$ and $p_{0,B}$.  In particular, VMO regularity need only be
assumed on the core model; the coefficients farther out are governed by
the chosen end model.  

An immediate consequence of Theorem~\ref{thm:intro-neumann-reverse} and Theorem~\ref{thm:intro-forward} is the following.

\begin{corollary}
\label{cor:intro-neumann-equivalence}
For every $1<p<P_N$,
\[
 \|L_N^{1/2}f\|_{L^p(\Omega)}
 \simeq
 \|\nabla f\|_{L^p(\Omega)},
 \qquad f\in C_c^\infty(\overline\Omega).
\]
Under \eqref{eq:overlapping-models}, the same conclusion holds for
$1<p<\widehat P_N$.
\end{corollary}

\subsection*{Comparison with previous results}

Closely related to Theorem~\ref{thm:intro-forward} is the recent work
of Jiang and Lin \cite{JiangLin24}. They characterize the boundedness
of the Dirichlet and Neumann Riesz transforms on exterior Lipschitz
domains through reverse H\"older estimates for local harmonic
functions satisfying the corresponding boundary conditions.
Combining these characterizations with the bounded-domain regularity
theory of Shen \cite{Shen05} and Geng \cite{Geng12}, together with
the results of Auscher--Qafsaoui \cite{AuscherQafsaoui} and
Auscher--Tchamitchian \cite{AuscherTchamitchian01}, they obtain
the following VMO ranges. Here $A\in\mathrm{VMO}(\R^n)$ means
that each entry belongs to BMO and
\[
 \lim_{r\downarrow0}\sup_{x\in\R^n}
 \fint_{B(x,r)}|A(y)-A_{B(x,r)}|\,\dd y=0.
\]

\begin{corollary}[Jiang--Lin~{\cite[Theorems~1.4 and~1.5]{JiangLin24}}]
\label{cor:intro-VMO}
Assume $A\in\mathrm{VMO}(\R^n)$.
\begin{enumerate}[label=\textup{(\roman*)}]
\item If $n\geq3$ and $\Omega$ is Lipschitz, then
\[
 \nabla L_D^{-1/2}\text{ is bounded on }L^p(\Omega),
 \qquad 1<p<\min\{n,p_L,3+\varepsilon\}.
\]
If $\partial\Omega$ is $C^1$, the upper endpoint is
$\min\{n,p_L\}$.
\item If $\Omega$ is Lipschitz, then
\[
 \nabla L_N^{-1/2}\text{ is bounded on }L^p(\Omega),
 \quad
 \begin{cases}
  1<p<\min\{p_L,3+\varepsilon\},&n\geq3,\\
  1<p<\min\{p_L,4+\varepsilon\},&n=2.
 \end{cases}
\]
If $\partial\Omega$ is $C^1$, the range is $1<p<p_L$.
\end{enumerate}
Here $\varepsilon>0$ depends on the ellipticity and Lipschitz data, and no
endpoint is included.
\end{corollary}

We provide an alternative proof of this corollary in
Section~\ref{sec:applications}. Theorem~\ref{thm:intro-forward}
derives the exterior estimate directly from the whole-space and
pure bounded-core Riesz bounds. These enter as separate inputs
through $p_L$ and $p_{c,B}$, with the additional restriction
$p<n$ in the Dirichlet case.

The distinction is in the mechanism of proof. Jiang--Lin establish
reverse H\"older estimates on the exterior domain at all scales.
Our Poisson parametrix separates the two model contributions,
and a homogeneous Fredholm argument constructs the exterior
inverse needed to correct the resulting error. This formulation
allows the whole-space and boundary theories to be used
independently. In particular, different coefficient assumptions
may be imposed on the core and at infinity, provided that the
models agree on the overlap as in
Theorem~\ref{thm:intro-forward}.

\subsection*{Organization}

Part~1 proves the reverse results.  Section~\ref{sec:connected-stability}
develops the connected-sum argument and proves Theorem
\ref{thm:intro-connected}; Section~\ref{sec:exterior-reverse} treats the
Neumann exterior problem and proves Theorem
\ref{thm:intro-neumann-reverse}.  Part~2 constructs the forward Poisson
parametrix, proves a homogeneous Fredholm theorem, and uses it to resolve
the integrated form error in every dimension, yielding Theorem
\ref{thm:intro-forward}.  Part~3 records the coefficient applications,
the application to nilpotent model ends, and additional stability problems.

\part{Stability of the square-root inequalities}\label{part1}

All form identities in Parts~\ref{part1} and~\ref{part2} are first
verified on the indicated smooth form cores.  Whenever both sides extend
continuously to the displayed Sobolev spaces, passage to the completion is
understood and will not be repeated.

\section{Stability on connected sums}
\label{sec:connected-stability}

Throughout this section, $\Delta$ denotes the nonnegative Laplace--Beltrami
operator on a complete Riemannian manifold.  We use $d$ for the exterior
differential on functions.  The strategy is Carron's: solve the Poisson
extension problem on a compact core and on each model end, glue the solutions,
and control the compactly supported cylinder error.

\subsection{Carron's parametrix for the Poisson semigroup}

We first recall the semigroup consequence of a Sobolev inequality; see also \cite[Lemma~2.3]{Carron07}.

\begin{lemma}\label{lem:connected-poisson-smoothing}
Let $X$ be a complete Riemannian manifold and let $\nu\geq2$.
Suppose that either $\nu>2$ and
\[
 \|h\|_{L^{2\nu/(\nu-2)}(X)}
 \leq C\|dh\|_{L^2(X)},
 \qquad h\in C_c^\infty(X),
\]
or $\nu=2$ and
\[
 \|h\|_{L^2(X)}^4
 \leq C\|dh\|_{L^2(X)}^2\|h\|_{L^1(X)}^2,
 \qquad h\in C_c^\infty(X).
\]
Then, for $1\leq r\leq q\leq\infty$ and $s>0$,
\[
 \bigl\|e^{-s\sqrt{\Delta_X}}\bigr\|_{L^r(X)\to L^q(X)}
 \leq C_{r,q}s^{-\nu(1/r-1/q)}.
\]
\end{lemma}

\begin{proof}
In either case, the Sobolev--Nash ultracontractivity criterion gives
the heat-kernel bound
\[
 0\leq h_t^X(x,y)\leq Ct^{-\nu/2},
 \qquad t>0;
\]
see \cite{VaropoulosSaloffCosteCoulhon92}.
Interpolation with the $L^r$ contractivity of the heat semigroup yields
\[
 \|e^{-t\Delta_X}\|_{L^r(X)\to L^q(X)}
 \leq C_{r,q}t^{-\frac{\nu}{2}(1/r-1/q)}.
\]
The conclusion follows by inserting this estimate into the
subordination formula
\[
 e^{-s\sqrt{\Delta_X}}
 =\frac{s}{2\sqrt{\pi}}
   \int_0^\infty
   e^{-s^2/(4t)}e^{-t\Delta_X}\,\frac{dt}{t^{3/2}}.
\]
\end{proof}

% \begin{lemma}\label{lem:connected-poisson-smoothing}
% \cite[Lemma~2.3]{Carron07}
% Let $X$ be complete and suppose that, for some $\nu>2$,
% \[
%  \norm h{L^{2\nu/(\nu-2)}(X)}\leq C\norm{dh}{L^2(X)},
%  \qquad h\in C_c^\infty(X).
% \]
% Then, for $1\leq r\leq q\leq\infty$ and $s>0$,
% \[
%  \norm{e^{-s\sqrt{\Delta_X}}}{L^r(X)\to L^q(X)}
%  \leq C_{r,q}s^{-\nu(1/r-1/q)}.
% \]
% \end{lemma}

Write
\[
 M\setminus K=\bigsqcup_{j=1}^{\ell}E_j,
 \qquad E_j\simeq M_j\setminus K_j,
\]
where $M_j$ are complete model manifolds.  We assume throughout
that the finitely many models have Ricci curvature bounded below and
satisfy the corresponding Sobolev or Nash hypothesis of Lemma~\ref{lem:connected-poisson-smoothing} with the same exponent $\nu \ge2$.
The constants may depend on $j$; since the number of ends is finite, their
maximum will be used without further comment.
Choose a bounded smooth neighborhood $\widetilde K$ of $K$, a partition of
unity $\rho_0+\sum_{j=1}^{\ell}\rho_j=1$, and cutoffs
$\phi_0,\ldots,\phi_\ell$ such that
\begin{align*}
 &\supp\rho_0,\supp\phi_0\subset\widetilde K,
 \qquad \supp\rho_j,\supp\phi_j\subset E_j\quad(1\leq j\leq\ell),\\
 &\phi_j\rho_j=\rho_j,
 \qquad \dist(\supp d\phi_j,\supp\rho_j)>0,
 \quad 0\leq j\leq\ell.
\end{align*}
For $j\geq1$, let $\Delta_j$ be the Laplacian on $M_j$.
Let $\Delta_0$ be the Dirichlet Laplacian on a bounded smooth core whose
artificial boundary has positive distance from the supports of $\rho_0$ and
$\phi_0$.  Thus $\Delta_0$ has a spectral gap.  Set
\[
 \mathcal S_su=\sum_{j=0}^{\ell}
 \phi_je^{-s\sqrt{\Delta_j}}(\rho_ju).
\]
Because $\phi_j\rho_j=\rho_j$, $\mathcal S_0u=u$.  In the distributional
sense on $(0,\infty)\times M$,
\[
 \left(-\partial_s^2+\Delta\right)\mathcal S_su
 =\sum_{j=0}^{\ell}E_{j,s}u,
 \qquad
 E_{j,s}u=[\Delta,\phi_j]
 e^{-s\sqrt{\Delta_j}}(\rho_ju).
\]
The errors are supported in a fixed compact set.  Since the Laplacian is
smooth, the commutator reads
\[
 E_{j,s}u=(\Delta\phi_j)e^{-s\sqrt{\Delta_j}}(\rho_ju)
 -2\,d\phi_j\mathbin{\cdo}d e^{-s\sqrt{\Delta_j}}(\rho_ju).
\]

The formal Green correction for this cylinder error motivates the gradient
parametrix displayed in the proof below.  We shall not integrate the Green
operator as a stand-alone zero-energy Bochner integral.  Instead, a
finite-cylinder Wronskian identity provides directly the only pairing needed
for the reverse inequality.  This keeps the two time boundaries visible and
avoids an unnecessary convergence question for the scalar correction.

The parametrix argument uses the following compact-annulus bound. Its proof from \cite[Lemma~3.1]{Carron07} works verbatim to the critical case $\nu=2$ with Nash inequality.

\begin{lemma}
\label{lem:connected-error-bound}
\cite[Lemma~3.1]{Carron07}
Let $1\leq a\leq\infty$.  For $u\in L^a(M)$ and
$0\leq j\leq\ell$,
\[
 \norm{E_{j,s}u}{L^1(M)}
 +\norm{E_{j,s}u}{L^b(M)}
 \leq C_{a,b}(1+s)^{-\nu/a}\norm u{L^a(M)},
 \qquad s\geq0,
\]
for every fixed $1\leq b\leq\infty$.
\end{lemma}

\begin{remark}
Carron states \cite[Lemma~3.1]{Carron07} with the $L^1+L^a$ norm for the
sum of the finitely many errors, but his proof estimates each error
separately.  Indeed, the separated Poisson estimates
\cite[Lemma~2.4]{Carron07} and the spectral gap of the bounded core imply
\[
 \|E_{j,s}u\|_{L^\infty(M)}
 \lesssim (1+s)^{-\nu/a}\|u\|_{L^a(M)}.
\]
Since all the errors are supported in one fixed compact set, H\"older's
inequality yields the termwise $L^1$ and $L^b$ bounds in
Lemma~\ref{lem:connected-error-bound}.
\end{remark}

\subsection{Sobolev inequality on connected sums}

The proof of the reverse theorem needs the $L^p$ Sobolev inequality on the
glued manifold.  We record exactly the stability statement used below.

\begin{proposition}\label{prop:connected-Sobolev}
Let $1<p<\nu$ and $\nu\ge 2$. Suppose
\begin{equation}\label{eq:model-p-Sobolev}
 \|h\|_{L^{p^*}(M_j)}
 \leq C_j\|dh\|_{L^p(M_j)},
 \qquad p^*=\frac{p\nu}{\nu-p},
 \qquad h\in C_c^\infty(M_j),\quad 1\leq j\leq\ell.
\end{equation}
Then their connected sum $M$ satisfies
\begin{equation}\label{eq:sum-p-Sobolev}
 \|f\|_{L^{p^*}(M)}
 \leq C\|df\|_{L^p(M)},
 \qquad f\in C_c^\infty(M).
\end{equation}

\end{proposition}

% \begin{proposition}\label{prop:connected-Sobolev}
% Let $1<p<\nu$.  Suppose
% \begin{equation}\label{eq:model-p-Sobolev}
%  \norm h{L^{p^*}(M_j)}\leq C_j\norm{dh}{L^p(M_j)},
%  \qquad p^*=\frac{p\nu}{\nu-p},
%  \qquad h\in C_c^\infty(M_j),\quad 1\leq j\leq\ell.
% \end{equation}
% Then their connected sum $M$ satisfies
% \begin{equation}\label{eq:sum-p-Sobolev}
%  \norm f{L^{p^*}(M)}\leq C\norm{df}{L^p(M)},
%  \qquad f\in C_c^\infty(M).
% \end{equation}
% \end{proposition}

\begin{proof}
Choose smooth functions $\chi_0,\ldots,\chi_\ell$ with
$\sum_{j=0}^\ell\chi_j=1$, where $\chi_0$ and every $d\chi_j$ are supported
in one fixed compact set $K_0$, and where $\chi_j$ is supported in $E_j$ for
$j\geq1$.  Put $f_j=\chi_jf$.  For $j\geq1$, identify $f_j$ with a compactly
supported function on $M_j$.  Applying \eqref{eq:model-p-Sobolev}, we obtain
\[
 \norm{f_j}{L^{p^*}(M_j)}
 \leq C_j\norm{d(\chi_jf)}{L^p(M_j)}
 \leq C\norm{df}{L^p(M)}+C\norm f{L^p(K_0)}.
\]
On a slightly larger compact neighborhood of $K_0$, the local Sobolev
inequality yields the same estimate for $f_0$.  There is no hidden exponent
restriction here: the validity of \eqref{eq:model-p-Sobolev} on a smooth
model forces $p^*$ not to exceed the local Sobolev exponent, and the same
local exponent is available on the compact core.  Thus, after enlarging
$K_0$ to $\widetilde K$, we have
\begin{align}
 \norm{f_j}{p^*}&\leq C\norm{df}{p}+C\norm f{L^p(\widetilde K)},
 \label{eq:sum-Sobolev-end}\\
 \norm{f_0}{p^*}&\leq C\norm{df}{p}+C\norm f{L^p(\widetilde K)}.
 \label{eq:sum-Sobolev-core}
\end{align}
It remains to control the compact norm.  Inequality
\eqref{eq:model-p-Sobolev} implies $p$-hyperbolicity of each $M_j$: if
$E\subset M_j$ is a compact set of positive measure, then H\"older's
inequality and \eqref{eq:model-p-Sobolev} imply
\[
 \norm h{L^p(E)}
 \leq |E|^{1/p-1/p^*}\norm h{L^{p^*}(M_j)}
 \leq C_E\norm{dh}{L^p(M_j)}.
\]
The compact-gluing stability of $p$-hyperbolicity
\cite[Corollary~2.1]{Devyver15} provides the corresponding compact Poincar\'e
inequality on $M$:
\begin{equation}\label{eq:compact-poincare-sum}
 \norm f{L^p(\widetilde K)}\leq C_{\widetilde K}\norm{df}{L^p(M)}.
\end{equation}
Combining \eqref{eq:sum-Sobolev-end}, \eqref{eq:sum-Sobolev-core}, and
\eqref{eq:compact-poincare-sum}, and then summing the finitely many pieces,
proves \eqref{eq:sum-p-Sobolev}.

% For the Nash assertion, let $\lambda_1(U)$ denote the first Dirichlet
% eigenvalue of a relatively compact open set $U$.
% The model Nash inequality and
% $\|h\|_1^2\leq\operatorname{vol}(U)\|h\|_2^2$ give
% \[
%  \lambda_1(U)\geq
%  \frac{c_j}{\operatorname{vol}(U)},
%  \qquad U\Subset M_j.
% \]
% Thus $M_j$ admits the Faber--Krahn function $\Lambda_j(v)=c_jv^{-1}$.
% By \cite[Theorem~3.4]{GrigoryanSaloffCoste16}, there exist
% $c>0$ and $Q>1$ such that $M$ admits the Faber--Krahn function
% \[
%  \Lambda(v)=c\min_{1\leq j\leq\ell}\Lambda_j(Qv)
%  =c'v^{-1}.
% \]
% The equivalence between this Faber--Krahn inequality and the Nash
% inequality, recalled in \cite[Section~2.1]{GrigoryanSaloffCoste16},
% proves \eqref{eq:sum-Nash-two}.

\end{proof}

Under either alternative in Theorem~\ref{thm:intro-connected}, this proposition applies
end by end.  Indeed, the $L^2$ Sobolev or Nash inequality on $M_j$ implies the
fractional-integration estimate
\[
 \norm{\Delta_j^{-1/2}g}{L^{p^*}(M_j)}
 \leq C_j\norm g{L^p(M_j)},
 \qquad 1<p<\nu,
\]
by the heat-semigroup proof of the Hardy--Littlewood--Sobolev theorem
\cite{Varopoulos85}.  Hence the assumed reverse inequality on $M_j$ yields
\[
 \norm h{L^{p^*}(M_j)}
 =\norm{\Delta_j^{-1/2}\Delta_j^{1/2}h}{L^{p^*}(M_j)}
 \leq C_j\norm{\Delta_j^{1/2}h}{L^p(M_j)}
 \leq C_j\norm{dh}{L^p(M_j)}.
\]
The spectral identity in the first equality follows from the triviality of
the $L^2$ kernel, which is a consequence of the $L^2$ Sobolev inequality.
Thus \eqref{eq:model-p-Sobolev} holds for every end model, with a common
exponent and, after taking a maximum, a common constant.
In particular, applying Proposition~\ref{prop:connected-Sobolev} with $p=2$
directly to the assumed $L^2$ inequalities shows that $M$ satisfies the
$L^2$ Sobolev inequality with dimension $\nu$; hence
Lemma~\ref{lem:connected-poisson-smoothing} applies to the global Laplacian
on $M$.
In particular,
\[
 \ker_{L^2}(\Delta_j)=\ker_{L^2}(\Delta)=\{0\},
 \qquad 1\leq j\leq\ell,
\]
while the bounded-core operator $\Delta_0$ has a spectral gap.  We use these
facts below without further comment.

\subsection{Proof of Theorem~\ref{thm:intro-connected}}

\begin{proof}[Proof of Theorem~\ref{thm:intro-connected}]
Fix $1<p<\nu$, let $f,u\in C_c^\infty(M)$, and normalize
$\|u\|_{p'}\leq1$.  It is enough to estimate
\[
 \mathcal B(f,u)=\ip{\Delta^{1/2}f}{u}.
\]
We obtain the cancellation on a finite cylinder.  Put
\[
 X_s=e^{-s\sqrt\Delta}f,\qquad Y_s=\mathcal S_su,\qquad
 E_su=\sum_{j=0}^{\ell}E_{j,s}u,
\]
and define
\[
 W(s)=\ip{X_s}{\partial_sY_s}-\ip{\partial_sX_s}{Y_s}.
\]
Fix \(0<\varepsilon<R<\infty\). Since
\[
 \partial_s^2X_s=\Delta X_s,\qquad
 \partial_s^2Y_s=\Delta Y_s-E_su
\]
in the form sense, the product rule yields
\[
 W'(s)=-\ip{X_s}{E_su}.
\]
Thus
\begin{equation}\label{eq:connected-truncated-Wronskian}
 W(R)-W(\varepsilon)
 =-\int_\varepsilon^R\ip{e^{-s\sqrt\Delta}f}{E_su}\,\dd s.
\end{equation}
The heat-semigroup bounds established above imply that the end Laplacians and $\Delta$ have trivial $L^2$ kernels; the bounded-core operator has a spectral gap.
The spectral theorem therefore implies that as $R\to \infty$,
\[
 X_R,\ \partial_RX_R\longrightarrow0\quad\text{in }L^2(M),
\]
and, for every \(j\), both
\(\phi_je^{-R\sqrt{\Delta_j}}(\rho_ju)\) and its \(R\)-derivative tend
to zero in \(L^2(M)\).  Hence \(W(R)\to0\).  At the lower endpoint,
strong form-domain continuity yields that as $\varepsilon \to 0$,
\[
 X_\varepsilon\to f,\qquad
 \partial_\varepsilon X_\varepsilon\to-\Delta^{1/2}f,\qquad
 Y_\varepsilon\to u,\qquad
 \partial_\varepsilon Y_\varepsilon
 \to-\sum_{j=0}^{\ell}\phi_j\Delta_j^{1/2}(\rho_ju)
\]
in \(L^2\).  Consequently
\[
 W(0)
 =-\sum_{j=0}^{\ell}\ip f{\phi_j\Delta_j^{1/2}(\rho_ju)}
   +\ip{\Delta^{1/2}f}{u}.
\]
The short and large time estimates below show that the integral in
\eqref{eq:connected-truncated-Wronskian} converges as
\(\varepsilon\downarrow0\) and \(R\uparrow\infty\).  Passing to those
limits yields
\begin{equation}\label{eq:connected-main-identity}
 \mathcal B(f,u)
 =\sum_{j=0}^{\ell}\ip f{\phi_j\Delta_j^{1/2}(\rho_ju)}
 +\sum_{j=0}^{\ell}\int_0^\infty
 \ip{e^{-s\sqrt\Delta}f}{E_{j,s}u}\,\dd s.
\end{equation}
% The identity has the following familiar parametrix interpretation.  At the
% level of gradient pairings, the formal formula to keep in mind is
% \[
%  d\Delta^{-1/2}u
%  \overset{\mathrm{formal}}{=}
%  \sum_{j=0}^{\ell}
%    d\bigl(\phi_j\Delta_j^{-1/2}(\rho_ju)\bigr)
%    -\sum_{j=0}^{\ell}dg_j .
% \]
% Here \(g_j\) is the Green correction for the cylinder error \(E_{j,s}u\).
% In the full range \(1<p<\nu\), we do not need a separately convergent
% scalar formula for \(g_j\): the finite-cylinder identity
% \eqref{eq:connected-truncated-Wronskian} defines precisely the pairing of
% \(dg_j\) with \(df\), and its limit is the last term in
% \eqref{eq:connected-main-identity}.  Thus
% \eqref{eq:connected-main-identity} is the rigorous realization of this gradient parametrix.
We estimate the two model terms.  Choose
$\widetilde\phi_0\in C_c^\infty(M)$ equal to one on $\supp\phi_0$.  The
compactly supported operator
$\widetilde\phi_0d\Delta_0^{-1/2}\rho_0$ is of order zero and hence is
bounded on every $L^q$, $1<q<\infty$.  Indeed, both cutoffs lie away from
the artificial boundary.  On this interior region the short-time
Dirichlet heat-kernel gradient has the standard Calder\'on--Zygmund size
and smoothness bounds, while the spectral gap makes the large-time part
smooth and exponentially decaying.  Integration of the heat representation
of \(\Delta_0^{-1/2}\) proves the localized bound.  Therefore
\begin{align}
 \abs{\ip f{\phi_0\Delta_0^{1/2}(\rho_0u)}}
 &=\abs{\ip{d(\phi_0f)}
 {\widetilde\phi_0d\Delta_0^{-1/2}(\rho_0u)}}\notag\\
 &\leq C\bigl(\norm{df}{p}+\norm f{L^p(\widetilde K)}\bigr)\norm u{p'}
 \leq C\norm{df}{p}\norm u{p'},
 \label{eq:connected-core-model}
\end{align}
where the last inequality is \eqref{eq:compact-poincare-sum}.  For every
end $1\leq j\leq\ell$, restriction to \(M_j\), self-adjointness,
and the model reverse inequality give
\begin{align}
 \abs{\ip f{\phi_j\Delta_j^{1/2}(\rho_ju)}}
 &=\abs{\ip{\Delta_j^{1/2}(\phi_jf)}{\rho_ju}_{M_j}}\notag\\
 &\leq C\norm{d(\phi_jf)}p\norm u{p'}
 \leq C\norm{df}p\norm u{p'}.
 \label{eq:connected-end-model}
\end{align}
The last step again uses \eqref{eq:compact-poincare-sum} on the support of
$d\phi_j$.

It remains to estimate the error integral in
\eqref{eq:connected-main-identity}.  For $0<s\leq1$, the Poisson semigroup
is contractive on $L^{p^*}$.  The Sobolev inequality on $M$ and the fixed
compact support of $E_{j,s}u$ therefore yield
\begin{align}
 \abs{\ip{e^{-s\sqrt\Delta}f}{E_{j,s}u}}
 &\leq
 \norm{e^{-s\sqrt\Delta}f}{L^{p^*}(M)}
 \norm{E_{j,s}u}{L^{(p^*)'}(M)}\notag\\
 &\leq C\norm f{p^*}\norm u{p'}
 \leq C\norm{df}p\norm u{p'}.
 \label{eq:connected-small-error}
\end{align}
For $s>1$ and $\nu>2$, combine
Lemma~\ref{lem:connected-poisson-smoothing},
Lemma~\ref{lem:connected-error-bound}, and
\eqref{eq:sum-p-Sobolev}:
\begin{align}
 \bigl|\langle e^{-s\sqrt\Delta}f,E_{j,s}u\rangle\bigr|
 &\leq
 \|e^{-s\sqrt\Delta}f\|_\infty\|E_{j,s}u\|_1
 \notag\\
 &\leq
 Cs^{-\nu/p^*-\nu/p'}\|f\|_{p^*}\|u\|_{p'}
 \leq Cs^{1-\nu}\|df\|_p\|u\|_{p'}.
 \label{eq:connected-large-error}
\end{align}
Here $\nu/p^*+\nu/p'=\nu-1$, and
$\int_1^\infty s^{1-\nu}\,ds<\infty$.

To this end, we assume $\nu=2$.
Fix $1<p<2$ and put $p^*=2p/(2-p)$.
By \cite[Theorem~4.3]{BCLS95} and
\eqref{eq:sum-p-Sobolev}, we get the following sharp Lorentz estimate
\begin{equation}\label{eq:connected-critical-Lorentz-Sobolev}
 \|f\|_{L^{p^*,p}(M)}\leq C\|df\|_{L^p(M)}.
\end{equation}
We use the following consequence of interpolation. Let $Y$ be a Banach space and suppose that a family of linear operators
$T_s$, $s\geq1$, satisfies
\[
 \|T_s\|_{L^1(M)\to Y}\leq Cs^{-2},
 \qquad
 \|T_s\|_{L^\infty(M)\to Y}\leq C.
\]
For every decomposition $h=h_1+h_\infty$, these bounds imply
\[
 \|T_sh\|_Y
 \leq C\bigl(s^{-2}\|h_1\|_1+\|h_\infty\|_\infty\bigr).
\]
Define real-interpolation functional (see \cite[Chapter~3,~5]{BerghLofstrom}) $K(t,h;L^1,L^\infty):= \inf_{h=h_1+h_\infty} \left(\|h_1\|_1+t\|h_\infty\|_\infty\right)$. Taking the infimum gives
\[
 \|T_sh\|_Y
 \leq Cs^{-2}K(s^2,h;L^1(M),L^\infty(M)).
\]
Since
$(L^1,L^\infty)_{1-1/r,b}=L^{r,b}$, the change of variables $t=s^2$
in the real-interpolation norm yields
\begin{equation}\label{eq:connected-critical-time-interpolation}
 \left(
 \int_1^\infty
 \bigl[s^{2/r}\|T_sh\|_Y\bigr]^b\,\frac{ds}{s}
 \right)^{1/b}
 \leq C_{r,b}\|h\|_{L^{r,b}(M)},
 \qquad 1<r<\infty,\quad 1\leq b<\infty.
\end{equation}
Apply \eqref{eq:connected-critical-time-interpolation} first to
$T_s=e^{-s\sqrt\Delta}$ with $Y=L^\infty(M)$.
Its endpoint bounds follow from
Lemma~\ref{lem:connected-poisson-smoothing} and contractivity.
With $(r,b)=(p^*,p)$, we obtain
\[
 \left(
 \int_1^\infty
 \bigl[s^{2/p^*}\|e^{-s\sqrt\Delta}f\|_\infty\bigr]^p
 \,\frac{ds}{s}
 \right)^{1/p}
 \leq C\|f\|_{L^{p^*,p}(M)}
 \leq C\|df\|_p.
\]
Next apply the same estimate to $T_s=E_{j,s}$ with $Y=L^1(M)$.
Lemma~\ref{lem:connected-error-bound}, at $a=1$ and $a=\infty$,
supplies the two endpoint bounds.
Taking $r=b=p'$ and using $L^{p',p'}=L^{p'}$ gives
\[
 \left(
 \int_1^\infty
 \bigl[s^{2/p'}\|E_{j,s}u\|_1\bigr]^{p'}
 \,\frac{ds}{s}
 \right)^{1/p'}
 \leq C\|u\|_{p'}.
\]
Since $2/p^* + 2/p'=1$, H\"older's inequality with respect to $ds/s$ now concludes
\begin{align*}
 &\int_1^\infty
 \bigl|\langle e^{-s\sqrt\Delta}f,E_{j,s}u\rangle\bigr|\,ds
 \notag\\
 &\quad\leq
 \int_1^\infty
 \bigl[s^{2/p^*}\|e^{-s\sqrt\Delta}f\|_\infty\bigr]
 \bigl[s^{2/p'}\|E_{j,s}u\|_1\bigr]\,\frac{ds}{s} \leq C\|df\|_p\|u\|_{p'}.
\end{align*}
Combining \eqref{eq:connected-core-model},
\eqref{eq:connected-end-model}, and
\eqref{eq:connected-small-error} with the corresponding large-time
estimate verifies
\[
 |\mathcal B(f,u)|\leq C\|df\|_p\|u\|_{p'}.
\]
Taking the $L^{p'}$ supremum proves \eqref{eq:RR-p} on $M$.
\end{proof}

\begin{remark}
The exponent $\nu=2$ is critical for the large-time error estimate
in this proof. Although the Nash-to-heat-kernel implication remains
valid for $\nu>0$, formally repeating the interpolation argument with
$1<p<\nu$ and $p^*=\nu p/(\nu-p)$ gives
\[
 \frac{\nu}{p^*}+\frac{\nu}{p'}=\nu-1,
\]
leaving the factor $s^{2-\nu}$ in the error pairing with respect to
$ds/s$. The Lorentz refinement therefore closes the argument at
$\nu=2$, but does not control this growing factor when $1<\nu<2$.
\end{remark}

\section{Stability on exterior Lipschitz domains}
\label{sec:exterior-reverse}

Let $\Omega\subset\R^n$ be an exterior Lipschitz domain and let
$A\colon\R^n\to\R^{n\times n}$ be real, symmetric, bounded, and uniformly
elliptic:
\[
 \lambda\abs{\xi}^2\leq A(x)\xi\mathbin{\cdo}\xi
 \leq\Lambda\abs\xi^2,
 \qquad \xi\in\R^n,\quad\text{for a.e. }x\in\R^n.
\]
The whole-space realization of $-\diver(A\nabla)$ is denoted by $L$.
On $\Omega$, let $L_D$ and $L_N$ be associated with
\begin{align*}
	\mathfrak q_D(f,g)&=\int_\Omega A\nabla f\mathbin{\cdo}\nabla\overline g,
&D(\mathfrak q_D)&=W_0^{1,2}(\Omega),\\
 \mathfrak q_N(f,g)&=\int_\Omega A\nabla f\mathbin{\cdo}\nabla\overline g,
&D(\mathfrak q_N)&=W^{1,2}(\Omega).
\end{align*}

\begin{remark}\label{rem:Sobolev-form-core}
For $1\le r<\infty$, we use the following conventions:
\[
 C_c^\infty(\overline\Omega)
 :=\{F|_\Omega:F\in C_c^\infty(\R^n)\},
\]
and
\[
 W_0^{1,r}(\Omega)
 =\overline{C_c^\infty(\Omega)}^{\,\|\cdot\|_{W^{1,r}(\Omega)}},
 \qquad
 W^{1,r}(\Omega)
 =\overline{C_c^\infty(\overline\Omega)}
 ^{\,\|\cdot\|_{W^{1,r}(\Omega)}}.
\]
Since $\Omega$ is Lipschitz, the latter space agrees with the usual
distributional Sobolev space.  Thus $C_c^\infty(\Omega)$ is a form core
for $L_D$, whereas $C_c^\infty(\overline\Omega)$ is a form core for $L_N$.

More generally, if $T$ is a nonnegative self-adjoint operator with closed
quadratic form $\mathfrak q_T$, a \emph{form core} for $T$ is a subspace of
$D(\mathfrak q_T)=D(T^{1/2})$ that is dense for the form
norm
\[
 \|v\|_{\mathfrak q_T}
 :=\bigl(\|v\|_2^2+\mathfrak q_T(v,v)\bigr)^{1/2}.
\]
In particular, $C_c^\infty(\R^n)$ is a form core for $L$,
and the analogous Dirichlet and Neumann test classes on $\Omega_c$ are
form cores for the bounded-core realizations.
\end{remark}

Since both \(\Omega\) and \(\R^n\) have infinite measure, the energy identity also
shows once and for all that
\begin{equation}\label{eq:exterior-L2-kernels}
 \ker_{L^2}L=\ker_{L^2}L_D=\ker_{L^2}L_N=\{0\}.
\end{equation}
Indeed, a form-kernel element has zero gradient and is therefore constant;
the Dirichlet trace removes that constant, and no nonzero constant belongs
to $L^2(\Omega)$ or $L^2(\mathbb R^n)$.

We next fix the kernel notation used throughout the paper.  Let
$h_t(x,y)$ and $h_t^B(x,y)$ be the integral kernels of $e^{-tL}$ and
$e^{-tL_B}$, respectively, and let $p_s(x,y)$ and $p_s^B(x,y)$ be the
kernels of $e^{-s\sqrt L}$ and $e^{-s\sqrt{L_B}}$, where
$B\in\{D,N\}$.  Thus
\[
 p_s(x,y)=\frac{s}{2\sqrt\pi}\int_0^\infty
 e^{-s^2/(4t)}h_t(x,y)\,\frac{\dd t}{t^{3/2}},
\]
and the same formula holds with $h_t,p_s$ replaced by $h_t^B,p_s^B$.
For the shifted bounded-core operator $H_{c,B}=1+L_{c,B}$, introduced
below, we use the notation $h_t^{c,B}$ and $p_s^{c,B}$.

The Gaussian estimates for $h_t$ and $h_t^D$ are standard; the Neumann
estimate follows from the inner-uniform geometry of exterior Lipschitz
domains~\cite{GyryaSaloffCoste}.  Semigroup analyticity and the
complex-time Gaussian argument~\cite{CoulhonSikora08} imply, for every
integer $m\ge0$,
\begin{align*}
 |\partial_t^m h_t(x,y)|
 &\le C_m t^{-n/2-m}
       \exp\!\left(-\frac{|x-y|^2}{c_mt}\right),
 &&x,y\in\R^n,\\
 |\partial_t^m h_t^B(x,y)|
 &\le C_m t^{-n/2-m}
       \exp\!\left(-\frac{|x-y|^2}{c_mt}\right),
 &&x,y\in\Omega.
\end{align*}
For the shifted core kernels one likewise has
\[
 |\partial_t^m h_t^{c,B}(x,y)|
 \le C_m t^{-n/2-m}e^{-c_mt}
       \exp\!\left(-\frac{|x-y|^2}{c_mt}\right),
 \qquad x,y\in\Omega_c.
\]
Differentiating the subordination formula yields
\begin{align*}
 |\partial_s^{2m}p_s(x,y)|
 &\le C_m\frac{s}{(s+|x-y|)^{n+2m+1}},&
 |\partial_s^{2m+1}p_s(x,y)|
 &\le C_m\frac{1}{(s+|x-y|)^{n+2m+1}},\\
 |\partial_s^{2m}p_s^B(x,y)|
 &\le C_m\frac{s}{(s+|x-y|)^{n+2m+1}},&
 |\partial_s^{2m+1}p_s^B(x,y)|
 &\le C_m\frac{1}{(s+|x-y|)^{n+2m+1}}.
\end{align*}
For the shifted core kernels, the same estimates hold for $0<s\le1$;
the spectral gap yields
\[
 |\partial_s^m p_s^{c,B}(x,y)|\le C_m e^{-c s},
  \qquad s\ge1,\quad x,y\in\Omega_c.
\]
We shall use the following four standard consequences.

\begin{enumerate}[label=\textup{(\roman*)},leftmargin=*]
\item For $B\in\{D,N\}$, the low-exponent Gaussian-semigroup theory
\cite{CoulhonDuong,Sikora} implies
\begin{equation}\label{eq:exterior-low-Riesz}
 \norm{\nabla L_B^{-1/2}u}{L^r(\Omega)}
 \leq C_r\norm u{L^r(\Omega)},\qquad 1<r\leq2,
\end{equation}
\item The whole-space reverse square-root estimate of
Auscher--Tchamitchian~\cite{AuscherTchamitchian98} states that
\begin{equation}\label{eq:whole-space-reverse}
 \norm{L^{1/2}v}{L^r(\R^n)}
 \leq C_r\norm{\nabla v}{L^r(\R^n)},\qquad 1<r<\infty,
\end{equation}
\item The strongly Lipschitz-domain theorem
\cite[Theorem~1]{AuscherTchamitchian01} asserts
\begin{equation}\label{eq:localized-square-root}
 \norm{L_B^{1/2}v}{L^r(\Omega)}
 \leq C_r\bigl(\norm v{L^r(\Omega)}
                    +\norm{\nabla v}{L^r(\Omega)}\bigr),
 \qquad 1<r<\infty.
\end{equation}
Its hypotheses hold here as follows: symmetry and the
Kato theorem supply the required $L^2$ square-root estimate; the real
coefficient heat kernels have the Gaussian and De Giorgi--Nash H\"older
bounds; and the symmetric heat semigroups are contractions on every
$L^r$.  We use \eqref{eq:localized-square-root} only for the compactly
supported pieces at the low exponents \(r<n\) occurring below.  For those
pieces the zeroth-order term is removed by
\eqref{eq:exterior-local-Poincare}.

On a bounded strongly Lipschitz domain $U$, the same theorem and the
bounded functional calculus yield
\begin{equation}\label{eq:core-reverse}
 \norm{(1+L_{U,B})^{1/2}v}{L^r(U)}
 \leq C_r\bigl(\norm v{L^r(U)}+\norm{\nabla v}{L^r(U)}\bigr).
\end{equation}
Here the pure boundary condition $B$ is imposed on all of $\partial U$.
\item For Dirichlet data, zero extension reduces the homogeneous Sobolev
inequality to its Euclidean version.  For Neumann data, extend across the
compact Lipschitz boundary after subtracting the average on one fixed
boundary neighborhood, and then add that constant back.  The local
Lipschitz extension theorem \cite[Chapter~VI, \S3, Theorem~5,
p.~181]{SteinSingular70} and Poincar\'e's inequality produce an extension operator \(E\) satisfying
\(\|\nabla Ef\|_{L^r(\mathbb R^n)}\leq C\|\nabla f\|_{L^r(\Omega)}\).
The Euclidean homogeneous Sobolev inequality therefore implies, for
$1<r<n$ and $r^*=nr/(n-r)$,
\begin{equation}\label{eq:exterior-Sobolev}
 \norm f{L^{r^*}(\Omega)}\leq C_r\norm{\nabla f}{L^r(\Omega)}
\end{equation}
for compactly supported Neumann test functions and for zero-trace
Dirichlet test functions.  In particular, for $1<r<n$ and every fixed
bounded set $K$,
\begin{equation}\label{eq:exterior-local-Poincare}
 \norm f{L^r(K\cap\Omega)}\leq C_{K,r}\norm{\nabla f}{L^r(\Omega)}.
\end{equation}
\end{enumerate}

\subsection{Bounded core and Devyver's modified Poisson parametrix}

Choose a bounded Lipschitz core
$\Omega_c=\Omega\cap B(0,R_c)$, with $R_c>0$ large enough, whose
artificial boundary is smooth.  The form domain of $L_{c,D}$ is
$W_0^{1,2}(\Omega_c)$, whereas that of $L_{c,N}$ is
$W^{1,2}(\Omega_c)$.  The core operator uses the same pure boundary
condition on the physical and artificial components.
Let $L_{c,B}$ be the associated operator and put
\begin{equation}\label{eq:shifted-core}
 H_{c,B}=1+L_{c,B},\qquad B\in\{D,N\}.
\end{equation}
For $B=N$ the unshifted pure core has the constant zero mode; imposing the
same Neumann condition on both boundary components therefore does not by
itself imply a spectral gap.  The shift in \eqref{eq:shifted-core} removes
that mode.  Subordination of
$e^{-tH_{c,B}}=e^{-t}e^{-tL_{c,B}}$ yields
\[
 \norm{e^{-s\sqrt{H_{c,B}}}}{L^r(\Omega_c)\to L^r(\Omega_c)}
 \leq e^{-s},\qquad 1\leq r\leq\infty.
\]

Choose a partition $\rho_c+\rho_e=1$ and cutoffs $\phi_c,\phi_e$ such that
$\phi_j\rho_j=\rho_j$ and
$\dist(\supp\nabla\phi_j,\supp\rho_j)>0$.  The core cutoffs vanish in a
collar of the artificial boundary, while $\phi_e$ vanishes near the physical
boundary and equals one at infinity.  Following Carron~\cite{Carron07},
we define the exterior parametrix by
\begin{equation}\label{eq:exterior-Carron}
 \mathcal E_{B,s}u=
 \phi_ce^{-s\sqrt{H_{c,B}}}(\rho_cu)
 +\phi_ee^{-s\sqrt L}(\rho_eu).
\end{equation}
Here \(\rho_cu\) is restricted to \(\Omega_c\), whereas \(\rho_eu\) is
extended by zero through the obstacle before the whole-space semigroup is
applied.  This zero extension is legitimate because \(\rho_e\) vanishes in
a collar of the physical boundary.  Conversely, \(\phi_c\) vanishes in a
collar of the artificial boundary, so the core solution extends by zero to
\(\Omega\).  The core, exterior and whole-space coefficient fields agree
on every cutoff transition region.

Let $\cA_1,\cA_2$ be connected bounded smooth annuli such that
$\overline{\cA_1}\subset\cA_2$,
$\supp\nabla\phi_e\subset\cA_1$ and
$\overline{\cA_2}\cap\supp\rho_e=\varnothing$.  Fix
$\eta\in C_c^\infty(\cA_2)$, $\eta\geq0$, $\int\eta=1$, and set
\[
 m_s(u)=\int_{\cA_2}\eta(z)e^{-s\sqrt L}(\rho_eu)(z)\,\dd z.
\]
Separation implies $m_0(u)=0$.  Adapting Devyver's
modification~\cite{Devyver15}, we consider
\begin{equation}\label{eq:exterior-Devyver}
 \mathcal S_{N,s}u=
 \mathcal E_{N,s}u-(\phi_e-1)m_s(u).
\end{equation}
Because $\phi_e-1$ is constant near $\partial\Omega$, this correction has
zero conormal derivative.

For a Lipschitz cutoff $\phi$, define the form commutator by
\begin{equation}\label{eq:form-commutator}
 \ip{[L_B,\phi]v}{\psi}
 =\int_\Omega vA\nabla\phi\mathbin{\cdo}\nabla\overline\psi
 -\int_\Omega\overline\psi A\nabla\phi\mathbin{\cdo}\nabla v.
\end{equation}
Our Hilbert pairing is linear in the first variable.  If a form distribution
$F$ occurs in the second variable, we use the anti-dual convention
$\ip{\psi}{F}=\overline{\ip{F}{\psi}}$.
Then, in the form sense,
\begin{align}
 (-\partial_s^2+L_N)\mathcal S_{N,s}u
 ={}&[L_N,\phi_c]e^{-s\sqrt{H_{c,N}}}(\rho_cu)
 -\phi_ce^{-s\sqrt{H_{c,N}}}(\rho_cu)\notag\\
 &+[L_N,\phi_e]\bigl(e^{-s\sqrt L}(\rho_eu)-m_s(u)\bigr)
 +(\phi_e-1)m_s''(u).
 \label{eq:Neumann-error}
\end{align}
To verify \eqref{eq:Neumann-error}, write the two model solutions as
$U_{c,s}=e^{-s\sqrt{H_{c,N}}}(\rho_cu)$ and
$U_{e,s}=e^{-s\sqrt L}(\rho_eu)$.  On the core transition region,
$L_N(\phi_cU_{c,s})=\phi_cL_{c,N}U_{c,s}+[L_N,\phi_c]U_{c,s}$, whereas
$\partial_s^2U_{c,s}=H_{c,N}U_{c,s}$.  Their difference is therefore
$[L_N,\phi_c]U_{c,s}-\phi_cU_{c,s}$.  Similarly, the uncorrected end produces
$[L_N,\phi_e]U_{e,s}$.  Finally, since a spatial constant is annihilated
locally by $L_N$ in the form sense on the cutoff region,
\[
 (-\partial_s^2+L_N)\bigl[-(\phi_e-1)m_s(u)\bigr]
 =-[L_N,\phi_e]m_s(u)+(\phi_e-1)m_s''(u).
\]
Adding these identities proves \eqref{eq:Neumann-error}.  In particular,
the zeroth-order core term is exactly the price of the shift
$H_{c,N}=1+L_{c,N}$.

\subsection{Caccioppoli and Poisson estimates}

We first record the following standard Caccioppoli estimate, used throughout
the paper.  It is valid both in the interior and up to either pure boundary
condition.

\begin{lemma}\label{lem:local-caccioppoli}
Let $U_1,U_2\subset\R^n$ be bounded open sets with
$\overline{U_1}\subset U_2$, and put
\[
 d=\dist(U_1,\R^n\setminus U_2)>0.
\]
Let $B\in\{D,N\}$.  Suppose that
$v\in W^{1,2}(U_2\cap\Omega)$, with zero trace on
$U_2\cap\partial\Omega$ when $B=D$, and that
\[
 \int_{U_2\cap\Omega}A\nabla v\mathbin{\cdo}\nabla\overline\psi
 =\int_{U_2\cap\Omega}f\,\overline\psi
  -\int_{U_2\cap\Omega}G\mathbin{\cdo}\nabla\overline\psi
\]
for every $B$-admissible test function $\psi$ supported in $U_2$.
Here $B$-admissible means $\psi\in W_0^{1,2}(\Omega)$ in the Dirichlet
case and $\psi\in W^{1,2}(\Omega)$ in the Neumann case.  Assume that
\[
 f\in L^2(U_2\cap\Omega),
 \qquad G\in L^2(U_2\cap\Omega;\mathbb C^n).
\]
Then
\begin{equation}\label{eq:local-caccioppoli}
 \norm{\nabla v}{L^2(U_1\cap\Omega)}
 \leq C\left[
 d^{-1}\norm v{L^2(U_2\cap\Omega)}
 +\norm f{L^2(U_2\cap\Omega)}^{1/2}
  \norm v{L^2(U_2\cap\Omega)}^{1/2}
 +\norm G{L^2(U_2\cap\Omega)}
 \right].
\end{equation}
The same estimate holds for the whole-space realization $L$.  It also
holds for $H_{c,B}=1+L_{c,B}$ when $v\in W_0^{1,2}(\Omega_c)$ for $B=D$
and $v\in W^{1,2}(\Omega_c)$ for $B=N$, and the weak
equation is read, for core-admissible tests supported in $U_2$, as
\[
 \int_{\Omega_c}A\nabla v\mathbin{\cdo}\nabla\overline\psi
 +\int_{\Omega_c}v\,\overline\psi
 =\int_{\Omega_c}f\,\overline\psi
 -\int_{\Omega_c}G\mathbin{\cdo}\nabla\overline\psi.
\]
\end{lemma}

\begin{proof}
The proof is standard and we include it here for the sake of completeness. Choose $\zeta\in C_c^\infty(U_2)$ such that
\[
 0\leq\zeta\leq1,
 \qquad \zeta=1\ \text{on }U_1,
 \qquad |\nabla\zeta|\leq C d^{-1}.
\]
The function $\psi=\zeta^2v$ is an admissible test function.  Indeed, in
the Dirichlet case it has zero trace because $v$ has zero trace; in the Neumann case no trace restriction is
imposed.

Substituting $\psi=\zeta^2v$ into the weak equation and taking real parts,
we obtain
\begin{align*}
 \textrm{Re}\int \zeta^2A\nabla v\mathbin{\cdo}\nabla\overline v
 ={}&-2\textrm{Re}\int
   \zeta\overline v\,A\nabla v\mathbin{\cdo}\nabla\zeta\\
 &+\textrm{Re}\int f\zeta^2\overline v
 -\textrm{Re}\int G\mathbin{\cdo}
 \bigl(\zeta^2\nabla\overline v
       +2\zeta\overline v\nabla\zeta\bigr),
\end{align*}
where all integrals are over $U_2\cap\Omega$.  Set
\[
 X=\|\zeta\nabla v\|_2,\quad
 V=\|v\|_{L^2(U_2\cap\Omega)},\quad
 F=\|f\|_{L^2(U_2\cap\Omega)},\quad
 H=\|G\|_{L^2(U_2\cap\Omega)}.
\]
Ellipticity, the bound on $\nabla\zeta$, and Cauchy--Schwarz imply
\[
 \lambda X^2\leq C d^{-1}XV+FV+HX+C d^{-1}HV.
\]
Young's inequality implies
\[
 C d^{-1}XV\leq\frac{\lambda}{4}X^2+C d^{-2}V^2,
 \qquad
 HX\leq\frac{\lambda}{4}X^2+CH^2,
\]
and
\[
 C d^{-1}HV\leq CH^2+C d^{-2}V^2.
\]
After absorbing the two $X^2$ terms into the left-hand side, we obtain
\[
 X^2\leq C\bigl(d^{-2}V^2+FV+H^2\bigr).
\]
Since $\zeta=1$ on $U_1$, taking square roots proves
\eqref{eq:local-caccioppoli}.

The whole-space proof is identical.  For $H_{c,B}=1+L_{c,B}$, testing
the weak equation produces the additional term
$\int_{\Omega_c}\zeta^2|v|^2$ on the left-hand side.  It is nonnegative
and may be discarded, so the same estimate follows.
\end{proof}

\begin{lemma}\label{lem:m-small}
For $r>1$ and $0<s\leq1$,
\begin{equation}\label{eq:m-small}
 \abs{m_s(u)}+\abs{m_s''(u)}
 \leq Cs\norm u{L^r(\Omega)}.
\end{equation}
Moreover,
\begin{equation}\label{eq:m-prime-zero}
 m_0'(u)=-\ip{\rho_eu}{L^{1/2}\eta}_{\R^n},
 \qquad \abs{m_0'(u)}\leq C_r\norm u{L^r(\Omega)}.
\end{equation}
\end{lemma}

\begin{proof}
Let
$d_0=\dist(\supp\eta,\supp\rho_e)>0$.  The heat-kernel Gaussian estimate
and the even-order Poisson-kernel estimates recorded at the beginning of
this section, with $m=0$ and $m=1$, yield the required bounds for $p_s$
and $\partial_s^2p_s$.
If
\[
 k_s(y)=\rho_e(y)\int_{\cA_2}\eta(x)p_s(x,y)\,\dd x,
\]
then, for $0<s\leq1$ and $y\in\supp\rho_e$,
\[
 \abs{k_s(y)}\leq
 Cs\int_{\cA_2}\frac{\eta(x)}{\abs{x-y}^{n+1}}\,\dd x.
\]
The right-hand side, and the analogous expression with power $n+3$, has
$L^{r'}(\dd y)$ norm at most $C_rs$.  Thus H\"older's inequality proves
\eqref{eq:m-small}.

For $u\in L^2\cap L^r$, self-adjointness yields
\[
 m_s(u)=\ip{\rho_eu}{e^{-s\sqrt L}\eta}_{\R^n}.
\]
Since \(\eta\in D(L^{1/2})\), the spectral theorem implies
\(L^{1/2}e^{-s\sqrt L}\eta\to L^{1/2}\eta\) in \(L^2\) as $s\to0$.  Hence
\[
 m_0'(u)
 =-\ip{\rho_eu}{L^{1/2}\eta}_{\R^n}.
\]
Moreover, \eqref{eq:whole-space-reverse} at exponent $r'$ gives $\norm{L^{1/2}\eta}{L^{r'}(\R^n)}\leq C_{r'}\norm{\nabla\eta}{L^{r'}(\R^n)}<\infty$, so the second estimate in \eqref{eq:m-prime-zero} follows by again H\"older's inequality.  Density removes the temporary $L^2$
assumption.
\end{proof}

\begin{proposition}\label{prop:local-Poisson-estimates}
The following estimates hold.
\begin{enumerate}[label=\textup{(\roman*)}]
\item If $r>2$, there is $\varepsilon\in(0,1)$ such that, for $s\geq1$,
\begin{align}
 \norm{e^{-s\sqrt L}(\rho_eu)-m_s(u)}{W^{1,2}(\cA_1)}
 &\leq Cs^{-n/r-\varepsilon}\norm u{L^r(\Omega)},
 \label{eq:centered-W12}\\
 \abs{m_s''(u)}
 &\leq Cs^{-2-n/r}\norm u{L^r(\Omega)}.
 \label{eq:m-second-large}
\end{align}
\item Let $B\in\{D,N\}$, let $K\subset\R^n$ be bounded and open, and let
$1<r<\infty$.  For $s\geq1$,
\begin{equation}\label{eq:domain-local}
 \norm{e^{-s\sqrt{L_B}}f}{W^{1,2}(K\cap\Omega)}
 \leq C_{K,r}s^{-n/r}\norm f{L^r(\Omega)}.
\end{equation}
If $K$ and $\supp f$ are separated, then for $0<s\leq1$,
\begin{equation}\label{eq:domain-small-off}
 \norm{e^{-s\sqrt{L_B}}f}{W^{1,2}(K\cap\Omega)}
 \leq C_{K,r}s\norm f{L^r(\Omega)}.
\end{equation}
\end{enumerate}
\end{proposition}

\begin{proof}
We first justify the use of Caccioppoli for arbitrary \(L^r\) data.
Approximate general \(L^r\) data in \(L^r\) by functions in $L^2\cap L^r$.
The displayed local \(L^2\)
kernel estimates, applied to differences, and
Lemma~\ref{lem:local-caccioppoli} show convergence in \(W^{1,2}\) on a
slightly smaller set.  The limits are the kernel-defined Poisson orbits,
so all estimates pass to general \(L^r\) data.  We use the same
approximation below when invoking the whole-space version of the
off-diagonal estimate.

Let $p_s(x,y)$ be the whole-space Poisson kernel.  De Giorgi--Nash
regularity and subordination yield, for some $\alpha\in(0,1)$ and every
$0<\varepsilon<\alpha$,
\[
 \abs{p_s(x,y)-p_s(z,y)}
 \leq C\abs{x-z}^{\varepsilon}
 \frac{s}{(s+\dist(y,\cA_2))^{n+1+\varepsilon}},
 \qquad x,z\in\cA_2.
\]
Averaging in $z$ against $\eta$ and applying H\"older's inequality yields
\[
 \norm{w_s}{L^2(\cA_2)}
 \leq Cs^{-n/r-\varepsilon}\norm u{L^r},
 \qquad
 w_s=e^{-s\sqrt L}(\rho_eu)-m_s(u).
\]
The even-order estimate with $m=1$ similarly yields
\[
 \norm{Lw_s}{L^2(\cA_2)}
 \leq Cs^{-2-n/r}\norm u{L^r}.
\]
Lemma~\ref{lem:local-caccioppoli}, applied with
$U_1=\cA_1$ and $U_2=\cA_2$, now proves \eqref{eq:centered-W12} (after decreasing
$\varepsilon$ if necessary).  Averaging the differentiated kernel against
$\eta$ proves \eqref{eq:m-second-large}.

For the second part, use the corresponding even-order bounds for
$p_s^B$ and $\partial_s^2p_s^B$ recorded above.  Choose bounded open sets
$K\subset K_1$ and $\overline{K_1}\subset K_2$.  On $K_2\cap\Omega$, H\"older's inequality
yields, for $s\geq1$,
\[
 \norm{e^{-s\sqrt{L_B}}f}{L^2(K_2\cap\Omega)}\leq Cs^{-n/r}\norm f{L^r},
 \qquad
 \norm{L_Bv_s}{L^2(K_2\cap\Omega)}\leq Cs^{-2-n/r}\norm f{L^r}.
\]
Lemma~\ref{lem:local-caccioppoli}, with $U_1=K_1$ and $U_2=K_2$, proves
\eqref{eq:domain-local}.  If $K$ and $\supp f$ are separated, the sets can
be chosen so that $K_2$ and $\supp f$ are separated.  The same two kernel bounds are
$O(s)$ in the relevant $L^{r'}$ norms for $0<s\leq1$.  A second
application of Lemma~\ref{lem:local-caccioppoli} proves
\eqref{eq:domain-small-off}.
\end{proof}

\subsection{Proof of Theorem~\ref{thm:intro-neumann-reverse}}

We first use a Wronskian identity that avoids applying $L_N^{-1/2}$ to a
form-valued error.

\begin{lemma}\label{lem:Wronskian}
For $j=1,2$, let $T_j$ be a nonnegative self-adjoint operator on a Hilbert
space $\mathcal H_j$.  Let $\theta:\mathcal H_2\to\mathcal H_1$ be bounded
and map $D(T_2^{1/2})$ continuously into $D(T_1^{1/2})$.  For
$a\in\mathcal H_1$ and $b\in\mathcal H_2$, set, for $s>0$,
\[
 x_s=e^{-sT_1^{1/2}}a,
 \qquad y_s=e^{-sT_2^{1/2}}b,
\]
and
\[
 D(s)=\ip{T_1^{1/2}x_s}{\theta y_s}
      -\ip{x_s}{\theta T_2^{1/2}y_s}.
\]
The expression on the right of the identity below is a form-dual pairing:
\[
 \ip{x}{(T_1\theta-\theta T_2)y}
 :=\mathfrak q_{T_1}(x,\theta y)-\ip{x}{\theta T_2y}.
\]
Then, in the form sense,
\begin{equation}\label{eq:Wronskian-derivative}
 D'(s)=-\ip{x_s}{(T_1\theta-\theta T_2)y_s}.
\end{equation}
\end{lemma}

\begin{proof}
% Fix $0<\delta<R<\infty$.  Positive-time Poisson smoothing implies, for
% every integer $k\ge0$,
% \[
%  \sup_{\delta\le s\le R}
%  \|T_j^{k/2}e^{-sT_j^{1/2}}\|_{2\to2}<\infty,
%  \qquad j=1,2.
% \]
% Consequently all terms created by differentiating \(D(s)\), including
% \(T_1x_s\) and \(T_2y_s\), belong to \(L^2\) on \([\delta,R]\).

% The computation below is therefore an ordinary Hilbert-space
% differentiation.  Since \(\delta>0\) is arbitrary, the resulting identity
% holds for every \(s>0\).

Differentiation yields
\begin{align*}
 D'(s)
 ={}&-\ip{T_1x_s}{\theta y_s}
     -\ip{T_1^{1/2}x_s}{\theta T_2^{1/2}y_s}\\
 &\quad+\ip{T_1^{1/2}x_s}{\theta T_2^{1/2}y_s}
     +\ip{x_s}{\theta T_2y_s}.
\end{align*}
The middle terms cancel.  Since $\theta y_s\in D(T_1^{1/2})$,
\[
 \ip{T_1x_s}{\theta y_s}
 =\mathfrak q_{T_1}(x_s,\theta y_s).
\]
The remaining expression is therefore
\[
 -\mathfrak q_{T_1}(x_s,\theta y_s)
 +\ip{x_s}{\theta T_2y_s},
\]
which is precisely the form interpretation of
\eqref{eq:Wronskian-derivative}. 

% Since $[\delta,R]$ was arbitrary, the
% identity holds on $(0,\infty)$.
\end{proof}

From now on, $F_{N,s}u$ denotes the error
$(-\partial_s^2+L_N)\mathcal S_{N,s}u$, namely the four terms displayed in
\eqref{eq:Neumann-error}.
When Lemma~\ref{lem:Wronskian} is applied to the core term,
\(\theta_cv\) means \(\phi_cv\) extended by zero from \(\Omega_c\) to
\(\Omega\).  For the end term, \(\theta_ev\) means the restriction of
\(\phi_ev\) from \(\mathbb R^n\) to \(\Omega\).  

% The cutoff support
% conditions make both maps bounded on \(L^2\) and continuous between the
% corresponding form domains.  We suppress the subscript in each application
% of the abstract lemma.

\begin{proposition}\label{prop:small-pairing}
Let $1<p<2$ and let $f,u\in C_c^\infty(\overline\Omega)$.  Then
\begin{equation}\label{eq:small-pairing}
 \sup_{0<\delta<1}
 \abs{\int_\delta^1
 \ip{e^{-s\sqrt{L_N}}f}{F_{N,s}u}\,\dd s}
 \leq C\norm{\nabla f}{L^p(\Omega)}\norm u{L^{p'}(\Omega)}.
\end{equation}
Moreover, the integral has a limit as $\delta\downarrow0$.
\end{proposition}

\begin{proof}
Let $K$ contain the supports of $\nabla\phi_c$, $\nabla\phi_e$,
$\phi_c$, and $\phi_e-1$.  Choose
$\chi\in C_c^\infty(\R^n)$ equal to one on a small neighborhood of $K$,
and put
\[
 f_{\mathrm{loc}}=\chi f,\qquad f_{\mathrm{far}}=(1-\chi)f.
\]
Because $1<p<2\leq n$, we have $p<n$.  The compact Poincar\'e estimate
\eqref{eq:exterior-local-Poincare}, applied on the support of $\chi$, yields
\[
 \norm{f_{\mathrm{loc}}}{W^{1,p}(\Omega)}
 \leq C\norm{\nabla f}{L^p(\Omega)}.
\]
Indeed,
\(\nabla(\chi f)=\chi\nabla f+f\nabla\chi\), and the compact Poincar\'e
estimate controls the second term as well as \(\|\chi f\|_p\).
The global Sobolev inequality \eqref{eq:exterior-Sobolev} also implies, with
$p^*=np/(n-p)$,
\begin{equation}\label{eq:ffar-pstar}
 \norm{f_{\mathrm{far}}}{L^{p^*}(\Omega)}
 \leq C\norm{\nabla f}{L^p(\Omega)}.
\end{equation}
Here the same product calculation yields
\(\|\nabla f_{\mathrm{far}}\|_p\leq C\|\nabla f\|_p\).

We first treat the local part.  To keep the endpoint calculation readable,
write
\[
 X_s=e^{-s\sqrt{L_N}}f_{\mathrm{loc}},\qquad
 U_{c,s}=e^{-s\sqrt{H_{c,N}}}(\rho_cu),\qquad
 U_{e,s}=e^{-s\sqrt L}(\rho_eu).
\]
For the core, set
\[
 D_c(s)=\ip{L_N^{1/2}X_s}{\phi_cU_{c,s}}
       -\ip{X_s}{\phi_cH_{c,N}^{1/2}U_{c,s}}.
\]
Lemma~\ref{lem:Wronskian}, with
$T_1=L_N$, $T_2=H_{c,N}$, and $\theta=\phi_c$, yields for
$0<\delta<1$
\begin{equation}\label{eq:core-Wronskian-truncated}
 \int_\delta^1\ip{X_s}
 {[L_N,\phi_c]U_{c,s}-\phi_cU_{c,s}}\,\dd s
 =D_c(\delta)-D_c(1).
\end{equation}
At zero, \eqref{eq:localized-square-root}, \eqref{eq:core-reverse}, and
H\"older's inequality yield
\begin{align*}
 \abs{\ip{L_N^{1/2}f_{\mathrm{loc}}}{\phi_c\rho_cu}}
 &\leq C\norm{f_{\mathrm{loc}}}{W^{1,p}(\Omega)}
          \norm u{L^{p'}(\Omega)},\\
 \abs{\ip{f_{\mathrm{loc}}}
 {\phi_cH_{c,N}^{1/2}(\rho_cu)}}
 &=\abs{\ip{H_{c,N}^{1/2}(\phi_cf_{\mathrm{loc}})}{\rho_cu}}\\
 &\leq C\norm{\phi_cf_{\mathrm{loc}}}{W^{1,p}(\Omega_c)}
          \norm u{L^{p'}(\Omega)}.
\end{align*}
To pass to the lower endpoint, commute $L_N^{1/2}$ with its Poisson
semigroup in the first pairing and move $H_{c,N}^{1/2}$ onto
$\phi_cX_\delta$ in the second.  The low-exponent Riesz bound,
\eqref{eq:localized-square-root}, and strong semigroup continuity imply
$X_\delta\to f_{\mathrm{loc}}$ in $W^{1,p}$; then
\eqref{eq:core-reverse} and the strong $L^{p'}$-continuity of the core
semigroup imply $D_c(\delta)\to D_c(0)$, and
$|D_c(0)|\leq C\|\nabla f\|_p\|u\|_{p'}$.  At $s=1$, analyticity of the
two Poisson semigroups implies
\begin{align*}
 \norm{L_N^{1/2}e^{-\sqrt{L_N}}f_{\mathrm{loc}}}{L^p}
 &\leq C\norm{f_{\mathrm{loc}}}{L^p},&
 \norm{e^{-\sqrt{H_{c,N}}}(\rho_cu)}{L^{p'}}
 &\leq C\norm u{L^{p'}},\\
 \norm{e^{-\sqrt{L_N}}f_{\mathrm{loc}}}{L^p}
 &\leq\norm{f_{\mathrm{loc}}}{L^p},&
 \norm{H_{c,N}^{1/2}e^{-\sqrt{H_{c,N}}}(\rho_cu)}{L^{p'}}
 &\leq C\norm u{L^{p'}}.
\end{align*}
Consequently $|D_c(1)|\leq
C\|\nabla f\|_p\|u\|_{p'}$.

For the end, define
\[
 D_e(s)=\ip{L_N^{1/2}X_s}{\phi_eU_{e,s}}
       -\ip{X_s}{\phi_eL^{1/2}U_{e,s}}.
\]
Lemma~\ref{lem:Wronskian}, now with $T_2=L$, yields
\begin{equation}\label{eq:end-Wronskian-truncated}
 \int_\delta^1\ip{X_s}{[L_N,\phi_e]U_{e,s}}\,\dd s
 =D_e(\delta)-D_e(1).
\end{equation}
The first term of $D_e(0)$ is bounded by
\eqref{eq:localized-square-root}.  For the second term, $\phi_ef_{\mathrm{loc}}$
vanishes near $\partial\Omega$ and therefore extends by zero to a function
in $W^{1,p}(\R^n)$.  Self-adjointness and
\eqref{eq:whole-space-reverse} imply
\[
 \abs{\ip{f_{\mathrm{loc}}}{\phi_eL^{1/2}(\rho_eu)}}
 =\abs{\ip{L^{1/2}(\phi_ef_{\mathrm{loc}})}{\rho_eu}}
 \leq C\norm{\phi_ef_{\mathrm{loc}}}{W^{1,p}(\Omega)}
          \norm u{L^{p'}}.
\]
The convergence $D_e(\delta)\to D_e(0)$ follows by the same argument as
before.  The two terms of $D_e(1)$ are bounded
exactly as at the core endpoint, using
analyticity on $L^p(\Omega)$ and $L^{p'}(\R^n)$.  Hence
\[
 \abs{D_e(0)}+\abs{D_e(1)}
 \leq C\norm{\nabla f}{L^p}\norm u{L^{p'}}.
\]
Equations \eqref{eq:core-Wronskian-truncated} and
\eqref{eq:end-Wronskian-truncated} control the uncorrected local error and
show that its integral has a limit at zero.

We next estimate the two local correction terms absolutely.  Since $m_s$
is spatially constant, \eqref{eq:form-commutator} implies
\[
 \abs{\ip{[L_N,\phi_e]m_s(u)}{X_s}}
 \leq C\abs{m_s(u)}\norm{\nabla X_s}{L^p(\Omega)}.
\]
Factor $\nabla e^{-s\sqrt{L_N}}=(\nabla L_N^{-1/2})L_N^{1/2}e^{-s\sqrt{L_N}}$. The Riesz estimate \eqref{eq:exterior-low-Riesz} and
analyticity imply
\[
 \norm{\nabla X_s}{L^p}
 \leq Cs^{-1}\norm{f_{\mathrm{loc}}}{L^p},\qquad 0<s\leq1.
\]
Together with \eqref{eq:m-small}, this produces an $s$-independent integrable
bound.  Similarly, by contractivity and \eqref{eq:m-small},
\[
 \abs{\ip{(\phi_e-1)m_s''(u)}{X_s}}
 \leq Cs\norm u{L^{p'}}\norm{f_{\mathrm{loc}}}{L^p}.
\]
It remains to treat $f_{\rm far}$.  Choose bounded open sets $K_1,K_2$ with
\[
 K\subset K_1,\qquad \overline{K_1}\subset K_2,\qquad
 \dist(K_2,\supp f_{\rm far})>0.
\]
Applying the off-diagonal estimate \eqref{eq:domain-small-off} with exponent
$p^*$ and using \eqref{eq:ffar-pstar}, we obtain
\begin{equation}\label{eq:far-Poisson-small}
 \norm{e^{-s\sqrt{L_N}}f_{\rm far}}{W^{1,2}(K_1\cap\Omega)}
 \leq Cs\norm{\nabla f}{L^p},\qquad0<s\leq1.
\end{equation}
For the core model, the $L^2$ Kato estimate and spectral calculus yield
\begin{align*}
 \norm{U_{c,s}}{L^2(\Omega_c)}
 &\leq C\norm u{L^{p'}},\\
 \norm{\nabla U_{c,s}}{L^2(\Omega_c)}
 &\leq C\norm{H_{c,N}^{1/2}e^{-s\sqrt{H_{c,N}}}(\rho_cu)}{L^2}
 \leq Cs^{-1}\norm u{L^{p'}}.
\end{align*}
Here $L^{p'}(\Omega_c)\subset L^2(\Omega_c)$ because $p'>2$ and the core has finite measure.  On the end transition region, the supports of
$\rho_eu$ and $\nabla\phi_e$ are separated.  The whole-space version of
the calculation proving \eqref{eq:domain-small-off} shows
\begin{equation}\label{eq:end-small-W12}
 \norm{U_{e,s}}{W^{1,2}(\cA_1)}
 +\abs{m_s(u)}+\abs{m_s''(u)}
 \leq Cs\norm u{L^{p'}}.
\end{equation}
Let $X_s^{\mathrm{far}}=e^{-s\sqrt{L_N}}f_{\rm far}$.  The form commutator
\eqref{eq:form-commutator} and
\eqref{eq:far-Poisson-small}--\eqref{eq:end-small-W12} now conclude
\begin{align*}
 \abs{\ip{[L_N,\phi_c]U_{c,s}}{X_s^{\mathrm{far}}}}
 &\leq C\norm{U_{c,s}}{W^{1,2}}
          \norm{X_s^{\mathrm{far}}}{W^{1,2}(K_1\cap\Omega)}
 \leq C\norm u{p'}\norm{\nabla f}p,\\
 \abs{\ip{\phi_cU_{c,s}}{X_s^{\mathrm{far}}}}
 &\leq C\norm{U_{c,s}}2\norm{X_s^{\mathrm{far}}}{L^2(K_1\cap\Omega)}
 \leq Cs\norm u{p'}\norm{\nabla f}p,\\
 \abs{\ip{[L_N,\phi_e](U_{e,s}-m_s)}{X_s^{\mathrm{far}}}}
 &\leq Cs^2\norm u{p'}\norm{\nabla f}p,\\
 \abs{\ip{(\phi_e-1)m_s''}{X_s^{\mathrm{far}}}}
 &\leq Cs^2\norm u{p'}\norm{\nabla f}p.
\end{align*}
Every far-part term is absolutely $s$-integrable on $(0,1)$.  Together with
the two Wronskian identities and the local correction estimates, this proves
\eqref{eq:small-pairing} and the existence of the improper integral.
\end{proof}

\begin{proposition}\label{prop:large-pairing}
Let $1<p<2$, and let
$f,u\in C_c^\infty(\overline\Omega)$.  Then
\begin{equation}\label{eq:large-pairing}
 \int_1^\infty\abs{\ip{e^{-s\sqrt{L_N}}f}{F_{N,s}u}}\,\dd s
 \leq C\norm{\nabla f}{L^p(\Omega)}\norm u{L^{p'}(\Omega)}.
\end{equation}
\end{proposition}

\begin{proof}
Let $K$ be a fixed bounded set containing the supports of all four terms in
\eqref{eq:Neumann-error}.  The shifted core has a spectral gap.  More
precisely, the $L^2$ Kato estimate and spectral calculus yield, for $s\geq1$,
\[
 \norm{e^{-s\sqrt{H_{c,N}}}(\rho_cu)}{W^{1,2}(\Omega_c)}
 \leq Ce^{-cs}\norm{\rho_cu}{L^2(\Omega_c)}
 \leq Ce^{-cs}\norm u{L^{p'}(\Omega)}.
\]
For a cutoff $\phi$ supported in a fixed transition region, the commutator
formula \eqref{eq:form-commutator} implies
\[
 \abs{\ip{[L_N,\phi]v}{\psi}}
 \leq C\norm v{W^{1,2}(\supp\nabla\phi)}
          \norm\psi{W^{1,2}(\supp\nabla\phi)}.
\]
Apply this inequality to the core and end terms.  Proposition
\ref{prop:local-Poisson-estimates}, with $r=p'>2$, and \eqref{eq:m-second-large} then give, for every \(\psi\in D(\mathfrak q_N)\),
\begin{equation}\label{eq:large-error-dual}
 \abs{\ip{F_{N,s}u}{\psi}}
 \leq C\left(e^{-cs}+s^{-n/p'-\varepsilon}
                  +s^{-2-n/p'}\right)
 \norm u{L^{p'}(\Omega)}
 \norm\psi{W^{1,2}(K\cap\Omega)}.
\end{equation}
Because \(p<n\), the exterior Sobolev inequality implies
$\|f\|_{L^{p^*}}\leq C\|\nabla f\|_{L^p}$.  Employ
\eqref{eq:domain-local} with exponent $p^*$:
\begin{equation}\label{eq:large-exterior-solution}
 \norm{e^{-s\sqrt{L_N}}f}{W^{1,2}(K\cap\Omega)}
 \leq Cs^{-n/p^*}\norm{\nabla f}{L^p(\Omega)},\qquad s\geq1.
\end{equation}
Multiplying \eqref{eq:large-error-dual} and
\eqref{eq:large-exterior-solution}, the end term has power
\[
 \frac n{p^*}+\frac n{p'}+\varepsilon
 =n-1+\varepsilon>1,
\]
where we used $1/p^*=1/p-1/n$.  The $m_s''$ term has power
$n/p^*+2+n/p'=n+1$, and the core term is exponentially integrable.
Integrating on $[1,\infty)$ proves \eqref{eq:large-pairing}.
\end{proof}

We are now in a position to prove Theorem~\ref{thm:intro-neumann-reverse}.

\begin{proof}[Proof of Theorem~\ref{thm:intro-neumann-reverse}]

By duality, it suffices to assume \(1<p<2\).  We take $f$ and $u$ in
the smooth form-core classes fixed in Remark~\ref{rem:Sobolev-form-core}. Differentiating \eqref{eq:exterior-Devyver} at $s=0$ and using
\eqref{eq:m-prime-zero}, we obtain
\[
 -\partial_s\mathcal S_{N,s}u|_{s=0}
 =\phi_cH_{c,N}^{1/2}(\rho_cu)
 +\phi_eL^{1/2}(\rho_eu)
 -(\phi_e-1)\ip{\rho_eu}{L^{1/2}\eta}_{\R^n}.
\]
For the core term, self-adjointness and \eqref{eq:core-reverse} yield
\begin{align*}
 \abs{\ip f{\phi_cH_{c,N}^{1/2}(\rho_cu)}}
 &=\abs{\ip{H_{c,N}^{1/2}(\phi_cf)}{\rho_cu}}\\
 &\leq C\|\phi_cf\|_{W^{1,p}(\Omega_c)}\|u\|_{L^{p'}(\Omega)}\\
 &\leq C\|\nabla f\|_{L^p(\Omega)}\|u\|_{L^{p'}(\Omega)}.
\end{align*}
The last inequality follows from \eqref{eq:exterior-local-Poincare} on
the fixed support of $\phi_c$. By \eqref{eq:whole-space-reverse} 
\[
 \abs{\ip f{\phi_eL^{1/2}(\rho_eu)}}
 =\abs{\ip{L^{1/2}(\phi_ef)}{\rho_eu}}
 \leq C\|\nabla f\|_{L^p(\Omega)}\|u\|_{L^{p'}(\Omega)}.
\]
The same estimate applied to the fixed function $\eta$:
\[
 \abs{\ip{\rho_eu}{L^{1/2}\eta}}
 \leq C\|\nabla\eta\|_{L^p(\R^n)}\|u\|_{L^{p'}(\Omega)}.
\]
Moreover,
\[
 \abs{\ip f{\phi_e-1}}
 \leq C\|f\|_{L^p(\supp(\phi_e-1))}
 \leq C\|\nabla f\|_{L^p(\Omega)}.
\]
Consequently,
\[
 \abs{\ip f{-\partial_s\mathcal S_{N,s}u|_{s=0}}}
 \leq C\|\nabla f\|_{L^p(\Omega)}\|u\|_{L^{p'}(\Omega)}.
\]
To proceed, set $X_s=e^{-s\sqrt{L_N}}f$, $Y_s=\mathcal S_{N,s}u$, and $W(s)=\ip{X_s}{\partial_sY_s}-\ip{\partial_sX_s}{Y_s}$. Now, since
\[
 \partial_s^2X_s=L_NX_s,
 \qquad \partial_s^2Y_s=L_NY_s-F_{N,s}u
\]
in the form sense, the standard form product rule on every finite interval
$[\delta,R]\subset(0,\infty)$ gives
\begin{align*}
 W'(s)
 &=\ip{X_s}{\partial_s^2Y_s}
   -\ip{\partial_s^2X_s}{Y_s}\\
 &=\mathfrak q_N(X_s,Y_s)-\ip{X_s}{F_{N,s}u}
   -\mathfrak q_N(X_s,Y_s)
 =-\ip{X_s}{F_{N,s}u}.
\end{align*}
Hence
\[
 W(R)-W(\delta)
 =-\int_\delta^R
   \ip{e^{-s\sqrt{L_N}}f}{F_{N,s}u}\,\dd s.
\]
We next let \(R\to\infty\).  By
\eqref{eq:exterior-L2-kernels}, the spectral theorem implies $X_R\longrightarrow0$, $\partial_RX_R\longrightarrow0$ in $L^2(\Omega)$. The shifted core terms in $Y_R$ and $\partial_RY_R$ decay exponentially. The whole-space end terms and their derivatives likewise converge to zero in \(L^2\).  Finally,
\[
 m_R(u)=\ip{\rho_eu}{e^{-R\sqrt L}\eta},
 \qquad
 m_R'(u)=-\ip{\rho_eu}{L^{1/2}e^{-R\sqrt L}\eta}
\]
converge to zero by the spectral theorem.  Thus $W(R)\to0$.

At the lower endpoint, form-norm continuity yields
\[
 X_0=f,
 \qquad \partial_sX_s|_{s=0}=-L_N^{1/2}f,
 \qquad Y_0=u.
\]
Moreover, the spectral theorem and the form-domain assumptions on the
model inputs imply
$\partial_sY_s\to\partial_s\mathcal S_{N,s}u|_{s=0}$ in $L^2$; for the
scalar correction this is the convergence $m_s'(u)\to m_0'(u)$.
Consequently $W(\delta)\to W(0)$ and
\[
 W(0)=\ip f{\partial_s\mathcal S_{N,s}u|_{s=0}}
       +\ip{L_N^{1/2}f}{u}.
\]
Proposition~\ref{prop:small-pairing} guarantees convergence of the integral at
zero, while Proposition~\ref{prop:large-pairing} shows absolute
convergence at infinity.  Letting $R\uparrow\infty$ and
$\delta\downarrow0$, we get
\[
 \ip{L_N^{1/2}f}{u}
 =\ip f{-\partial_s\mathcal S_{N,s}u|_{s=0}}
 +\int_0^\infty
  \ip{e^{-s\sqrt{L_N}}f}{F_{N,s}u}\,\dd s.
\]
Combining the model estimate above with
Propositions~\ref{prop:small-pairing} and \ref{prop:large-pairing}, we conclude
\[
 \abs{\ip{L_N^{1/2}f}{u}}
 \leq C\norm{\nabla f}{L^p(\Omega)}\norm u{L^{p'}(\Omega)}.
\]
Taking the supremum over $u$ with $\|u\|_{p'}\leq1$ completes the proof.
\end{proof}

\begin{remark}
The purpose of Devyver's modified Poisson parametrix is not merely to improve
a constant.  On the
fixed transition annulus, the unmodified whole-space Poisson extension has
size $s^{-n/p'}\|u\|_{p'}$.  The exterior Poisson solution paired with it
has size $s^{-n/p^*}\|\nabla f\|_p$, and
\[
 \frac n{p^*}+\frac n{p'}=n-1.
\]
Thus Carron's uncentered large-time integral is governed by $\int_1^\infty s^{-(n-1)}\,\dd s$. The resulting absolute majorant is integrable for $n\geq3$ but is only
logarithmically nonintegrable when $n=2$.
The value $m_s(u)$ is the local constant mode of the end solution.  After it
is subtracted, the H\"older continuity in the spatial variable of the
Poisson kernel supplies the extra factor $s^{-\varepsilon}$ in
\eqref{eq:centered-W12}.  The critical planar integral becomes
$\int_1^\infty s^{-1-\varepsilon}\,\dd s$.  Because
$\phi_e-1$ is constant near the physical boundary, this centering preserves
the Neumann condition.  The same correction would create a nonzero boundary
trace in the Dirichlet problem, which is why the cancellation is genuinely
boundary-condition dependent.
\end{remark}

\begin{remark}\label{rem:Dirichlet-reverse}
The Poisson-parametrix and Wronskian strategy also has a Dirichlet
analogue.  In dimensions $n\geq3$ the unmodified Dirichlet error is already
summable.  In the planar case the Neumann correction
\eqref{eq:exterior-Devyver} is unavailable because it destroys the zero
trace, so a separate trace-compatible low-energy estimate is required.
We do not duplicate that parallel analysis.  Jiang and Lin already proved
\[
 \norm{L_D^{1/2}f}{L^p(\Omega)}
 \leq C_p\norm{\nabla f}{L^p(\Omega)},
 \qquad 1<p<\infty,
\]
by Littlewood--Paley theory from \cite{KillipVisanZhang16} and a comparison of the whole-space and
Dirichlet heat kernels; see
\cite[Theorem~1.3 and Section~2]{JiangLin24}.  We therefore present the
detailed proof only for the Neumann realization, where the constant mode
and its boundary-compatible centering are the distinctive features.
\end{remark}

\part{Stability of the Riesz transform}\label{part2}

\section{Exterior Poisson parametrix}
\label{sec:forward-parametrix}

We first use the models $L$ and $L_{c,B}$ from Part~\ref{part1}, with
the exponents $P_B$ defined below.  The final paragraph of the proof in
Section~7 verifies the same construction for the overlapping models in
\eqref{eq:overlapping-models}.

The reverse argument in Part~\ref{part1} only pairs the form-valued gluing error with a Poisson solution.  The forward argument must take the spatial gradient of the Green correction.  When $A$ is merely bounded and measurable, the exact
commutator is
\[
 [L_B,\phi]v=-\diver(vA\nabla\phi)-A\nabla v\mathbin{\cdo}\nabla\phi
 \in W^{-1,p}+L^p,
\]
not a scalar function involving derivatives of $A$.  This section constructs the explicit part of the forward parametrix and records its form-valued error.  The next section integrates that error in the homogeneous
dual and solves it by the zero-energy inverse.

For the whole-space end and the pure bounded core introduced in
Section~\ref{sec:exterior-reverse}, recall
\begin{align*}
 p_L&=\sup\bigl\{r>2:\ \nabla L^{-1/2}\text{ is bounded on every }
                  L^q(\R^n),\ 2\leq q<r\bigr\},\\
 p_{c,B}&=\sup\bigl\{r>2:\ \nabla H_{c,B}^{-1/2}\text{ is bounded on every }
                  L^q(\Omega_c),\ 2\leq q<r\bigr\},\\
 P_B&=\min\{p_L,p_{c,B}\},\qquad B\in\{D,N\}.
\end{align*}
The core exponent cannot be omitted: a Lipschitz physical boundary does not have
an automatic full high-$p$ Riesz range.  The following spaces are needed
only to state precisely where the form-valued commutators live:
\[
 W_D^{1,r}(\Omega)=W_0^{1,r}(\Omega),\qquad
 W_N^{1,r}(\Omega)=W^{1,r}(\Omega),
 \qquad W_B^{-1,r}=(W_B^{1,r'})^*.
\]
If $f\in L^r(\Omega)$ and $H\in L^r(\Omega;\R^n)$, then it is clear that
\begin{equation}\label{eq:compact-form-norm}
 \norm{f+\diver H}{W_B^{-1,r}}
 \leq \norm f{L^r}+\norm H{L^r}.
\end{equation}

For Dirichlet data use the unmodified parametrix
\(\mathcal E_{D,s}\) from \eqref{eq:exterior-Carron}; for Neumann data use the centered parametrix \(\mathcal S_{N,s}\) from
\eqref{eq:exterior-Devyver}.  Set
\begin{align*}
 F_{D,s}u={}&[L_D,\phi_c]e^{-s\sqrt{H_{c,D}}}(\rho_cu)
 -\phi_ce^{-s\sqrt{H_{c,D}}}(\rho_cu)
 +[L_D,\phi_e]e^{-s\sqrt L}(\rho_eu),\\
 F_{N,s}u={}&(-\partial_s^2+L_N)\mathcal S_{N,s}u,
\end{align*}
where the second expression is expanded in \eqref{eq:Neumann-error}.  Thus
\begin{equation}\label{eq:parametrix-error-casewise}
 F_{B,s}u=(-\partial_s^2+L_B)U_s,\qquad
 U_s=
 \begin{cases}
  \mathcal E_{D,s}u,&B=D,\\
  \mathcal S_{N,s}u,&B=N.
 \end{cases}
\end{equation}
To be precise,
\[
 U_s=
 \begin{cases}
  \phi_ce^{-s\sqrt{H_{c,D}}}(\rho_cu)
  +\phi_ee^{-s\sqrt L}(\rho_eu),&B=D,\\[1mm]
  \phi_ce^{-s\sqrt{H_{c,N}}}(\rho_cu)
  +\phi_ee^{-s\sqrt L}(\rho_eu)
  -(\phi_e-1)m_s(u),&B=N.
 \end{cases}
\]
Every commutator is understood in the form sense
\eqref{eq:form-commutator}.  The error is supported in one fixed bounded
set and belongs to \(W_B^{-1,p}\) whenever the model Poisson
pieces belong locally to \(W^{1,p}\).  The centering correction is used only for
Neumann data: for Dirichlet data the term
\(-(\phi_e-1)m_s(u)\) would have the nonzero boundary trace \(m_s(u)\).

We next define the explicit part of the decomposition.  The Dirichlet
potential \(L^{-1/2}(\rho_eu)\) is locally in \(L^p\) when \(p<n\), by
fractional integration.  For Neumann data, put
\[
 W_R=\int_0^R e^{-s\sqrt L}(\rho_eu)\,\dd s,\qquad
 c_R=\int_{\mathcal A_2}\eta W_R .
\]
Fubini and the definition of \(m_s\) imply
\(c_R=\int_0^Rm_s(u)\,\dd s\), while
\[
 \nabla W_R
 =\nabla L^{-1/2}(I-e^{-R\sqrt L})(\rho_eu).
\]
If \(2<p<p_L\), the whole-space Riesz bound and strong stability of the
Poisson semigroup on \(L^p(\mathbb R^n)\) show that
\begin{equation}\label{eq:WR-gradient-limit}
 \nabla W_R\longrightarrow
 \nabla L^{-1/2}(\rho_eu)\quad\text{in }L^p(\mathbb R^n),\qquad
 \sup_R\|\nabla W_R\|_p\lesssim\|u\|_p .
\end{equation}
We shall use the following standard endpoint limits without further
comment.  If \(T=L\), with \(p<p_L\), or \(T=H_{c,B}\), then
\begin{equation}\label{eq:model-Poisson-endpoints}
 \begin{aligned}
 e^{-s\sqrt T}f&\longrightarrow f,&
 s\sqrt T\,e^{-s\sqrt T}f&\longrightarrow0&& (s\downarrow0),\\
 e^{-s\sqrt T}f&\longrightarrow0,&
 s\sqrt T\,e^{-s\sqrt T}f&\longrightarrow0&& (s\uparrow\infty)
 \end{aligned}
 \qquad\text{in }L^p.
\end{equation}
For \(L\), strong stability follows by interpolating the \(L^2\) spectral
limit with \(L^q\)-contractivity for some \(p<q<p_L\).  Analyticity then
gives the differentiated limits.  For \(H_{c,B}\), the shift gives
exponential decay.

% Here \(\ker_{L^2}L=\{0\}\): a form-kernel element has zero gradient and
% is therefore a constant, and no nonzero constant belongs to
% \(L^2(\mathbb R^n)\).

Here is the Poincar\'e argument used to normalize \(W_R\).  For
\(v\in W^{1,p}(\mathcal A_2)\), write
\(v_{\mathcal A_2}=|\mathcal A_2|^{-1}\int_{\mathcal A_2}v\) and
\(m_\eta(v)=\int_{\mathcal A_2}\eta v\).  Since \(\int\eta=1\),
\[
 |m_\eta(v)-v_{\mathcal A_2}|
 \le \|\eta\|_{p'}\|v-v_{\mathcal A_2}\|_p .
\]
The usual Poincar\'e inequality on the connected bounded annulus therefore
yields
\begin{equation}\label{eq:weighted-Poincare}
 \|v-m_\eta(v)\|_{L^p(\mathcal A_2)}
 \le C\|\nabla v\|_{L^p(\mathcal A_2)} .
\end{equation}
Apply this to \(v=W_{R_2}-W_{R_1}\).  Because
\(m_\eta(v)=c_{R_2}-c_{R_1}\), one obtains
\begin{equation}\label{eq:WR-Cauchy}
 \|(W_{R_2}-c_{R_2})-(W_{R_1}-c_{R_1})\|_{W^{1,p}(\mathcal A_2)}
 \le C\|\nabla(W_{R_2}-W_{R_1})\|_{L^p(\mathcal A_2)} .
\end{equation}
Thus \eqref{eq:WR-gradient-limit} determines a homogeneous whole-space
class with gradient \(\nabla L^{-1/2}(\rho_eu)\).  Let \(\widetilde W_u\)
be its representative normalized by
\(\int_{\mathcal A_2}\eta\widetilde W_u=0\).  If \(K\) is any fixed
bounded connected Lipschitz domain containing \(\mathcal A_2\), the same
weighted Poincar\'e argument yields
\[
 \|v-m_\eta(v)\|_{L^p(K)}\leq C_K\|\nabla v\|_{L^p(K)}.
\]
Consequently \(W_R-c_R\to\widetilde W_u\) in \(W^{1,p}(K)\), and
\begin{equation}\label{eq:normalized-potential-bound}
 \int_{\mathcal A_2}\eta\widetilde W_u=0,\qquad
 \nabla\widetilde W_u=\nabla L^{-1/2}(\rho_eu)\ \text{on }\mathbb R^n,
 \qquad
 \|\widetilde W_u\|_{W^{1,p}(\mathcal A_2)}
 \le C\|u\|_p,
\end{equation}
where the last inequality follows by applying \eqref{eq:weighted-Poincare} to the normalized limit.

We define the following operators:
\begin{align*}
 \mathcal R_Du={}&
 \phi_e\nabla L^{-1/2}(\rho_eu)
 +(\nabla\phi_e)L^{-1/2}(\rho_eu)+\phi_c\nabla H_{c,D}^{-1/2}(\rho_cu)
 +(\nabla\phi_c)H_{c,D}^{-1/2}(\rho_cu),\\
 \mathcal R_Nu={}&
 \phi_e\nabla L^{-1/2}(\rho_eu)
 +(\nabla\phi_e)\widetilde W_u +\phi_c\nabla H_{c,N}^{-1/2}(\rho_cu)
 +(\nabla\phi_c)H_{c,N}^{-1/2}(\rho_cu).
\end{align*}
Only the restriction of \(\widetilde W_u\) to
\(\supp\nabla\phi_e\subset\mathcal A_2\) enters \(\mathcal R_N\).
We use \(\mathcal R_B\) to mean \(\mathcal R_D\) when \(B=D\) and
\(\mathcal R_N\) when \(B=N\).

\begin{proposition}\label{prop:explicit-parametrix}
One has
\begin{align*}
 \|\mathcal R_Nu\|_{L^p(\Omega)}
 &\le C_p\|u\|_{L^p(\Omega)},&&2<p<P_N,\\
 \|\mathcal R_Du\|_{L^p(\Omega)}
 &\le C_p\|u\|_{L^p(\Omega)},&&2<p<\min\{n,P_D\}.
\end{align*}
\end{proposition}

\begin{proof}
The end-gradient term (i.e., $\nabla L^{-1/2}$) uses precisely \(p<p_L\), and the core-gradient term (i.e., $\nabla H^{-1/2}_{c,B}$)
uses \(p<p_{c,B}\).  The normalized Neumann potential is controlled by
\eqref{eq:normalized-potential-bound}.  Since \(H_{c,B}\ge1\),
subordination yields
\(\|H_{c,B}^{-1/2}f\|_p\le\|f\|_p\).  Finally, Gaussian bounds imply
\[
 L^{-1/2}:L^p(\mathbb R^n)\longrightarrow L^{p^*}(\mathbb R^n),
 \qquad p^*=\frac{np}{n-p},\quad1<p<n.
\]
On the fixed support of \(\nabla\phi_e\), H\"older's inequality controls
the Dirichlet potential term.  This proves the two estimates.
\end{proof}

% \begin{remark}\label{rem:forward-exponent-bookkeeping}
% There is no exponent restriction in the weighted Poincar\'e argument.
% For \(B=N\), the inequalities \(p<p_L\) and \(p<p_{c,N}\) are used,
% respectively, for the end and core gradient terms above.  The same two
% restrictions, through the single condition \(p<P_N\), also permit the
% choice \(p<r<P_N\) in Lemma~\ref{lem:automatic-ST} and in the large-time
% estimate of Lemma~\ref{lem:large-time-homogeneous-error} below.  In the Dirichlet case
% the additional restriction \(p<n\) controls the uncentered capacitary
% potential.

The reason for studying the zero-energy inverse next can already be seen
formally.  If \(U_s\) is the Poisson parametrix and
\((-\partial_s^2+L_B)U_s=F_{B,s}u\), integration of the half-line Green
kernel in its first variable suggests that, for each \(s>0\), we should
first form the balanced error
\[
 G_{B,s}u\sim(I-e^{-s\sqrt{L_B}})F_{B,s}u
\]
and then set \(g_{B,s}u=L_B^{-1}G_{B,s}u\).  Here \(g_{B,s}u\) is the
contribution of the source at height \(s\) to the integrated correction;
it is not the correction at cylinder height \(s\).  We must therefore
know that \(L_B\) is invertible between the relevant homogeneous Sobolev
spaces before applying it to each balanced error.  This is the purpose
of Section~\ref{sec:homogeneous-invertibility}.  Section~\ref{sec:forward-caccioppoli}
defines these errors by duality and proves the estimates needed to
integrate the already inverted contributions.

% \end{remark}

\section{Invertibility of the exterior operator}
\label{sec:homogeneous-invertibility}

This section constructs the zero-energy inverse needed in the proof of
Theorem~\ref{thm:intro-forward}.  We first treat the whole-space end and
the bounded core, and then glue the model inverses.  The Dirichlet and
Neumann arguments are stated separately because constants must be handled
differently in the Neumann case.

% The point of the construction is to obtain an exterior inverse from
% model Riesz bounds alone.  The commutators are compact between the
% homogeneous spaces, although their scalar parts need not be compact in
% $L^r$.  For Neumann data, balancing the localized sources makes the
% scalar errors have zero mean; this also places them in the energy dual
% in dimension two.  The Fredholm and interpolation theorems used below
% are classical.  Their role here is to turn these explicit form
% remainders into an inverse compatible with the energy solution.

Fix \(p>2\) in one of the ranges of Theorem~\ref{thm:intro-forward}.
For Dirichlet data, let
\[
 \dot W^{1,r}_0(\Omega)
 :=\overline{C_c^\infty(\Omega)}^{\,\|\nabla\cdot\|_r},
 \qquad
 \dot W^{-1,r}_D(\Omega)
 :=\bigl(\dot W^{1,r'}_0(\Omega)\bigr)^*.
\]
When \(r<n\), every element of \(\dot W^{1,r}_0(\Omega)\) has a unique
representative in \(L^{r^*}(\Omega)\), where \(r^*=nr/(n-r)\).  We call
this the canonical Sobolev representative.  

% Thus this convention concerns
% the solution space \(\dot W^{1,r}_0(\Omega)\), not its negative dual.

For Neumann data, use
\[
 \dot W^{1,r}(\Omega)/\mathbb C
 :=\{v\in W^{1,r}_{\mathrm{loc}}(\Omega):
          \nabla v\in L^r(\Omega)\}/\mathbb C,
 \qquad
 \dot W^{-1,r}_N(\Omega)
 :=\bigl(\dot W^{1,r'}(\Omega)/\mathbb C\bigr)^*.
\]
Equivalently, every \(F\in\dot W^{-1,r}_N(\Omega)\) annihilates constant
test functions: \(\langle F,1\rangle=0\).  We write
\(\dot W^{-1,r}_B(\Omega)\) for the corresponding dual space in either
boundary condition.
On the whole space, we use the shorthand
\[
 \dot W^{-1,r}(\mathbb R^n)
 :=\bigl(\dot W^{1,r'}(\mathbb R^n)/\mathbb C\bigr)^*.
\]

For each class \([v]\in\dot W^{1,r}(\Omega)/\mathbb C\), use the function
\(\eta\) already fixed on \(\mathcal A_2\), and choose the unique representative
satisfying \(\int_\Omega\eta v=0\).  If \(K\cap\Omega\) is bounded,
connected, Lipschitz, and contains \(\supp\eta\), the Poincar\'e inequality
implies
\begin{equation}\label{eq:homogeneous-local-poincare}
 \|v\|_{W^{1,r}(K\cap\Omega)}
 \leq C_{K,r}\|\nabla v\|_{L^r(\Omega)},
 \qquad \int_\Omega\eta v=0.
\end{equation}
This estimate permits cutoff multiplication and local Rellich compactness.
The homogeneous Neumann space is reflexive because its gradient image is a
closed subspace of \(L^r(\Omega;\mathbb C^n)\).

On the bounded core, we use
\[
 W_D^{1,r}(\Omega_c)=W^{1,r}_0(\Omega_c),\qquad
 W_N^{1,r}(\Omega_c)
 =\left\{v\in W^{1,r}(\Omega_c):\int_{\Omega_c}v=0\right\}.
\]
The second space is the mean-zero realization of
\(W^{1,r}(\Omega_c)/\mathbb C\).  In either case,
\[
 W_B^{-1,r}(\Omega_c)
 :=\bigl(W_B^{1,r'}(\Omega_c)\bigr)^*;
\]
for \(B=N\), the mean-subtraction projection identifies this dual with
the functionals on \(W^{1,r'}(\Omega_c)\) that annihilate constants.

We shall use the following standard facts.  The homogeneous Dirichlet and
Neumann spaces, together with their duals, form compatible complex
interpolation scales on compact exponent intervals.  This follows from the
usual extension and restriction operators for Lipschitz domains and the
retraction theorem; see \cite[Chapter~6]{BerghLofstrom}.  We shall also use
the stability of the Fredholm index on interpolation scales from
\cite{KaltonMitrea98}.  Local compactness will always come from the
Rellich--Kondrachov theorem, the compact-adjoint theorem of Schauder, or a
finite-rank term.  These tools are recalled at the points where they enter
the proof.

% The next lemma uses three elementary ideas.  The Hahn--Banach theorem
% writes a homogeneous negative Sobolev datum as \(-\diver G\).  The Riesz
% transform bounds then control the factorization of \(\nabla L^{-1}(-\diver G)\).  On the bounded core, we first invert the shifted operator \(H_{c,B}\), and then regard \(L_{c,B}\) as  its compact perturbation.

\begin{lemma}
\label{lem:homogeneous-model-inverses}
Let \(1<r<\infty\), and suppose that the whole-space Riesz transform
\(\nabla L^{-1/2}\) is bounded on both \(L^r(\mathbb R^n)\) and
\(L^{r'}(\mathbb R^n)\).  Then every
\[
 F\in\bigl(\dot W^{1,r'}(\mathbb R^n)/\mathbb C\bigr)^*
\]
has a unique solution \(v\in\dot W^{1,r}(\mathbb R^n)/\mathbb C\) of
\(Lv=F\), and
\begin{equation}\label{eq:model-zero-inverse}
 \|\nabla v\|_r\leq C_r
 \|F\|_{(\dot W^{1,r'}(\mathbb R^n)/\mathbb C)^*}.
\end{equation}
If the datum also belongs to
\(\dot W^{-1,2}(\mathbb R^n)\), this solution has the same gradient as
the Lax--Milgram energy solution.

For the same exponent \(r\), suppose that \(2\leq r<p_{c,B}\).  On the
bounded core $\Omega_c$, every datum in the corresponding negative space has a unique solution of
\[
 L_{c,B}v=F
\]
in \(W^{1,r}_0(\Omega_c)\) when \(B=D\), or in
\(W^{1,r}(\Omega_c)/\mathbb C\) when \(B=N\). The solution satisfies the corresponding gradient estimate. Whenever the datum also belongs to the dual energy space at exponent \(2\), this solution agrees with the Lax--Milgram solution, modulo constants in the Neumann case.
\end{lemma}

\begin{proof}
We first explain the whole-space assertion.  Consider the isometry
\[
 J:\dot W^{1,r'}(\mathbb R^n)/\mathbb C
   \longrightarrow L^{r'}(\mathbb R^n;\mathbb C^n),
 \qquad J\psi=\nabla\psi.
\]
The rule \(J\psi\mapsto\langle F,\psi\rangle\) is a bounded
conjugate-linear functional
on \(\operatorname{Ran}J\).  The Hahn--Banach theorem extends it to all of
\(L^{r'}(\mathbb R^n;\mathbb C^n)\), and \(L^{r'}\)-duality produces
\(G\in L^r(\mathbb R^n;\mathbb C^n)\) such that
\[
 \langle F,\psi\rangle
 =\int_{\mathbb R^n}G\mathbin{\cdo}\nabla\overline\psi
 =\langle-\diver G,\psi\rangle,
 \qquad \|G\|_r\leq\|F\|.
\]
Thus \(F=-\diver G\) in the form sense.

The factorization $\nabla L^{-1}F
 =\bigl(\nabla L^{-1/2}\bigr)
  \bigl(L^{-1/2}(-\diver)\bigr)G$ is bounded on \(L^r\), because the second factor is the Banach adjoint of
the Riesz transform on \(L^{r'}\).  To justify the formula, choose
\(G_k\in C_c^\infty(\mathbb R^n;\mathbb C^n)\) with
\(G_k\to G\) in \(L^r\), and put \(F_k=-\diver G_k\).  Ellipticity and the
Lax--Milgram theorem yield a unique class
\(v_k\in\dot W^{1,2}(\mathbb R^n)/\mathbb C\) such that
\[
 \int_{\mathbb R^n}A\nabla v_k\mathbin{\cdo}\nabla\overline\psi
 =\langle F_k,\psi\rangle,
 \qquad \psi\in\dot W^{1,2}(\mathbb R^n)/\mathbb C,
\]
and $\|\nabla v_k\|_{L^2(\mathbb R^n)}
 \leq C\|F_k\|_{\dot W^{-1,2}(\mathbb R^n)}$. In other words, \(v_k=L^{-1}F_k\) in the energy sense.  The factorization
shows that \(\nabla v_k\) is Cauchy in \(L^r\).  Since the gradient
realization of \(\dot W^{1,r}(\mathbb R^n)/\mathbb C\) is closed, the
limit is \(\nabla v\) for a unique class \(v\).  Passing to distributions
proves \(Lv=F\), and the same argument proves
\eqref{eq:model-zero-inverse}.

For uniqueness, suppose that \(Lu=0\) with
\(u\in\dot W^{1,r}(\mathbb R^n)/\mathbb C\).  Given
\(G\in C_c^\infty(\mathbb R^n;\mathbb C^n)\), solve
\(L\psi=-\diver G\) at exponent \(r'\).  Since
\(\nabla\psi\in L^{r'}\) and \(\nabla u\in L^r\), all the following
pairings are absolutely convergent.  Form symmetry gives
\(\langle L\psi,u\rangle=\langle\psi,Lu\rangle\), and hence
\[
 \int_{\mathbb R^n}G\mathbin{\cdo}\nabla\overline u
 =\langle L\psi,u\rangle
 =\langle\psi,Lu\rangle=0.
\]
It follows that \(\nabla u=0\), so \(u\) is zero in the quotient.

We next prove compatibility.  Let
\[
 F\in\dot W^{-1,r}(\mathbb R^n)
       \cap\dot W^{-1,2}(\mathbb R^n),
\]
and let \(v_r\) and \(v_2\) be the solutions obtained at exponents \(r\)
and \(2\), respectively.  Let \(\Delta_{\mathbb R^n}\) be the nonnegative
Euclidean Laplacian and set
\[
 G=\nabla\Delta_{\mathbb R^n}^{-1}F.
\]
Applying the preceding Hahn--Banach representation at exponents \(r\)
and \(2\), write \(F=-\diver H_r=-\diver H_2\), where
\(H_r\in L^r\) and \(H_2\in L^2\).  The maps
\(H_r\mapsto G\) and \(H_2\mapsto G\) are matrices of double Euclidean
Riesz transforms.  Therefore \(G\in L^r\cap L^2\) and
\(F=-\diver G\). Choose one sequence \(G_k\in C_c^\infty(\mathbb R^n;\mathbb C^n)\)
converging to \(G\) in both \(L^r\) and \(L^2\).  For each \(G_k\), the
\(L^r\)- and \(L^2\)-factorizations are the same distributional identity,
because both extend the energy solution for the smooth datum
\(-\diver G_k\).  Passing to the limit in the two norms yields
\[
 \nabla v_r=\nabla v_2
\]
in the sense of distributions, and hence almost everywhere.

We now consider the bounded core.  Put \(H_{c,B}=1+L_{c,B}\).  In the
Neumann case, identify \(W^{1,r}(\Omega_c)/\mathbb C\) with its mean-zero
representatives and identify the dual with the functionals that annihilate
constants.  The definition of \(p_{c,B}\), the standard low-exponent
estimate, and duality imply
\[
 H_{c,B}^{-1/2}:L^r(\Omega_c)\longrightarrow W_B^{1,r}(\Omega_c),
 \qquad
 H_{c,B}^{-1/2}:L^{r'}(\Omega_c)\longrightarrow W_B^{1,r'}(\Omega_c).
\]
Here and below, the Neumann maps are restricted to the mean-zero
subspaces.  By self-adjointness, dualizing the second map and
% \[
%  H_{c,B}^{-1/2}:W_B^{-1,r}(\Omega_c)\longrightarrow L^r(\Omega_c).
% \]
composing the two half-powers gives a bounded inverse
\[
 H_{c,B}^{-1}:W_B^{-1,r}(\Omega_c)\longrightarrow W_B^{1,r}(\Omega_c),
 \qquad
 \|H_{c,B}^{-1}F\|_{W_B^{1,r}}\leq C\|F\|_{W_B^{-1,r}}.
\]
The identity \(H_{c,B}H_{c,B}^{-1}F=F\) follows first on a dense
functional-calculus core and then for every \(F\) by continuity.  Thus
\(H_{c,B}\) is onto.  It is one-to-one because \(r\geq2\), the boundedness
of \(\Omega_c\) gives \(W_B^{1,r}\subset W_B^{1,2}\), and the energy
identity for \(H_{c,B}u=0\) implies \(u=0\) (by Lax-Milgram). For clarity, if \(B=N\), \(F(1)=0\), and \(H_{c,N}u=F\), then
\begin{equation}\label{eq:core-shifted-weak}
 \int_{\Omega_c}A\nabla u\mathbin{\cdo}\nabla\overline\psi
 +\int_{\Omega_c}u\overline\psi
 =\langle F,\psi\rangle.
\end{equation}
Taking \(\psi=1\) shows that \(\int_{\Omega_c}u=0\).  Hence the inverse
just constructed preserves the chosen mean-zero realization.

Let
\[
 \mathcal J:W_B^{1,r}(\Omega_c)\longrightarrow W_B^{-1,r}(\Omega_c),
 \qquad
 \langle\mathcal Jv,\psi\rangle=\int_{\Omega_c}v\overline\psi,
\]
again on mean-zero spaces when \(B=N\).  Rellich compactness, together
with Poincar\'e's inequality in the Neumann case, shows that
\(\mathcal J\) is compact.  Therefore
\[
 L_{c,B}=H_{c,B}-\mathcal J
        =H_{c,B}(I-H_{c,B}^{-1}\mathcal J)
\]
is Fredholm of index zero by the Riesz--Schauder theorem (and isomorphism does not change the index of $I-H_{c,B}^{-1}\mathcal J$).  If
\(L_{c,B}u=0\), then \(u\in W_B^{1,2}(\Omega_c)\) (again, since $|\Omega_c|<\infty$), and the energy identity
forces \(u=0\) for Dirichlet data and \(\nabla u=0\) for Neumann data.
Thus the kernel is zero in the relevant space.  Index zero now implies
surjectivity, and the bounded inverse theorem yields the required gradient
estimate.  If a datum belongs to both the \(r\)- and \(2\)-dual spaces,
the \(r\)-solution is already an energy function; Lax--Milgram uniqueness
proves compatibility.

Finally, by duality the core inverse conclusion is also valid for \(r\) in
a sufficiently small interval below \(2\), provided \(r'<p_{c,B}\).  We
shall use this only to place the Fredholm family on an open exponent
interval containing \([2,p]\).
\end{proof}

\begin{remark}
\label{rem:model-inverse-compatibility}
The last assertion of Lemma~\ref{lem:homogeneous-model-inverses} will be
used several times.  If a model datum belongs both to the exponent-\(r\)
negative space and to the energy negative space, then its \(r\)-solution
and its Lax--Milgram solution (i.e., the $2$-solution) have the same gradient.  They agree exactly
on the bounded Dirichlet core and modulo constants on a Neumann model.
On the whole-space Dirichlet end, when the canonical Sobolev
representatives at both exponents are available, they agree as functions:
their difference is constant, and no nonzero constant belongs to the sum
of two finite-exponent Lebesgue spaces on \(\mathbb R^n\).  Thus this
remark permits an exponent-\(p\) model solution to be read as an energy
solution once its datum is known to belong to the energy negative space.
\end{remark}

\subsection{Dirichlet boundary condition}

We first recall the tools used below.  If \(K\) is a bounded Lipschitz
domain, then the inclusion map
\[
 W^{1,r}(K)\longrightarrow L^r(K)
\]
is compact by the Rellich--Kondrachov theorem
\cite[Theorem~6.3]{AdamsFournier03}.  Schauder's compact-adjoint theorem
\cite[Theorem~6.4]{Brezis11} says that a bounded operator is compact if
and only if its Banach adjoint is compact.  Define
\(\iota_{K,r}:L^r(K)\to\dot W^{-1,r}_D(\Omega)\) by
\[
 \langle\iota_{K,r}b,\psi\rangle
 =\int_K b\,\overline\psi,
 \qquad \psi\in\dot W^{1,r'}_0(\Omega).
\]
This functional realization is compact because it is the Banach adjoint
of the compact restriction
\(\dot W^{1,r'}_0(\Omega)\to L^{r'}(K)\).

% The same statement holds for the two model negative spaces.  Atkinson's
% theorem \cite{Atkinson51} says that an operator with a left and a right
% inverse modulo compact operators is Fredholm.  Finally, the Fredholm index is locally
% constant on a compatible interpolation scale \cite[Theorem~2.9]{KaltonMitrea98}.

\begin{theorem}
\label{thm:dirichlet-exterior-inverse}
Let \(n\geq3\) and
\[
 2<p<\min\{n,P_D\}.
\]
Then
\[
 L_D:\dot W^{1,p}_0(\Omega)\longrightarrow\dot W^{-1,p}_D(\Omega)
\]
is an isomorphism.  Equivalently, every
\(F\in\dot W^{-1,p}_D(\Omega)\) has a unique solution
\(v\in\dot W^{1,p}_0(\Omega)\) of \(L_Dv=F\), and
\begin{equation}\label{eq:homogeneous-isomorphism}
 \|\nabla v\|_p\leq C_p\|F\|_{\dot W^{-1,p}_D(\Omega)}.
\end{equation}
If \(F\) also belongs to \(\dot W^{-1,2}_D(\Omega)\), then this solution
coincides with the homogeneous Lax--Milgram solution.
\end{theorem}

\begin{proof}
Choose \(\varepsilon>0\) so small that
\[
 I=(2-\varepsilon,p+\varepsilon)
 \quad\text{satisfies}\quad
 p+\varepsilon<\min\{n,P_D\},
 \qquad (2-\varepsilon)'<\min\{n,P_D\}.
\]
The model inverses in Lemma~\ref{lem:homogeneous-model-inverses}, together
with its final duality observation, are available for every \(r\in I\).

\emph{Step 1: parametrix construction.}
Let $r\in I$. Use the partition \(\rho_c+\rho_e=1\) and the enlarged cutoffs
\(\phi_c,\phi_e\) from Section~3.  Thus \(\phi_j=1\) on a neighborhood of
\(\supp\rho_j\), and all cutoff derivatives are supported in one bounded
transition region, say $K$.  Localization of a functional $F$ is defined by
\[
 \langle\rho_jF,\psi\rangle=\langle F,\rho_j\psi\rangle.
\]
Put
\[
 S_{c,r}=L_{c,D}^{-1},\qquad S_{e,r}=L^{-1},
\]
where the whole-space inverse is taken with its canonical Sobolev
representative, and define
\[
 Q_rF=\phi_cS_{c,r}(\rho_cF)+\phi_eS_{e,r}(\rho_eF).
\]
The Sobolev inequality, cutoff multiplication, and
Lemma~\ref{lem:homogeneous-model-inverses} show that
\[
 Q_r:\dot W^{-1,r}_D(\Omega)\longrightarrow\dot W^{1,r}_0(\Omega)
\]
is bounded.

\emph{Step 2: compactness of the remainder.}
Straightforwardly, the form product rule yields
\begin{align*}
 L_DQ_rF-F
 &=\sum_{j\in\{c,e\}}[L_D,\phi_j]S_{j,r}(\rho_jF),\\
 Q_rL_Dv-v
 &=-\sum_{j\in\{c,e\}}\phi_jS_{j,r}[L_D,\rho_j]v.
\end{align*}
Here
\[
 [L_D,\phi]v
 =-\diver(vA\nabla\phi)-A\nabla v\mathbin{\cdo}\nabla\phi
\]
in the form sense.

We check compactness of the two terms on the commutator separately.  For a bounded sequence of model solutions, say $\{v_k\}$, local Rellich compactness makes \(\{v_k A\nabla\phi\}\) relatively compact in \(L^r\) on the bounded transition set. The bounded map
\(H\mapsto-\diver H\) on $L^r(K;\mathbb C^n) \to \dot W^{-1,r}_D(\Omega)$ therefore makes the divergence part compact in the negative space.  The scalar coefficient
\(-A\nabla v\mathbin{\cdo}\nabla\phi\) is only bounded in \(L^r(K)\).
The compact inclusion \(L^r(K)\to\dot W^{-1,r}_D(\Omega)\), and its model
versions, makes the scalar part compact.  The same argument applies to
the cutoffs \(\rho_j\).  Hence both \(L_DQ_r-I\) and \(Q_rL_D-I\) are
compact.

\emph{Step 3: invertibility.}
Atkinson's theorem \cite{Atkinson51} shows that
\[
 L_D:\dot W^{1,r}_0(\Omega)\longrightarrow\dot W^{-1,r}_D(\Omega)
\]
is Fredholm for every \(r\in I\).  These are compatible realizations of
the same form operator on compatible complex interpolation scales.  The
scales have the common smooth core and the usual intersection property.
The interpolation theorem of Kalton and Mitrea \cite[Theorem~2.9]{KaltonMitrea98} therefore makes the index locally constant on \(I\).  Since \(I\) is connected and homogeneous
Lax--Milgram makes the realization at \(r=2\) an isomorphism, hence,
\[
 \operatorname{ind}L_{D,r}=\operatorname{ind}L_{D,2}=0,
 \qquad r\in I.
\]
Suppose that \(v\in\dot W^{1,p}_0(\Omega)\) and \(L_Dv=0\).  The left
parametrix identity shows
\[
 v=\sum_{j\in\{c,e\}}\phi_jS_{j,p}C_j(v),
 \qquad C_j(v)=[L_D,\rho_j]v.
\]
Each source is supported in $K$ and has the form
\[
 C_j(v)=-\diver(vA\nabla\rho_j)
         -A\nabla v\mathbin{\cdo}\nabla\rho_j.
\]
Since \(p>2\), local Sobolev and finite-measure inclusions imply
\[
 vA\nabla\rho_j\in L^2,
 \qquad
 A\nabla v\mathbin{\cdo}\nabla\rho_j
 \in L^p(K)\subset L^{2n/(n+2)}(K).
\]
The second inclusion places the scalar term in the energy negative space
by the homogeneous Sobolev inequality.  Thus each \(C_j(v)\) belongs to
both the exponent-\(p\) and energy negative spaces of the corresponding
model.  Remark~\ref{rem:model-inverse-compatibility} promotes every model
solution in the last display to an energy solution.  Multiplication by
\(\phi_j\) preserves the energy space; hence
\(v\in\dot W^{1,2}_0(\Omega)\).

The equation holds in the energy dual, so testing with \(v\) gives
\[
 0=\int_\Omega
 A\nabla v\mathbin{\cdo}\nabla\overline v
 \geq\lambda\|\nabla v\|_2^2.
\]
Therefore \(v=0\).  Since the Fredholm index at \(p\) is zero, \(L_D\) is onto.  The bounded inverse theorem
proves \eqref{eq:homogeneous-isomorphism}.

\emph{Step 4: compatibility.}
Let
\(F\in\dot W^{-1,p}_D(\Omega)\cap\dot W^{-1,2}_D(\Omega)\), and let \(v\)
be the \(p\)-solution.  The left identity gives
\[
 v=Q_pF+
 \sum_{j\in\{c,e\}}\phi_jS_{j,p}[L_D,\rho_j]v.
\]
The localized data in \(Q_pF\) belong to both model negative spaces, and
the preceding compact-support argument puts every commutator source in
both negative spaces as well.  Remark~\ref{rem:model-inverse-compatibility}
therefore shows that \(v\in\dot W^{1,2}_0(\Omega)\).  It satisfies the
energy equation, so Lax--Milgram uniqueness identifies it with the energy
solution.
\end{proof}

\subsection{Neumann boundary condition}

The Neumann proof uses the same Fredholm tools and three additional facts.
First, the normalization in \eqref{eq:homogeneous-local-poincare} turns a
bounded family of quotient classes into a bounded family in
\(W^{1,r}(K)\), so local Rellich applies.  Second,
\[
 L^r_0(K):=\left\{b\in L^r(K):\int_Kb=0\right\}
 \longrightarrow\dot W^{-1,r}_N(\Omega)
\]
is compact.  Indeed, it is the Banach adjoint of the compact normalized
restriction
\[
 \dot W^{1,r'}(\Omega)/\mathbb C
 \longrightarrow L^{r'}(K)/\mathbb C.
\]
Third, the Bogovski\u\i\ operator on a bounded connected Lipschitz domain
\(K\) maps \(L^2_0(K)\) boundedly into \(W^{1,2}_0(K;\mathbb C^n)\) and
satisfies \cite[Theorem~2.5]{GeissertHeckHieber06}
\[
 \diver(\mathcal B_Kb)=b.
\]
For complex-valued \(b\), apply the real Bogovski\u\i\ operator to its
real and imaginary parts.

\begin{theorem}
\label{thm:neumann-exterior-inverse}
Let \(n\geq2\) and \(2<p<P_N\).  Then
\[
 L_N:\dot W^{1,p}(\Omega)/\mathbb C
 \longrightarrow\dot W^{-1,p}_N(\Omega)
\]
is an isomorphism.  Thus every \(F\in\dot W^{-1,p}_N(\Omega)\) has a
unique solution modulo constants \(v\in\dot W^{1,p}(\Omega)/\mathbb C\)
of \(L_Nv=F\), and
\[
 \|\nabla v\|_p\leq C_p\|F\|_{\dot W^{-1,p}_N(\Omega)}.
\]
If \(F\) also belongs to \(\dot W^{-1,2}_N(\Omega)\), then the
\(p\)-solution agrees modulo constants with the homogeneous Lax--Milgram
solution.
\end{theorem}

\begin{proof}
Choose \(\varepsilon>0\) such that
\[
 J=(2-\varepsilon,p+\varepsilon),
 \qquad p+\varepsilon<P_N,
 \qquad (2-\varepsilon)'<P_N.
\]
The model inverses and the compatible interpolation scales are available
throughout \(J\).

\emph{Step 1: parametrix construction.}
Use the cutoffs \(\rho_c,\rho_e,\phi_c,\phi_e\), enlarging their transition
neighborhoods if necessary.  Choose \(\theta\in C_c^\infty(\Omega)\) in
the overlap such that $\int_\Omega\theta=1$, $\phi_c=\phi_e=1$ on $\supp\theta$. For \(r\in J\) and \(F\in\dot W^{-1,r}_N(\Omega)\), set
\[
 a(F)=\langle F,\rho_e\rangle,
\]
which is well defined because \(\nabla\rho_e\) is compactly supported.
Moreover, $|a(F)|\leq \|F\|_{\dot W^{-1,r}_N} \|\nabla\rho_e\|_{r'}$.

Define 
\begin{align*}
 \langle F_e,\psi\rangle
 =\langle F,\rho_e\psi\rangle-a(F)\int_\Omega\theta\overline\psi, \qquad
 \langle F_c,\psi\rangle
 =\langle F,\rho_c\psi\rangle+a(F)\int_\Omega\theta\overline\psi.
\end{align*}
Note that if \(\psi\) is
replaced by \(\psi+c\), the above two terms change by equal and
opposite multiples of \(\overline c\).  Hence \(F_e\) and \(F_c\) are well defined on
classes modulo constants and
\[
 F_e(1)=a(F)-a(F)=0,
 \qquad
 F_c(1)=\langle F,1-\rho_e\rangle+a(F)=0.
\]
Moreover, normalized Poincar\'e and cutoff multiplication show that the maps
\[
 F\longmapsto F_e:
 \dot W^{-1,r}_N(\Omega)\longrightarrow
 \dot W^{-1,r}(\mathbb R^n),
 \qquad
 F\longmapsto F_c:
 \dot W^{-1,r}_N(\Omega)\longrightarrow
 W^{-1,r}_N(\Omega_c)
\]
are bounded.

Let \(S_{e,r}\) and \(S_{c,r}\) be the model inverses from
Lemma~\ref{lem:homogeneous-model-inverses}, with representatives normalized
by fixed averages on the transition sets.  For simplicity, write
\(L_e=L\) and \(L_c=L_{c,N}\).  Put
\[
 Q_rF=\phi_eS_{e,r}F_e+\phi_cS_{c,r}F_c.
\]

\emph{Step 2: compactness of the right remainder.}
Put \(v_j=S_{j,r}F_j\).  Since \(L_jv_j=F_j\), we have
\[
 L_NQ_rF
 =\phi_eF_e+\phi_cF_c
  +\sum_{j\in\{c,e\}}[L_N,\phi_j]v_j.
\]
The function \(\phi_j\) equals one on the support of \(F_j\), which is
contained in the union of the supports of \(\rho_j\) and \(\theta\).
Hence \(\phi_jF_j=F_j\).  Moreover, for every test function \(\psi\),
\[
 \begin{aligned}
 \langle F_e+F_c,\psi\rangle
 =\langle F,(\rho_e+\rho_c)\psi\rangle
   +[-a(F)+a(F)]\int_\Omega\theta\overline\psi=\langle F,\psi\rangle.
 \end{aligned}
\]
Thus the two multiples of \(\theta\) cancel and the principal terms add
to \(F\).  Consequently,
\[
 L_NQ_rF-F
 =\sum_{j\in\{c,e\}}[L_N,\phi_j]S_{j,r}F_j.
\]
The scalar coefficient of the \(j\)-th commutator is $b_j=-A\nabla v_j\mathbin{\cdo}\nabla\phi_j$. It has zero integral:
\[
 \int b_j
 =-\int A\nabla v_j\mathbin{\cdo}\nabla\phi_j
 =-\langle F_j,\phi_j\rangle
 =-\langle F_j,1\rangle=0.
\]
The last equality uses that \(\phi_j=1\) on the support of \(F_j\).
Normalized local Poincar\'e and Rellich compactness make the divergence part of the commutator compact. The compact map
\(L^r_0(K)\to\dot W^{-1,r}_N(\Omega)\) controls the scalar part. Hence \(L_NQ_r-I\) is compact.

\emph{Step 3: compactness of the left remainder.}
Before writing the commutators, choose the representative of
each input class \(v\) by \(\int\eta v=0\).  Put \(F=L_Nv\) and define
\[
 C_e(v)=[L,\rho_e]v+a(F)\theta,
 \qquad
 C_c(v)=[L_{c,N},\rho_c]v-a(F)\theta.
\]
The product rule gives $F_j=L_j(\rho_jv)-C_j(v)$ and $C_j(v)(1)=0$. More explicitly, $C_j(v)=-\diver(vA\nabla\rho_j)+b_j(v)$, where
\[
 b_e(v)=-A\nabla v\mathbin{\cdo}\nabla\rho_e+a(F)\theta,
 \qquad
 b_c(v)=-A\nabla v\mathbin{\cdo}\nabla\rho_c-a(F)\theta.
\]
Since \(\rho_c+\rho_e=1\), it is clear that
\[
 a(L_Nv)=\int A\nabla v\mathbin{\cdo}\nabla\rho_e,
 \qquad
 \int A\nabla v\mathbin{\cdo}\nabla\rho_c=-a(L_Nv).
\]
Consequently, $\int b_e(v)=\int b_c(v)=0$.

Because the model inverses act on quotient classes, their fixed
normalizations yield
\[
 S_{j,r}L_j(\rho_jv)=\rho_jv-c_j(v),
\]
where \(c_j\) is a bounded linear functional.  Thus
\[
 Q_rL_Nv-v
 =-\sum_{j\in\{c,e\}}\phi_jS_{j,r}C_j(v)
  -\sum_{j\in\{c,e\}}c_j(v)\phi_j.
\]
The first sum is compact by a similar argument before while the second sum is finite rank. Hence \(Q_rL_N-I\) is compact.

\emph{Step 4: invertibility.}
Now, Atkinson's theorem guarantees that $L_N:\dot W^{1,r}(\Omega)/\mathbb C \longrightarrow\dot W^{-1,r}_N(\Omega)$ is Fredholm for every \(r\in J\). Then Kalton--Mitrea theorem and homogeneous Lax--Milgram at \(r=2\) imply $\operatorname{ind}L_{N,r}=\operatorname{ind}L_{N,2}=0$ for all $r\in J$.

Suppose that \(v\in\dot W^{1,p}(\Omega)/\mathbb C\) and \(L_Nv=0\), and use the fixed normalization above.  The left identity expresses \(v\),
modulo constants, through the compactly supported sources \(C_j(v)\) and the finite-rank cutoff terms.  Since \(p>2\),
\[
 vA\nabla\rho_j\in L^2(K;\mathbb C^n),
 \qquad b_j(v)\in L^2_0(K).
\]
Choose the bounded connected Lipschitz domain \(K\) so that it
contains all these supports and
\(\overline K\subset\Omega\cap\Omega_c\).  The Bogovski\u\i\
operator produces \(B_j\in W^{1,2}_0(K;\mathbb C^n)\) such that
\[
 \diver B_j=b_j(v).
\]
After extension by zero, $C_j(v)=-\diver\bigl(vA\nabla\rho_j-B_j\bigr)$. Thus each \(C_j(v)\) belongs to the energy negative space on the end, the core, and the exterior domain.  Remark~\ref{rem:model-inverse-compatibility} promotes every model solution in the left identity to the energy space. The finite-rank cutoff functions also belong to that space.  Therefore
\[
 v\in\dot W^{1,2}(\Omega)/\mathbb C.
\]
Testing with \(v\) gives
\[
 0=\operatorname{Re}\int_\Omega
 A\nabla v\mathbin{\cdo}\nabla\overline v
 \geq\lambda\|\nabla v\|_2^2.
\]
Thus \(v\) is constant and represents zero in the quotient.  Index zero
implies surjectivity, and the bounded inverse theorem proves the gradient
estimate.

\emph{Step 5: compatibility.}
Let
\(F\in\dot W^{-1,p}_N(\Omega)\cap\dot W^{-1,2}_N(\Omega)\), and apply the
left identity to its \(p\)-solution.  The balanced localized data belong
to both model negative spaces.  The zero-mean scalar commutators have the
\(L^2\)-divergence representation just constructed, and the divergence
coefficients are locally in \(L^2\).  Remark~\ref{rem:model-inverse-compatibility}
therefore promotes the \(p\)-solution to
\(\dot W^{1,2}(\Omega)/\mathbb C\).  Lax--Milgram uniqueness modulo
constants identifies it with the energy solution.
\end{proof}

From now on, \(L_B^{-1}\) denotes the zero-energy inverse furnished by the
appropriate theorem above. Its
compatibility with the energy solution will be used in the construction
of the correction term.

\section{Error test}
\label{sec:forward-caccioppoli}

% Section~\ref{sec:forward-parametrix} controls the explicit gradient term
% \(\mathcal R_Bu\).  We now construct and estimate the correction.

For a scalar spectral parameter \(\sigma>0\), the Dirichlet Green kernel
of \(-\partial_t^2+\sigma\) on \((0,\infty)\) is
\[
 \mathcal G_\sigma(t,s)
 =\frac{1}{2\sqrt\sigma}
  \left(e^{-\sqrt\sigma|t-s|}-e^{-\sqrt\sigma(t+s)}\right),
 \qquad s,t>0.
\]
A direct integration in \(t\) gives
\[
 \int_0^\infty\mathcal G_\sigma(t,s)\,\dd t
 =\frac{1-e^{-s\sqrt\sigma}}{\sigma}.
\]
This calculation motivates the following order of construction.  For each \(s>0\), define the
balanced error on the homogeneous test space by
\begin{equation}\label{eq:pointwise-balanced-error}
 \langle G_{B,s}u,\psi\rangle
 :=\langle F_{B,s}u,(I-e^{-s\sqrt{L_B}})\psi\rangle.
\end{equation}
This is the rigorous meaning of the formal expression
\[
 G_{B,s}u=(I-e^{-s\sqrt{L_B}})F_{B,s}u
\]
when \(F_{B,s}u\) is form-valued.  For Neumann data,
\((I-e^{-s\sqrt{L_N}})\psi\) is unchanged when a constant is added to
\(\psi\).  Hence \eqref{eq:pointwise-balanced-error} is well defined on
the quotient by constants, even though \(F_{N,s}u\) itself need not
annihilate constants.

Once we have proved that \(G_{B,s}u\) belongs to
\(\dot W^{-1,p}_B(\Omega)\), the inverse constructed in
Section~\ref{sec:homogeneous-invertibility} allows us to set
\begin{equation}\label{eq:pointwise-correction}
 g_{B,s}u=L_B^{-1}G_{B,s}u.
\end{equation}
We shall then integrate \(g_{B,s}u\) in \(s\).  Thus the exterior
operator is inverted before the height variable is integrated.  The
finite-cylinder identity in Section~\ref{sec:proof-forward} will show that
\[
 \nabla L_B^{-1/2}u
 =\mathcal R_Bu-\int_0^\infty\nabla g_{B,s}u\,\dd s.
\]
Notice that
\(g_{B,s}u\) is the contribution of the error produced at height
\(s\) after the Green kernel has been integrated in its first variable;
it is not a pointwise solution of the cylinder equation at height \(s\).

Throughout this section, \(p>2\) lies in one of the ranges of Theorem~\ref{thm:intro-forward}; in particular, \(1<p'<2\). The next lemma links the compactly supported \(W_B^{-1,r}\)-errors from
Section~4 with the homogeneous dual spaces used here.  In its estimates,
\(\psi\in\dot W^{1,p'}_0(\Omega)\) for \(B=D\) and
\(\psi\in\dot W^{1,p'}(\Omega)/\mathbb C\) for \(B=N\).

\begin{lemma}
\label{lem:homogeneous-Poisson-tests}
Let \(s>0\), and let \(\psi\) belong to the appropriate homogeneous
test space at exponent \(p'\).  Put $w_s=(I-e^{-s\sqrt{L_B}})\psi$. Then the following estimates hold.

\begin{enumerate}[label=\textup{(\roman*)},leftmargin=*]
\item The Poisson semigroup is uniformly bounded on the homogeneous
test space:
\[
 \|\nabla e^{-s\sqrt{L_B}}\psi\|_{p'}
 \leq C\|\nabla\psi\|_{p'}.
\]

\item For every fixed bounded Lipschitz transition neighborhood
\(K\subset\overline\Omega\),
\begin{equation}\label{eq:homogeneous-Poisson-local}
 \|w_s\|_{W^{1,p'}(K\cap\Omega)}
 \leq C_K\|\nabla\psi\|_{p'}.
\end{equation}
Moreover,
\begin{equation}\label{eq:homogeneous-Poisson-global}
 \|w_s\|_{p'}
 \leq Cs\|\nabla\psi\|_{p'}.
\end{equation}

\item Let \(p\leq r<\infty\).  If a form
\(F\in W_B^{-1,r}(\Omega)\) is supported in \(K\), then
\begin{equation}\label{eq:compact-form-test}
 |\langle F,w_s\rangle|
 \leq C_K\|F\|_{W_B^{-1,r}(\Omega)}
       \|\nabla\psi\|_{p'}.
\end{equation}
If \(F=f+\diver H\), where \(f\in L^r(K)\) and
\(H\in L^r(K;\mathbb R^n)\), then
\[
 |\langle F,w_s\rangle|
 \leq C_K\bigl(\|f\|_{L^r(K)}+\|H\|_{L^r(K)}\bigr)
       \|\nabla\psi\|_{p'}.
\]
\end{enumerate}
\end{lemma}

\begin{proof}
We first work on the smooth homogeneous test classes.  The reverse
square-root estimates in Theorem~\ref{thm:intro-neumann-reverse} and
Remark~\ref{rem:Dirichlet-reverse} imply $\|L_B^{1/2}\psi\|_{p'}
 \lesssim\|\nabla\psi\|_{p'}$. By \eqref{eq:exterior-low-Riesz}, Poisson contractivity, and commutation
of the spectral multipliers,
\begin{align*}
 \|\nabla e^{-s\sqrt{L_B}}\psi\|_{p'}
 &=\|\nabla L_B^{-1/2}e^{-s\sqrt{L_B}}L_B^{1/2}\psi\|_{p'}\\
 &\lesssim\|e^{-s\sqrt{L_B}}L_B^{1/2}\psi\|_{p'}
 \leq\|L_B^{1/2}\psi\|_{p'}
 \lesssim\|\nabla\psi\|_{p'}.
\end{align*}
This proves part~\textup{(i)} and extends the semigroup continuously to
the homogeneous Dirichlet space and to the homogeneous Neumann space
modulo constants.

The identity
\[
 w_s=\int_0^s e^{-t\sqrt{L_B}}L_B^{1/2}\psi\,\dd t
\]
holds first on the smooth test classes.  In the Neumann case it also
defines an actual \(L^{p'}\)-representative of \(w_s\), independent of
the representative chosen for \(\psi\).  Contractivity and the reverse
square-root estimate show that
\[
 \|w_s\|_{p'}
 \leq\int_0^s\|e^{-t\sqrt{L_B}}L_B^{1/2}\psi\|_{p'}\,\dd t
 \lesssim s\|\nabla\psi\|_{p'}.
\]
This proves \eqref{eq:homogeneous-Poisson-global}.

We next prove the uniform local estimate.  Put
\(q=(p')^*=np'/(n-p')\); since \(p>2\), one has \(p'<n\).
For Dirichlet data use the canonical Sobolev representative.  For
Neumann data, \eqref{eq:exterior-Sobolev} extends by completion and
selects the unique representative \(\psi^\circ\in L^q(\Omega)\) of the
class \([\psi]\), with
\(\|\psi^\circ\|_q\lesssim\|\nabla\psi\|_{p'}\).  Its uniqueness follows
because \(\Omega\) has infinite measure.  In either case, write this
representative again as \(\psi\).  Poisson contractivity on \(L^q\) and
the finite measure of \(K\cap\Omega\) yield
\begin{align*}
 \|w_s\|_{L^{p'}(K\cap\Omega)}
 &\leq |K\cap\Omega|^{1/p'-1/q}
 \bigl(\|\psi\|_q+\|e^{-s\sqrt{L_D}}\psi\|_q\bigr)\\
 &\lesssim_K\|\psi\|_q
 \lesssim\|\nabla\psi\|_{p'}.
\end{align*}
This argument is valid for both boundary conditions because the Neumann
Poisson semigroup is contractive on \(L^q\), while
\(I-e^{-s\sqrt{L_N}}\) annihilates constants.  

Finally, part~\textup{(i)} gives
\[
 \|\nabla w_s\|_{p'}
 \leq\|\nabla\psi\|_{p'}
     +\|\nabla e^{-s\sqrt{L_B}}\psi\|_{p'}
 \lesssim\|\nabla\psi\|_{p'}.
\]
This completes the proof of part~\textup{(ii)}.

For part~\textup{(iii)}, enlarge \(K\), if necessary, to a bounded
connected Lipschitz transition neighborhood \(K_1\), and choose a
compactly supported Lipschitz function \(\chi\) such that
\(\chi=1\) near \(\supp F\) and \(\supp\chi\subset K_1\).  Since \(F\)
is supported in \(K\), $\langle F,w_s\rangle=\langle F,\chi w_s\rangle$. Multiplication by \(\chi\) preserves the relevant boundary condition. Moreover, \(r\geq p\) implies \(r'\leq p'\).  Hence
\[
 \|\chi w_s\|_{W^{1,r'}(\Omega)}
 \lesssim\|w_s\|_{W^{1,p'}(K_1\cap\Omega)}
 \lesssim_K\|\nabla\psi\|_{p'}.
\]
Testing \(F\) against \(\chi w_s\) proves
\eqref{eq:compact-form-test}.  If \(F=f+\diver H\), then
\[
 \langle F,w_s\rangle
 =\int_K f\,\overline{w_s}
  -\int_K H\mathbin{\cdo}\nabla\overline{w_s},
\]
and H\"older's inequality together with
\eqref{eq:homogeneous-Poisson-local} proves the coefficient estimate.
\end{proof}

% We now estimate the functional that will be inverted.  Notice that
% \(F_{N,s}u\) itself need not annihilate constants.  The test
% \((I-e^{-s\sqrt{L_N}})\psi\), however, is independent of the representative
% of \(\psi\), so the functional below is well defined on the homogeneous
% Neumann space. 

We shall use the following Davies--Gaffney estimates. The proof is included for the sake of completeness.

\begin{lemma}\label{lem:automatic-ST}
Let \(B\in\{D,N\}\), let \(2<p<P_B\), and let \(T\) denote either
\(L\) on \(\mathbb R^n\) or \(H_{c,B}=1+L_{c,B}\) on the bounded core.
If measurable sets \(E,F\) satisfy \(d(E,F)\ge d>0\), then
\begin{equation}\label{eq:separated-short-time-conclusion}
 \|e^{-s\sqrt T}(\mathbf 1_Fg)\|_{L^p(E)}
 +\|\nabla e^{-s\sqrt T}(\mathbf 1_Fg)\|_{L^p(E)}
 \le C_{p,d}\,s\,\|g\|_{L^p},
 \qquad 0<s\le1 .
\end{equation}
\end{lemma}

\begin{proof}
Choose \(r\) with \(p<r<P_B\).  The Davies--Gaffney argument, together
with the form Caccioppoli estimate of Lemma~\ref{lem:local-caccioppoli},
yields the two \(L^2\) off-diagonal families
\[
 \mathbf 1_Ee^{-tT}\mathbf 1_F,\qquad
 \sqrt t\,\mathbf 1_E\nabla e^{-tT}\mathbf 1_F
\]
with norm \(O(e^{-c d^2/t})\).  This is the standard form of the
Gaffney estimates; see \cite[\S3.1, in particular (3.1)]{ACDH04}.
The cutoff proof applies to both pure Dirichlet and pure Neumann form
domains, and the harmless shift in \(H_{c,B}\) only improves the bound.
For completeness, the gradient family follows from
Lemma~\ref{lem:local-caccioppoli}: if
\(v=e^{-tT}(\mathbf1_Fg)\) and \(0<t<d^2/64\), first replace $E$ by
$E_N=E\cap B(0,N)$.  Choose bounded open sets
$E_N\subset U_1$ and $\overline{U_1}\subset U_2$ so that
\[
 \dist(U_1,\R^n\setminus U_2)\simeq\sqrt t,
 \qquad \dist(U_2,F)\ge d/2.
\]
Apply Lemma~\ref{lem:local-caccioppoli} on $U_2$ with right-hand side
$f=Tv$ and $G=0$.  After multiplication by $\sqrt t$, we have
\[
 \sqrt t\,\|\nabla v\|_{L^2(E_N)}
 \le C\left(
 \|v\|_{L^2(U_2)}
 +\|tTv\|_{L^2(U_2)}^{1/2}\|v\|_{L^2(U_2)}^{1/2}
 \right).
\]
The scalar Davies--Gaffney estimates for $e^{-tT}$ and
$tTe^{-tT}=-t\partial_te^{-tT}$, applied between $F$ and $U_2$, now imply
\[
 \sqrt t\,\|\nabla v\|_{L^2(E_N)}
 \le Ce^{-cd^2/t}\|g\|_2.
\]
Letting $N\to\infty$ removes the temporary boundedness of $E$.  When
$t\ge d^2/64$, the global $L^2$ bounds for $e^{-tT}$ and
$\sqrt t\,\nabla e^{-tT}$ yield the same estimate after increasing the
constant, since $e^{-cd^2/t}$ is then bounded below by a positive
constant depending only on $c$.

The symmetric sub-Markov semigroups are bounded on \(L^r\).  Hence the
first family is uniformly bounded on \(L^r\); so is the second, since $\sqrt t\,\nabla e^{-tT}
 =\nabla T^{-1/2}(\sqrt{tT}\,e^{-tT})$
and the model Riesz transform is bounded on \(L^r\).  Interpolation between
the \(L^2\) off-diagonal estimates and these \(L^r\) bounds yields, for
both families, an \(L^p\) estimate
\(C e^{-c_p d^2/t}\).

Finally, use the subordination formula for \(e^{-s\sqrt T}\).
After inserting the preceding exponential factor, the scalar integrals
corresponding to the zeroth-order and gradient terms are bounded by
\[
 C\!\int_0^\infty s\,t^{-3/2}e^{-c(s^2+d^2)/t}\,dt,\quad
 C\!\int_0^\infty s\,t^{-2}e^{-c(s^2+d^2)/t}\,dt .
\]
For fixed \(d>0\) both are \(O_d(s)\) when \(0<s\le1\), which proves
\eqref{eq:separated-short-time-conclusion}.
\end{proof}

\begin{lemma}
\label{lem:short-time-homogeneous-error}
Let \(B\in\{D,N\}\) and suppose \(2<p<P_B\).  Assume also that
\(p<n\) when \(B=D\).  For every
\(0<s\leq1\), the functional \(G_{B,s}u\) defined in
\eqref{eq:pointwise-balanced-error} belongs to
\(\dot W^{-1,p}_B(\Omega)\), and
\[
 \|G_{B,s}u\|_{\dot W^{-1,p}_B}
 +\|\nabla g_{B,s}u\|_p
 \leq C_p\|u\|_p.
\]
Consequently, for every \(0<\tau\leq1\),
\[
 \int_0^\tau
 \left(\|G_{B,s}u\|_{\dot W^{-1,p}_B}
       +\|\nabla g_{B,s}u\|_p\right)\,\dd s
 \leq C_p\tau\|u\|_p.
\]
These estimates hold for every \(u\in L^p(\Omega)\).
\end{lemma}

\begin{proof}
All terms in \(F_{B,s}u\) are supported in one fixed bounded transition
set \(K\).  For a cutoff \(\phi\), the form product rule yields
\[
 \|[L_B,\phi]v\|_{W_B^{-1,p}}
 \leq C_\phi\bigl(\|v\|_{L^p(K)}
                  +\|\nabla v\|_{L^p(K)}\bigr).
\]
The support of \(\nabla\phi_c\) is separated from \(\supp\rho_c\), and the support of \(\nabla\phi_e\) is separated from \(\supp\rho_e\).
Lemma~\ref{lem:automatic-ST} and Poisson contractivity therefore imply,
for \(0<s\leq1\),
\begin{align*}
 \|[L_B,\phi_c]e^{-s\sqrt{H_{c,B}}}(\rho_cu)\|_{W_B^{-1,p}}
 &\lesssim s\|u\|_p,\\
 \|[L_B,\phi_e]e^{-s\sqrt L}(\rho_eu)\|_{W_B^{-1,p}}
 &\lesssim s\|u\|_p,\\
 \|\phi_ce^{-s\sqrt{H_{c,B}}}(\rho_cu)\|_{W_B^{-1,p}}
 &\lesssim\|u\|_p.
\end{align*}
For Neumann data, writing \(v_s=e^{-s\sqrt L}(\rho_eu)\),
\[
 [L_N,\phi_e](v_s-m_s(u))
 =[L_N,\phi_e]v_s-m_s(u)[L_N,\phi_e]1,
\]
and \eqref{eq:m-small} yields $|m_s(u)|+|m_s''(u)|\leq Cs\|u\|_p$. The formulas for \(F_{D,s}\) and \(F_{N,s}\) now show that
\[
 \|F_{B,s}u\|_{W_B^{-1,p}}\leq C\|u\|_p,
 \qquad 0<s\leq1.
\]
Apply \eqref{eq:compact-form-test} with \(r=p\), take the supremum over
\(\|\nabla\psi\|_{p'}\leq1\), and obtain
\[
 \|G_{B,s}u\|_{\dot W^{-1,p}_B}\leq C\|u\|_p.
\]
The appropriate exterior inverse theorem and
\eqref{eq:pointwise-correction} then imply
\[
 \|\nabla g_{B,s}u\|_p
 \leq C\|G_{B,s}u\|_{\dot W^{-1,p}_B}
 \leq C\|u\|_p.
\]
Integration over \((0,\tau)\) proves the last assertion.
\end{proof}

\begin{lemma}
\label{lem:large-time-homogeneous-error}
Let \(B\in\{D,N\}\) and suppose \(2<p<P_B\).  Assume also that
\(p<n\) when \(B=D\).  Choose \(r\) such that \(p<r<P_B\), and put
\[
 \alpha=n\left(\frac1p-\frac1r\right)>0.
\]
There are \(c>0\) and \(C_p>0\) such that, for every \(s\geq1\),
\[
 \|G_{B,s}u\|_{\dot W^{-1,p}_B}
 +\|\nabla g_{B,s}u\|_p
 \leq C_pa_B(s)\|u\|_p,
\]
where
\[
 a_N(s)=s^{-1-\alpha}+s^{-2-\alpha}+e^{-cs},
 \qquad
 a_D(s)=s^{-1-\alpha}+s^{-n/p}+e^{-cs}.
\]
In particular, these majorants are integrable on \((1,\infty)\).  More
precisely, for every \(R\geq1\), their tails are bounded by
\[
 \omega_N(R)=R^{-\alpha}+R^{-1-\alpha}+e^{-cR},
 \qquad
 \omega_D(R)=R^{-\alpha}+R^{1-n/p}+e^{-cR}.
\]
This holds for every \(u\in L^p(\Omega)\).
\end{lemma}

\begin{proof}
With the \(r\) chosen in the statement, put $w_s=e^{-s\sqrt L}(\rho_eu)-m_s(u)$. Gaussian smoothing and the model Riesz estimate at \(r\) imply
\[
 \|e^{-s\sqrt L}(\rho_eu)\|_r
 \leq Cs^{-\alpha}\|u\|_p,\qquad
 \|\nabla e^{-s\sqrt L}(\rho_eu)\|_r
 \leq Cs^{-1-\alpha}\|u\|_p.
\]
By the definition of \(m_s(u)\), the function \(w_s\) has zero
\(\eta\)-average on \(\mathcal A_2\).  The weighted Poincar\'e inequality
\eqref{eq:weighted-Poincare} therefore yields
\[
 \|w_s\|_{L^r(\mathcal A_2)}
 \lesssim\|\nabla w_s\|_{L^r(\mathcal A_2)}
 =\|\nabla e^{-s\sqrt L}(\rho_eu)\|_{L^r(\mathcal A_2)}
 \lesssim s^{-1-\alpha}\|u\|_p.
\]
Since \(\supp\nabla\phi_e\subset\mathcal A_2\), the form commutator formula
implies
\[
 \|[L_B,\phi_e]w_s\|_{W_B^{-1,r}}
 \leq Cs^{-1-\alpha}\|u\|_p.
\]
For the second derivative, use
\[
 Le^{-s\sqrt L}
 =\bigl(Le^{-(s/2)\sqrt L}\bigr)e^{-(s/2)\sqrt L}.
\]
Analyticity on \(L^r\) controls the first factor by \(Cs^{-2}\), while
Poisson smoothing from \(L^p\) to \(L^r\) controls the second by
\(Cs^{-\alpha}\).  Hence
\[
 \|Le^{-s\sqrt L}(\rho_eu)\|_r
 \leq Cs^{-2-\alpha}\|u\|_p.
\]
Since \(m_s''(u)\) is the \(\eta\)-average of
\(Le^{-s\sqrt L}(\rho_eu)\), H\"older's inequality gives
\[
 |m_s''(u)|\leq Cs^{-2-\alpha}\|u\|_p.
\]
Gaussian smoothing, the shifted-core spectral gap, and the core Riesz
bound at \(r\) yield
\[
 \|[L_B,\phi_c]e^{-s\sqrt{H_{c,B}}}(\rho_cu)\|_{W_B^{-1,r}}
 +\|\phi_ce^{-s\sqrt{H_{c,B}}}(\rho_cu)\|_{W_B^{-1,r}}
 \leq Ce^{-cs}\|u\|_p,
 \qquad s\geq1.
\]
Therefore
\[
 |\langle F_{N,s}u,(I-e^{-s\sqrt{L_N}})\psi\rangle|
 \leq C\bigl(s^{-1-\alpha}+s^{-2-\alpha}+e^{-cs}\bigr)
       \|u\|_p\|\nabla\psi\|_{p'}.
\]
Taking the supremum over \(\|\nabla\psi\|_{p'}\leq1\) bounds
\(G_{N,s}u\) in \(\dot W^{-1,p}_N\).  The Neumann exterior inverse then
bounds \(\|\nabla g_{N,s}u\|_p\) by the same majorant.  This proves the
Neumann estimate in every dimension \(n\geq2\).

For Dirichlet data, decompose the end commutator as
\[
 [L_D,\phi_e]e^{-s\sqrt L}(\rho_eu)
 =[L_D,\phi_e]w_s+m_s(u)T,
 \qquad T=[L_D,\phi_e]1=-\diver(A\nabla\phi_e).
\]
The first term has the estimate above, while the whole-space
Poisson-kernel bound and the definition of $m_s(u)$ give
\[
 |m_s(u)|
 \leq\|\eta\|_1
       \|e^{-s\sqrt L}(\rho_eu)\|_\infty
 \leq Cs^{-n/p}\|\rho_eu\|_p
 \leq Cs^{-n/p}\|u\|_p.
\]

The fixed distribution \(T\) is compactly supported in the interior, so
Lemma~\ref{lem:homogeneous-Poisson-tests} yields
\[
 |\langle F_{D,s}u,(I-e^{-s\sqrt{L_D}})\psi\rangle|
 \leq C\bigl(s^{-1-\alpha}+s^{-n/p}+e^{-cs}\bigr)
       \|u\|_p\|\nabla\psi\|_{p'}.
\]
Taking the supremum over \(\|\nabla\psi\|_{p'}\leq1\) and then applying
the Dirichlet exterior inverse proves the stated bounds for
\(G_{D,s}u\) and \(g_{D,s}u\).  Because \(p<n\), the power
\(s^{-n/p}\) is integrable on \((1,\infty)\).  Integrating the two
majorants over \((R,\infty)\) yields \(\omega_N(R)\) and
\(\omega_D(R)\).
\end{proof}

% \begin{proposition}
% \label{prop:integrated-homogeneous-error}
% For every \(u\in L^p(\Omega)\), the maps
% \[
%  s\longmapsto G_{B,s}u,
%  \qquad s\longmapsto g_{B,s}u
% \]
% are strongly measurable in the corresponding homogeneous negative and
% solution spaces.  The Bochner integrals
% \begin{equation}\label{eq:integrated-form-error}
%  G_Bu:=\int_0^\infty G_{B,s}u\,\dd s
% \end{equation}
% and
% \begin{equation}\label{eq:homogeneous-corrections}
%  g_Bu:=\int_0^\infty g_{B,s}u\,\dd s
%       =L_B^{-1}G_Bu
% \end{equation}
% converge.  Here the second integral takes values in
% \(\dot W^{1,p}_0(\Omega)\) for \(B=D\) and in
% \(\dot W^{1,p}(\Omega)/\mathbb C\) for \(B=N\).  Moreover,
% \begin{equation}\label{eq:homogeneous-correction-bound}
%  \|G_Bu\|_{\dot W^{-1,p}_B}+\|\nabla g_Bu\|_p
%  \leq C_p\|u\|_p.
% \end{equation}
% \end{proposition}

\begin{proposition}
\label{prop:integrated-homogeneous-error}
For every \(u\in L^p(\Omega)\), the Bochner integrals
\begin{equation}\label{eq:integrated-form-error}
 G_Bu:=\int_0^\infty G_{B,s}u\,\dd s
\end{equation}
and
\begin{equation}\label{eq:homogeneous-corrections}
 g_Bu:=\int_0^\infty g_{B,s}u\,\dd s
      =L_B^{-1}G_Bu
\end{equation}
converge.  Here the second integral takes values in
\(\dot W^{1,p}_0(\Omega)\) for \(B=D\) and in
\(\dot W^{1,p}(\Omega)/\mathbb C\) for \(B=N\).  Moreover,
\begin{equation}\label{eq:homogeneous-correction-bound}
 \|G_Bu\|_{\dot W^{-1,p}_B}+\|\nabla g_Bu\|_p
 \leq C_p\|u\|_p.
\end{equation}
\end{proposition}

\begin{proof}
For smooth \(u\), semigroup analyticity shows that both maps $s\longmapsto G_{B,s}u$ and $s\longmapsto g_{B,s}u$ are
norm-continuous on compact subintervals of \((0,\infty)\).  The general
case follows from the uniform operator estimates in
Lemmas~\ref{lem:short-time-homogeneous-error} and
\ref{lem:large-time-homogeneous-error}.  These estimates also provide
integrable norm majorants, so the Bochner integrals converge and satisfy
\eqref{eq:homogeneous-correction-bound}.  Since \(L_B^{-1}\) is bounded,
it commutes with these integrals. This proves \eqref{eq:homogeneous-corrections} and
\eqref{eq:homogeneous-correction-bound}.
\end{proof}

Finite truncations will be used in the form calculation that follows.
Take \(u\) in the relevant smooth test class and
\(0<\epsilon<R<\infty\), and define
\begin{equation}\label{eq:truncated-form-error}
 G_{B,\epsilon,R}u=\int_\epsilon^R G_{B,s}u\,\dd s,
 \qquad
 g_{B,\epsilon,R}u=\int_\epsilon^R g_{B,s}u\,\dd s
                  =L_B^{-1}G_{B,\epsilon,R}u.
\end{equation}
If
\(0<\epsilon\leq1\leq R\),
\[
 \|G_{B,\epsilon,R}u-G_Bu\|_{\dot W^{-1,p}_B}
 +\|\nabla(g_{B,\epsilon,R}u-g_Bu)\|_p
 \leq C_p\bigl(\epsilon+\omega_B(R)\bigr)\|u\|_p.
\]
Equivalently,
\begin{equation}\label{eq:truncated-correction-limit}
 g_{B,\epsilon,R}u\longrightarrow g_Bu
 \quad\text{in the homogeneous solution space}
\end{equation}
as \(\epsilon\downarrow0\) and \(R\uparrow\infty\).

On each fixed interval \([\epsilon,R]\), semigroup analyticity justifies
the form calculation in the next section.

\section{Proof of Theorem~\ref{thm:intro-forward}}
\label{sec:proof-forward}

For \(u\in L^2(\Omega)\), \(e^{-t\sqrt{L_B}}u\) is the unique spectral solution in the decaying energy class of
\[
 (-\partial_t^2+L_B) e^{-t\sqrt{L_B}} u=0,
 \qquad \lim_{t\to0}e^{-t\sqrt{L_B}} u=u,
\]
with the physical Dirichlet or Neumann boundary condition encoded by
\(L_B\).  Since \(U_t\) has the same initial trace and
\((-\partial_t^2+L_B)U_t=F_{B,t}u\), the half-line Green formula suggests
\[
 e^{-t\sqrt{L_B}} u=U_t-\int_0^\infty
 \frac{e^{-|t-s|\sqrt{L_B}}-e^{-(t+s)\sqrt{L_B}}}
      {2\sqrt{L_B}}F_{B,s}u\,\dd s.
\]
After integration in \(t\), the scalar identity at the beginning of
Section~\ref{sec:forward-caccioppoli} becomes, formally,
\[
 L_B^{-1/2}u
 =\int_0^\infty U_t\,\dd t
  -\int_0^\infty L_B^{-1}(I-e^{-s\sqrt{L_B}})F_{B,s}u\,\dd s.
\]
Thus the second integral is precisely the correction constructed in
Section~\ref{sec:forward-caccioppoli}: at each source height \(s\), we
first form \(G_{B,s}u\), then apply \(L_B^{-1}\), and integrate last.
The displayed Green formula is only a motivation because
\(F_{B,s}u\) is form-valued.  The following finite identity is its
rigorous replacement.

\begin{lemma}
\label{lem:finite-cylinder-identity}
Let \(u\) belong to the relevant smooth test class, and let
\(0<\epsilon<R<\infty\).  For \(B=D\), let
\(\psi\in\dot W^{1,p'}_0(\Omega)\); for \(B=N\), let
\(\psi\in\dot W^{1,p'}(\Omega)/\mathbb C\).  Then
\begin{align}
 \left\langle L_B\left(\int_\epsilon^R U_s\,\dd s
              -\int_\epsilon^R g_{B,s}u\,\dd s\right),\psi\right\rangle
 =\left[
   \langle\partial_sU_s,(I-e^{-s\sqrt{L_B}})\psi\rangle
   -\langle U_s,L_B^{1/2}e^{-s\sqrt{L_B}}\psi\rangle
   \right]_{s=\epsilon}^{s=R}.
 \label{eq:finite-cylinder-identity}
\end{align}
\end{lemma}

\begin{proof}
First let \(\psi\) belong to the common smooth form core.  By
\eqref{eq:parametrix-error-casewise},
\(F_{B,s}u=-\partial_s^2U_s+L_BU_s\).  Since
\(\partial_se^{-s\sqrt{L_B}}
=-L_B^{1/2}e^{-s\sqrt{L_B}}\), direct differentiation yields
\begin{align*}
 \frac{\dd}{\dd s}\Bigl\{
  \langle\partial_sU_s,(I-e^{-s\sqrt{L_B}})\psi\rangle
  -\langle U_s,L_B^{1/2}e^{-s\sqrt{L_B}}\psi\rangle
 \Bigr\}
 ={}&\langle\partial_s^2U_s,(I-e^{-s\sqrt{L_B}})\psi\rangle+\langle L_BU_s,e^{-s\sqrt{L_B}}\psi\rangle\\
 ={}&\langle L_BU_s,\psi\rangle
    -\langle F_{B,s}u,(I-e^{-s\sqrt{L_B}})\psi\rangle.
\end{align*}
Here
\(\langle U_s,L_Be^{-s\sqrt{L_B}}\psi\rangle
=\langle L_BU_s,e^{-s\sqrt{L_B}}\psi\rangle\) is the symmetric form identity. By
\eqref{eq:pointwise-balanced-error} and
\eqref{eq:pointwise-correction}, the last line equals $\langle L_B(U_s-g_{B,s}u),\psi\rangle$. Integration over \((\epsilon,R)\) proves the identity.  Semigroup
analyticity and Lemma~\ref{lem:homogeneous-Poisson-tests} justify the
calculation and its extension to the stated test space.
\end{proof}

\begin{proof}[Proof of Theorem~\ref{thm:intro-forward}]
The standard estimate \eqref{eq:exterior-low-Riesz} proves the theorem for
\(1<p\leq2\).  Fix \(p>2\) in one of the stated ranges.  We first take
\(u\) in the corresponding smooth test class. 

\emph{Step 1: convergence of the explicit potential.}
Define
\begin{align*}
 V_Du={}&\phi_eL^{-1/2}(\rho_eu)
          +\phi_cH_{c,D}^{-1/2}(\rho_cu),\\
 V_Nu={}&\bigl[\phi_e\widetilde W_u
          +\phi_cH_{c,N}^{-1/2}(\rho_cu)\bigr],
\end{align*}
where the brackets denote the class modulo constants.  The normalization
of \(\widetilde W_u\) in Section~\ref{sec:forward-parametrix} makes the
product \(\phi_e\widetilde W_u\) unambiguous.

For Dirichlet data, functional calculus yields
\begin{align*}
 \int_\epsilon^R U_s\,\dd s
 =\phi_eL^{-1/2}
       (e^{-\epsilon\sqrt L}-e^{-R\sqrt L})(\rho_eu)+\phi_cH_{c,D}^{-1/2}
       (e^{-\epsilon\sqrt{H_{c,D}}}
        -e^{-R\sqrt{H_{c,D}}})(\rho_cu).
\end{align*}
The model Riesz bounds, strong stability, the core spectral gap, and
fractional integration on the fixed support of \(\nabla\phi_e\) imply
\[
 \int_\epsilon^R U_s\,\dd s\longrightarrow V_Du
 \quad\text{in }\dot W^{1,p}_0(\Omega).
\]
For Neumann data, integrating the centered parametrix
\eqref{eq:exterior-Devyver} yields, modulo constants,
\begin{align*}
 \int_\epsilon^R U_s\,\dd s
 =\left[{}\phi_e \left((W_R-c_R)-(W_\epsilon-c_\epsilon) \right) +\phi_cH_{c,N}^{-1/2}
       (e^{-\epsilon\sqrt{H_{c,N}}}
        -e^{-R\sqrt{H_{c,N}}})(\rho_cu)\right].
\end{align*}
Equations \eqref{eq:WR-gradient-limit} and \eqref{eq:WR-Cauchy}, together
with the core spectral gap, give
\[
 \int_\epsilon^R U_s\,\dd s\longrightarrow V_Nu
 \quad\text{in }\dot W^{1,p}(\Omega)/\mathbb C.
\]
In both cases Proposition~\ref{prop:explicit-parametrix} implies
\begin{equation}\label{eq:explicit-potential-gradient}
 \nabla V_Bu=\mathcal R_Bu,\qquad
 \|\nabla V_Bu\|_p\leq C_p\|u\|_p.
\end{equation}

\emph{Step 2: the two boundaries.}
The reverse estimate at \(p'\) shows that
\[
 \ell_u(\psi):=\langle u,L_B^{1/2}\psi\rangle
\]
defines a functional in \(\dot W^{-1,p}_B(\Omega)\) with
\(\|\ell_u\|_{\dot W^{-1,p}_B}\lesssim\|u\|_p\).
Equations \eqref{eq:homogeneous-Poisson-local} and
\eqref{eq:homogeneous-Poisson-global} bound the bracket at time \(s\) in
\eqref{eq:finite-cylinder-identity} by $C\bigl(s\|\partial_sU_s\|_p+\|U_s\|_p\bigr)\|\nabla\psi\|_{p'}$. The model-semigroup limits \eqref{eq:model-Poisson-endpoints}, the identities
\[
 m_s(u)=\langle e^{-s\sqrt L}(\rho_eu),\eta\rangle,
 \qquad
 m_s'(u)=-\langle\sqrt L\,e^{-s\sqrt L}(\rho_eu),\eta\rangle.
\]
and the Poisson-kernel bounds recorded in
Section~\ref{sec:exterior-reverse}
imply $\epsilon|m_\epsilon'(u)|\longrightarrow0$, $|m_R(u)|+R|m_R'(u)|\longrightarrow0$. Together with \(m_0(u)=0\), these limits yield, in \(L^p\), as $\epsilon \to 0$: $U_\epsilon\to u$, $\epsilon\|\partial_sU_s|_{s=\epsilon}\|_p\to0$, and, as \(R\to\infty\), $U_R\to0$, $R\|\partial_sU_s|_{s=R}\|_p\to0$. The upper bracket therefore tends to zero in
\(\dot W^{-1,p}_B(\Omega)\).  At the lower endpoint, for\(\|\nabla\psi\|_{p'}\leq1\), 
%self-adjointness and commutation of
% \(e^{-\epsilon\sqrt{L_B}}\) with \(L_B^{1/2}\) imply
\begin{align*}
 |\langle U_\epsilon-u,
       L_B^{1/2}e^{-\epsilon\sqrt{L_B}}\psi\rangle|
 +|\langle u,L_B^{1/2}(e^{-\epsilon\sqrt{L_B}}-I)\psi\rangle|\leq C\bigl(\|U_\epsilon-u\|_p
 +\|(e^{-\epsilon\sqrt{L_B}}-I)u\|_p\bigr).
\end{align*}
For the second term we used the duality identity
\[
 \langle u,L_B^{1/2}(e^{-\epsilon\sqrt{L_B}}-I)\psi\rangle
 =\langle(e^{-\epsilon\sqrt{L_B}}-I)u,L_B^{1/2}\psi\rangle.
\]
The right-hand side tends to zero by strong \(L^p\)-continuity.  Together
with \(\epsilon\|\partial_sU_s|_{s=\epsilon}\|_p\to0\), this shows that the lower
bracket in \eqref{eq:finite-cylinder-identity} converges to \(-\ell_u\)
in the \(\dot W^{-1,p}_B\)-norm.

\emph{Step 3: identification of the two solutions.}
The convergence proved in Step~1, the correction convergence
\eqref{eq:truncated-correction-limit}, the time-boundary limits in
Step~2, and Lemma~\ref{lem:finite-cylinder-identity} imply
\begin{equation}\label{eq:homogeneous-resolvent-identity}
 L_B(V_Bu-g_Bu)=\ell_u
 \qquad\text{in }\dot W^{-1,p}_B(\Omega).
\end{equation}
The datum \(\ell_u\) also belongs to \(\dot W^{-1,2}_B(\Omega)\), because
the energy square-root identity gives
\[
 |\ell_u(\psi)|
 \leq\|u\|_2\|L_B^{1/2}\psi\|_2
 \lesssim\|u\|_2\|\nabla\psi\|_2.
\]

For the present \(u\in L^p\cap L^2\), the low-exponent Riesz inequality implies $\|L_B^{-1/2}u\|_{\dot W_B^{1,2}}\lesssim \|u\|_2$. For every energy test function \(\psi\), it follows by spectral theorem 
\[
 \langle L_BL_B^{-1/2}u,\psi\rangle
 =\langle L_B^{1/2}L_B^{-1/2}u,L_B^{1/2}\psi\rangle
 =\langle u,L_B^{1/2}\psi\rangle
 =\ell_u(\psi).
\]
Thus $L_B^{-1/2}u$ is an energy solution to $L_B v = \ell_u$. The compatibility assertion in
Theorem~\ref{thm:dirichlet-exterior-inverse} or
Theorem~\ref{thm:neumann-exterior-inverse} identifies
this energy solution with the \(p\)-solution in
\eqref{eq:homogeneous-resolvent-identity}.  Taking gradients yields
\[
 \nabla L_B^{-1/2}u
 =\mathcal R_Bu-\int_0^\infty\nabla g_{B,s}u\,\dd s
 =\mathcal R_Bu-\nabla g_Bu.
\]
Equations \eqref{eq:explicit-potential-gradient} and
\eqref{eq:homogeneous-correction-bound} bound the two terms by
\(C_p\|u\|_p\).  The estimate extends to every \(u\in L^p(\Omega)\) by
density.

To this end, we verify the assertion under \eqref{eq:overlapping-models}.
Choose the transition regions of the cutoffs and the supports of the
auxiliary functions $\eta$ and $\theta$ inside
$B(0,R_c)\setminus\overline{B(0,R_-)}$, with the properties required
in Section~\ref{sec:exterior-reverse}.  In particular, $\phi_c$ vanishes
near the artificial boundary and $\phi_e$ vanishes on a neighborhood of
$\overline{B(0,R_-)}$.  Use $L_{0,c,B}$ for the core and $L_\infty$
for the end.  For the corresponding model functions $v$ and $w$, the
matching assumptions give 
\[
 \begin{aligned}
 L_B(\phi_cv)&=\phi_cL_{0,c,B}v+[L_B,\phi_c]v,\\
 L_B(\phi_ew)&=\phi_eL_\infty w+[L_B,\phi_e]w.
 \end{aligned}
\]
The analogous identities with $\rho_c$ and $\rho_e$ give the left
parametrix.  There is no coefficient-mismatch term, since each cutoff
is supported where the relevant coefficients agree.  Replace
$H_{c,B}$ by $1+L_{0,c,B}$ and every whole-space semigroup and inverse
by that of $L_\infty$.  The model estimates then use exactly
$p<p_{0,B}$ and $p<p_\infty$.  The exterior form commutators, the
Neumann balancing identities, and their compactness proofs are
unchanged.  Thus the homogeneous inverse theorems and the error
estimates hold with $P_B$ replaced by $\widehat P_B$.  The three steps
above complete the proof, with $p<n$ still required for Dirichlet data.
\end{proof}

\part{Applications}\label{part3}

\section{Applications}
\label{sec:applications}

\subsection{VMO coefficients}

We derive the bounded-core estimates within the proof of
Corollary~\ref{cor:intro-VMO}.  Theorem~\ref{thm:intro-forward} then yields
an alternative proof in every stated dimension through the unified
Poisson--homogeneous construction of Sections~\ref{sec:forward-parametrix}--\ref{sec:proof-forward}.

\begin{proof}[Proof of Corollary~\ref{cor:intro-VMO}]
Assume that $A\in\mathrm{VMO}(\R^n)$.    Shen's Dirichlet theorem \cite[Theorem~C]{Shen05} yields, for some
$\varepsilon_D>0$,
\begin{equation}\label{eq:Shen-core-ranges}
 p_{c,D}\geq
 \begin{cases}
  3+\varepsilon_D,
     &n\geq3\text{ and }\partial\Omega\text{ is Lipschitz},\\
  \infty,
     &\partial\Omega\text{ is }C^1.
 \end{cases}
\end{equation}
For Neumann data, Geng's conormal estimate
\cite[Theorem~1.2]{Geng12} supplies a number $\varepsilon_N>0$ such that the estimate below holds for every
\[
 (3+\varepsilon_N)'<q<2\quad(n\geq3),
 \qquad
 (4+\varepsilon_N)'<q<2\quad(n=2),
\]
after decreasing $\varepsilon_N$ if necessary.  Given
\(F\in C_c^\infty(\Omega_c;\R^n)\), there exists a unique (up to constant) solution $u\in W^{1,q}(\Omega_c)$ such that $L_{c,N}u = \textrm{div} F$, and $\| \nabla u\|_{L^q(\Omega_c)} \lesssim \|F\|_{L^q(\Omega_c)}$. Normalize $u=L_{c,N}^{-1}\diver F$ to have mean zero. Then the Poincar\'e inequality gives $\|u\|_{L^q(\Omega_c)}\lesssim \|F\|_{L^q(\Omega_c)}$. The bounded-core reverse estimate \cite[Theorem~1]{AuscherTchamitchian01} yields
\[
 \begin{aligned}
 \|L_{c,N}^{-1/2}\diver F\|_{L^q(\Omega_c)}
 &=\|L_{c,N}^{1/2}u\|_{L^q(\Omega_c)}\lesssim\|u\|_{L^q(\Omega_c)}+\|\nabla u\|_{L^q(\Omega_c)}
 \lesssim\|F\|_{L^q(\Omega_c)}.
 \end{aligned}
\]
The estimate extends to every \(F\in L^q(\Omega_c;\R^n)\) by continuity. It follows by duality that 
\begin{equation}\label{eq:Geng-core-ranges}
 p_{c,N}\geq
 \begin{cases}
  3+\varepsilon_N,
     &n\geq3\text{ and }\partial\Omega\text{ is Lipschitz},\\
  4+\varepsilon_N,
     &n=2\text{ and }\partial\Omega\text{ is Lipschitz},\\
  \infty,
     &\partial\Omega\text{ is }C^1.
 \end{cases}
\end{equation}
The $C^1$ conclusions in \eqref{eq:Shen-core-ranges} and
\eqref{eq:Geng-core-ranges} follow from Auscher--Qafsaoui \cite{AuscherQafsaoui}. The cited bounded-domain estimates are commonly stated for the unshifted
operator.  They imply the shifted estimates in the definition of \(p_{c,B}\) through
\[
 \nabla(1+L_{c,B})^{-1/2}
 =
 \nabla L_{c,B}^{-1/2}
 \left(\frac{L_{c,B}}{1+L_{c,B}}\right)^{1/2}.
\]
The last factor is bounded on $L^p$, $1<p<\infty$, by the multiplier theorem from \cite{Stein70}.

Combining \eqref{eq:Shen-core-ranges}--\eqref{eq:Geng-core-ranges} with Theorem~\ref{thm:intro-forward} and the standard low-exponent estimate \eqref{eq:exterior-low-Riesz} yields the desired result.

\end{proof}

\begin{corollary}\label{cor:layered}
Assume that $\mathbb R^n\setminus\Omega\subset B(0,R)$ and, after an orthogonal change of coordinates, $A(x)=\mathfrak A(x_1)$, where $\mathfrak A:\mathbb R\to\mathbb R^{n\times n}$ is measurable,
real, symmetric and uniformly elliptic.  Suppose additionally that
\[
 \mathfrak A(t)=A_*
 \qquad\text{for } |t|\le 2R,
\]
where $A_*$ is a fixed positive definite matrix. Suppose that $\partial\Omega$ is $C^1$. Then for $n\ge 2$,
\[
 \nabla L_N^{-1/2}:L^p(\Omega)\longrightarrow
 L^p(\Omega;\mathbb R^n),
 \qquad 1<p<\infty,
\]
and
\[
 \nabla L_D^{-1/2}:L^p(\Omega)\longrightarrow
 L^p(\Omega;\mathbb R^n),
 \qquad 1<p<n.
\]
for $n\ge3$.
\end{corollary}

\begin{proof}
We first claim that $p_L=\infty$.  Fix $s>2$.  Since $A(x)$ depends only on $x_1$, its mean oscillation in the transverse variables is identically zero.  Dong--Kim's local estimate \cite[Corollary~8.3]{DK10}, applied with $q=2$ and $\lambda=0$, therefore yields
\[
 \left(\fint_B|\nabla u|^s\right)^{1/s}
 \le C_s
 \left(\fint_{2B}|\nabla u|^2\right)^{1/2}
\]
for every ball $B$ and every weak solution of $Lu=0$ in $2B$. The estimate holds at every scale because the transverse oscillation is zero.  Shen's whole-space characterization \cite[Theorem~A]{Shen05} consequently implies that $\nabla L^{-1/2}$ is bounded on $L^s(\mathbb R^n)$.  Since $s>2$ was arbitrary, $p_L=\infty$. Choose $R<R_c<2R$ and let $\Omega_c=\Omega\cap B(0,R_c)$. The coefficient equals the constant matrix $A_*$ throughout $\Omega_c$, and both components of $\partial\Omega_c$ are $C^1$. The all-exponent bounded-domain estimates therefore imply $p_{c,D}=p_{c,N}=\infty$. The conclusion now follows from Theorem~\ref{thm:intro-forward}.
\end{proof}

\begin{remark}
For example, take
\[
 A(x)=I+\alpha\mathbf 1_{\{x_1>3R\}}e_1\otimes e_1,
 \qquad \alpha>0.
\]
This belongs to the one-variable coefficient class treated by
Dong--Kim \cite[Theorem~2.2(iii)]{DK10}.
Every ball centered on $\{x_1=3R\}$ has mean oscillation
$\alpha/2$, independently of its radius. Thus
$A\notin\mathrm{VMO}(\mathbb R^n)$, while
Corollary~\ref{cor:layered} applies without any smallness
assumption on $\alpha$.
\end{remark}

\subsection{Nilpotent Lie group ends}

Theorem~\ref{thm:intro-connected} also applies to models that are not
products of Euclidean spaces and compact manifolds.

\begin{corollary}\label{cor:nilpotent-ends}
Let $\ell\geq2$, and let $G_1,\ldots,G_\ell$ be simply connected nilpotent Lie groups of
the same topological dimension $N\geq3$, each equipped with a
left-invariant Riemannian metric.  If $M$ is a connected sum of these
models with an arbitrary smooth compact gluing, then
\[
 \|\Delta_M^{1/2}f\|_{L^p(M)}
 \leq C_p\|df\|_{L^p(M)},
 \qquad f\in C_c^\infty(M),\quad 1<p<\infty.
\]
\end{corollary}

\begin{proof}
Write $V_j(r)$ for the volume of a radius-$r$ ball in $G_j$; left
invariance makes it independent of the center.  The volume and heat
kernel estimates on nilpotent groups give integers $Q_j\geq N$ such
that
\[
 V_j(r)\simeq
 \begin{cases}
  r^N,&0<r\leq1,\\
  r^{Q_j},&r\geq1,
 \end{cases}
 \qquad
 h_{j,t}(x,y)\leq
 \frac{C}{V_j(\sqrt t)}
 \exp\!\left(-\frac{d_j(x,y)^2}{Ct}\right);
\]
see \cite[Chapter~IV]{VaropoulosSaloffCosteCoulhon92}.  Consequently
$\|e^{-t\Delta_{G_j}}\|_{L^1\to L^\infty}\lesssim t^{-N/2}$ for
every $t>0$.  The semigroup characterization of the Sobolev inequality
\cite{Varopoulos85} yields
\[
 \|h\|_{L^{2N/(N-2)}(G_j)}\lesssim\|dh\|_{L^2(G_j)}.
\]
The Ricci curvature is bounded below by homogeneity.  Alexopoulos's
theorem \cite[Theorem~2]{Alexopoulos92}, applied to a left-invariant orthonormal
frame, gives the Riesz transform bound on each $G_j$ for every
$1<q<\infty$.  Duality therefore gives the model reverse inequality
for every $1<p<\infty$.  Theorem~\ref{thm:intro-connected}, with
$\nu=N$, proves the asserted estimate on $M$ for $1<p<N$.

To cover the remaining exponents, each model is non-collapsed, doubling,
and satisfies the displayed Gaussian upper bound.  Its volume growth
also satisfies
$V_j(R)/V_j(r)\simeq(R/r)^{Q_j}$ for $R\geq r\geq1$, with
$Q_j>2$.  The theorem of Jiang--Li--Lin
\cite{JiangLiLin25} thus gives the Riesz transform bound
on $M$ for $1<q<2$.  By duality, the reverse inequality holds on $M$
for $2<p<\infty$; at $p=2$ it is the energy identity.  Together with
the first range, this proves the corollary.
\end{proof}

For a concrete example, take $G_1=\R^3$ and
$G_2=\mathrm{Heis}_3$, the three-dimensional real Heisenberg group
with a left-invariant Riemannian metric.  The growth degrees at infinity
are $Q_1=3$ and $Q_2=4$.  A three-dimensional product
$\R^k\times\mathcal M$, with $\mathcal M$ compact, has growth degree
$k\leq3$, so the Heisenberg end is outside the class treated in
\cite{He26}. Note that the connected sum is also not globally doubling. Nevertheless, Corollary~\ref{cor:nilpotent-ends} gives
the reverse inequality for every $1<p<\infty$.

\section{Open problems}
\label{sec:open-problems}

The value of the parametrix is not limited to the main results above: it
turns global stability questions into explicit questions about compact
commutators and the zero-energy modes of the models.  The following
problems indicate where that point of view could lead.

\begin{problem}
\label{prob:minimal-reverse-gluing}
Replace the lower Ricci bound and common Sobolev dimension in
Theorem~\ref{thm:intro-connected} by assumptions expressed entirely in
terms of the model semigroups and their behavior on the fixed transition
annuli.  In particular, determine an endwise formulation that permits
ends with no common Sobolev dimension, without reducing to an explicit
product geometry.  Corollary~\ref{cor:nilpotent-ends} already allows
different volume-growth degrees when the models share a Sobolev dimension.
\end{problem}

Only local elliptic regularity is used on each fixed gluing annulus, whereas
the large-time decay comes from the model Sobolev inequality.  This suggests
that a lower Ricci bound is stronger than the mechanism itself.  The hard
point is to formulate stable off-diagonal derivative estimates with enough
uniformity to sum the Poisson error when the ends have incompatible
large-scale dimensions.

\begin{problem}
\label{prob:wave-gluing}
Let $M$ be a connected sum with complete end models
$M_1,\ldots,M_\ell$. Construct the wave parametrix for $\cos{t\sqrt{\Delta_M}}$. Find assumptions on the model propagators and the compact core and describe uniform $L^p$, dispersive, local-energy, or spectral-multiplier estimates for $\cos(t\sqrt{\Delta_M})$.
\end{problem}

Finite propagation controls the construction for bounded times, but an
arbitrary compact gluing may create trapped geodesics.  A global theorem
must therefore identify a quantitative hypothesis, such as local-energy
decay or a cutoff-resolvent estimate, which controls repeated returns of
the wave to the compact core.

\begin{problem}
\label{prob:spectral-multipliers}
Assume the model ends satisfy a H\"ormander spectral-multiplier theorem or a
quadratic estimate for a holomorphic functional calculus.  Determine whether compact gluing preserves the same differentiability threshold.
\end{problem}

The present proof controls one Poisson integral and one spatial derivative.
A multiplier theorem would require uniform vector-valued bounds for
families such as $(s\sqrt L)^ke^{-s\sqrt L}$ and a square-function estimate
for the form-valued Green correction.  This is a genuinely stronger demand
than estimating each fixed semigroup operator separately and may require
$R$-bounded, rather than scalar, off-diagonal estimates.

\smallskip

\textbf{Acknowledgments.}
Some of the results are included in the author's PhD thesis \cite{HePhD}. The author would like to thank his PhD supervisor, Adam Sikora, and Professor Lixin Yan for reading the notes and for their helpful comments. D.~He is supported by the China Postdoctoral Science Foundation (No.~2026T190831) and the NNSF of China (No.~42381947).

\bibliographystyle{abbrv}
\bibliography{references}

@book{AuscherTchamitchian98,
  author    = {Auscher, Pascal and Tchamitchian, Philippe},
  title     = {Square Root Problem for Divergence Operators and Related Topics},
  series    = {Ast\'erisque},
  number    = {249},
  pages     = {viii+172},
  publisher = {Soci\'et\'e Math\'ematique de France},
  year      = {1998}
}

@article{AuscherTchamitchian01,
  author  = {Auscher, Pascal and Tchamitchian, Philippe},
  title   = {Square roots of elliptic second order divergence operators on strongly {L}ipschitz domains: {$L^p$} theory},
  journal = {Math. Ann.},
  volume  = {320},
  number  = {3},
  year    = {2001},
  pages   = {577--623},
  doi     = {10.1007/PL00004487}
}

@article{AuscherTchamitchian03,
  author  = {Auscher, Pascal and Tchamitchian, Philippe},
  title   = {Square roots of elliptic second order divergence operators on
             strongly {L}ipschitz domains: {$L^2$} theory},
  journal = {J. Anal. Math.},
  volume  = {90},
  year    = {2003},
  pages   = {1--12},
  doi     = {10.1007/BF02786549}
}

@article{AHLMT02,
  author  = {Auscher, Pascal and Hofmann, Steve and Lacey, Michael and
             McIntosh, Alan and Tchamitchian, Philippe},
  title   = {The solution of the {K}ato square root problem for second
             order elliptic operators on {$\mathbb R^n$}},
  journal = {Ann. of Math. (2)},
  volume  = {156},
  number  = {2},
  year    = {2002},
  pages   = {633--654},
  doi     = {10.2307/3597201}
}

@article{AuscherCoulhon05,
  author  = {Auscher, Pascal and Coulhon, Thierry},
  title   = {Riesz transform on manifolds and {P}oincar\'e inequalities},
  journal = {Ann. Sc. Norm. Super. Pisa Cl. Sci. (5)},
  volume  = {4},
  number  = {3},
  year    = {2005},
  pages   = {531--555},
  doi     = {10.2422/2036-2145.2005.3.07}
}

@article{AuscherQafsaoui,
  author  = {Auscher, Pascal and Qafsaoui, Mahmoud},
  title   = {Observations on {$W^{1,p}$} estimates for divergence elliptic equations with {VMO} coefficients},
  journal = {Boll. Unione Mat. Ital. Sez. B Artic. Ric. Mat. (8)},
  volume  = {5},
  number  = {2},
  year    = {2002},
  pages   = {487--509}
}

@article{Carron07,
  author  = {Carron, Gilles},
  title   = {Riesz transforms on connected sums},
  journal = {Ann. Inst. Fourier (Grenoble)},
  volume  = {57},
  number  = {7},
  year    = {2007},
  pages   = {2329--2343},
  doi     = {10.5802/aif.2334}
}

@article{CarronCoulhonHassell06,
  author  = {Carron, Gilles and Coulhon, Thierry and Hassell, Andrew},
  title   = {Riesz transform and {$L^p$}-cohomology for manifolds with {E}uclidean ends},
  journal = {Duke Math. J.},
  volume  = {133},
  number  = {1},
  year    = {2006},
  pages   = {59--93},
  doi     = {10.1215/S0012-7094-06-13313-6}
}

@article{CoulhonDuong,
  author  = {Coulhon, Thierry and Duong, Xuan Thinh},
  title   = {Riesz transforms for {$1\leq p\leq2$}},
  journal = {Trans. Amer. Math. Soc.},
  volume  = {351},
  number  = {3},
  year    = {1999},
  pages   = {1151--1169},
  doi     = {10.1090/S0002-9947-99-02090-5}
}

@article{Devyver15,
  author  = {Devyver, Baptiste},
  title   = {A perturbation result for the {R}iesz transform},
  journal = {Ann. Sc. Norm. Super. Pisa Cl. Sci. (5)},
  volume  = {14},
  number  = {3},
  year    = {2015},
  pages   = {937--964},
  doi     = {10.2422/2036-2145.201107_010}
}

@article{Byun05,
  author  = {Byun, Sun-Sig},
  title   = {Elliptic equations with {BMO} coefficients in {L}ipschitz domains},
  journal = {Trans. Amer. Math. Soc.},
  volume  = {357},
  number  = {3},
  year    = {2005},
  pages   = {1025--1046},
  doi     = {10.1090/S0002-9947-04-03624-4}
}

@article{Geng12,
  author  = {Geng, Jun},
  title   = {{$W^{1,p}$} estimates for elliptic problems with {N}eumann boundary conditions in {L}ipschitz domains},
  journal = {Adv. Math.},
  volume  = {229},
  number  = {4},
  year    = {2012},
  pages   = {2427--2448},
  doi     = {10.1016/j.aim.2012.01.004}
}

@book{GyryaSaloffCoste,
  author    = {Gyrya, Pavel and Saloff-Coste, Laurent},
  title     = {{N}eumann and {D}irichlet Heat Kernels in Inner Uniform Domains},
  series    = {Ast\'erisque},
  number    = {336},
  pages     = {vii+145},
  publisher = {Soci\'et\'e Math\'ematique de France},
  year      = {2011},
  doi       = {10.24033/ast.899}
}

@article{He26,
  author  = {He, Dangyang},
  title   = {Reverse {R}iesz inequality on manifolds with ends},
  journal = {J. Funct. Anal.},
  volume  = {290},
  number  = {10},
  year    = {2026},
  pages   = {111424},
  doi     = {10.1016/j.jfa.2026.111424}
}

@article{JiangLin24,
  author  = {Jiang, Renjin and Lin, Fanghua},
  title   = {Riesz transform on exterior {L}ipschitz domains and applications},
  journal = {Adv. Math.},
  volume  = {453},
  year    = {2024},
  pages   = {109852},
  doi     = {10.1016/j.aim.2024.109852}
}

@article{JiangYang25,
  author  = {Jiang, Renjin and Yang, Sibei},
  title   = {Some remarks on {R}iesz transforms on exterior {L}ipschitz domains},
  journal = {Forum Math. Sigma},
  volume  = {13},
  year    = {2025},
  pages   = {e58},
  doi     = {10.1017/fms.2025.19}
}

@article{HassellSikora09,
  author  = {Hassell, Andrew and Sikora, Adam},
  title   = {Riesz transforms in one dimension},
  journal = {Indiana Univ. Math. J.},
  volume  = {58},
  number  = {2},
  year    = {2009},
  pages   = {823--852},
  doi     = {10.1512/iumj.2009.58.3514}
}

@article{KillipVisanZhang16,
  author  = {Killip, Rowan and Vi\c{s}an, Monica and Zhang, Xiaoyi},
  title   = {Riesz transforms outside a convex obstacle},
  journal = {Int. Math. Res. Not. IMRN},
  volume  = {2016},
  number  = {19},
  year    = {2016},
  pages   = {5875--5921},
  doi     = {10.1093/imrn/rnv338}
}

@article{KaltonMitrea98,
  author  = {Kalton, Nigel J. and Mitrea, Marius},
  title   = {Stability results on interpolation scales of quasi-{B}anach
             spaces and applications},
  journal = {Trans. Amer. Math. Soc.},
  volume  = {350},
  number  = {10},
  year    = {1998},
  pages   = {3903--3922},
  doi     = {10.1090/S0002-9947-98-02008-X}
}

@book{BerghLofstrom,
  author    = {Bergh, J{\"o}ran and L{\"o}fstr{\"o}m, J{\"o}rgen},
  title     = {Interpolation Spaces: An Introduction},
  series    = {Grundlehren der mathematischen Wissenschaften},
  volume    = {223},
  publisher = {Springer-Verlag},
  address   = {Berlin},
  year      = {1976},
  doi       = {10.1007/978-3-642-66451-9}
}

@article{Shen05,
  author  = {Shen, Zhongwei},
  title   = {Bounds of {R}iesz transforms on {$L^p$} spaces for second order elliptic operators},
  journal = {Ann. Inst. Fourier (Grenoble)},
  volume  = {55},
  number  = {1},
  year    = {2005},
  pages   = {173--197},
  doi     = {10.5802/aif.2094}
}

@article{Sikora,
  author  = {Sikora, Adam},
  title   = {Riesz transform, Gaussian bounds and the method of wave equation},
  journal = {Math. Z.},
  volume  = {247},
  number  = {3},
  year    = {2004},
  pages   = {643--662},
  doi     = {10.1007/s00209-003-0639-3}
}

@article{Varopoulos85,
  author  = {Varopoulos, Nicholas Th.},
  title   = {{H}ardy--{L}ittlewood theory for semigroups},
  journal = {J. Funct. Anal.},
  volume  = {63},
  number  = {2},
  year    = {1985},
  pages   = {240--260},
  doi     = {10.1016/0022-1236(85)90075-4}
}

@article{ACDH04,
  author  = {Auscher, Pascal and Coulhon, Thierry and Duong, Xuan Thinh and Hofmann, Steve},
  title   = {Riesz transform on manifolds and heat kernel regularity},
  journal = {Ann. Sci. \`Ec. Norm. Sup\'er. (4)},
  volume  = {37},
  number  = {6},
  year    = {2004},
  pages   = {911--957},
  doi     = {10.1016/j.ansens.2004.10.003}
}

@article{CoulhonSikora08,
  author  = {Coulhon, Thierry and Sikora, Adam},
  title   = {Gaussian heat kernel upper bounds via the {Phragm\'en--Lindel\"of} theorem},
  journal = {Proc. Lond. Math. Soc. (3)},
  volume  = {96},
  number  = {2},
  year    = {2008},
  pages   = {507--544},
  doi     = {10.1112/plms/pdm050}
}

@book{Stein70,
  author    = {Stein, Elias M.},
  title     = {Topics in Harmonic Analysis Related to the
               {Littlewood--Paley} Theory},
  series    = {Annals of Mathematics Studies},
  volume    = {63},
  publisher = {Princeton University Press},
  address   = {Princeton, NJ},
  year      = {1970}
}

@book{AdamsFournier03,
  author    = {Adams, Robert A. and Fournier, John J. F.},
  title     = {Sobolev Spaces},
  edition   = {2},
  series    = {Pure and Applied Mathematics},
  volume    = {140},
  publisher = {Elsevier/Academic Press},
  address   = {Amsterdam},
  year      = {2003}
}

@book{Brezis11,
  author    = {Brezis, Haim},
  title     = {Functional Analysis, Sobolev Spaces and Partial Differential
               Equations},
  publisher = {Springer},
  address   = {New York},
  year      = {2011},
  doi       = {10.1007/978-0-387-70914-7}
}

@book{SteinSingular70,
  author    = {Stein, Elias M.},
  title     = {Singular Integrals and Differentiability Properties of
               Functions},
  series    = {Princeton Mathematical Series},
  volume    = {30},
  publisher = {Princeton University Press},
  address   = {Princeton, NJ},
  year      = {1970}
}

@article{Atkinson51,
  author  = {Atkinson, F. V.},
  title   = {The normal solvability of linear equations in normed spaces},
  journal = {Mat. Sb. (N.S.)},
  volume  = {28(70)},
  number  = {1},
  year    = {1951},
  pages   = {3--14}
}

@incollection{GeissertHeckHieber06,
  author    = {Gei{\ss}ert, Matthias and Heck, Horst and Hieber, Matthias},
  title     = {On the equation \(\operatorname{div} u=g\) and {B}ogovski\u{\i}'s
               operator in {S}obolev spaces of negative order},
  booktitle = {Partial Differential Equations and Functional Analysis},
  series    = {Operator Theory: Advances and Applications},
  volume    = {168},
  publisher = {Birkh{\"a}user},
  address   = {Basel},
  year      = {2006},
  pages     = {113--121},
  doi       = {10.1007/3-7643-7601-5_7}
}

@article{DK10,
  author  = {Dong, Hongjie and Kim, Doyoon},
  title   = {Elliptic equations in divergence form with partially {BMO}
             coefficients},
  journal = {Arch. Ration. Mech. Anal.},
  volume  = {196},
  number  = {1},
  pages   = {25--70},
  year    = {2010},
  doi     = {10.1007/s00205-009-0228-7}
}

@article {BCLS95,
    AUTHOR = {Bakry, D. and Coulhon, T. and Ledoux, M. and Saloff-Coste, L.},
     TITLE = {Sobolev inequalities in disguise},
   JOURNAL = {Indiana Univ. Math. J.},
  FJOURNAL = {Indiana University Mathematics Journal},
    VOLUME = {44},
      YEAR = {1995},
    NUMBER = {4},
     PAGES = {1033--1074},
      ISSN = {0022-2518,1943-5258},
   MRCLASS = {46E35 (26D10 35J99)},
  MRNUMBER = {1386760},
MRREVIEWER = {Steven\ George\ Krantz},
       DOI = {10.1512/iumj.1995.44.2019},
       URL = {https://doi.org/10.1512/iumj.1995.44.2019},
}

@article {Alexopoulos92,
    AUTHOR = {Alexopoulos, G.},
     TITLE = {An application of homogenization theory to harmonic analysis:
              {H}arnack inequalities and {R}iesz transforms on {L}ie groups
              of polynomial growth},
   JOURNAL = {Canad. J. Math.},
  FJOURNAL = {Canadian Journal of Mathematics. Journal Canadien de
              Math\'ematiques},
    VOLUME = {44},
      YEAR = {1992},
    NUMBER = {4},
     PAGES = {691--727},
      ISSN = {0008-414X,1496-4279},
   MRCLASS = {22E15 (22E30 43A80)},
  MRNUMBER = {1178564},
MRREVIEWER = {Ewa\ Damek},
       DOI = {10.4153/CJM-1992-042-x},
       URL = {https://doi.org/10.4153/CJM-1992-042-x},
}

@article {JiangLiLin25,
    AUTHOR = {Jiang, Ren-Jin and Li, Hong-Quan and Lin, Hai-Bo},
     TITLE = {Riesz transform on manifolds with ends of different volume
              growth for {$1<p<2$}},
   JOURNAL = {J. Geom. Anal.},
  FJOURNAL = {Journal of Geometric Analysis},
    VOLUME = {35},
      YEAR = {2025},
    NUMBER = {11},
     PAGES = {Paper No. 339, 42},
      ISSN = {1050-6926,1559-002X},
   MRCLASS = {42B20 (43A85 58J35)},
  MRNUMBER = {4955105},
       DOI = {10.1007/s12220-025-02169-z},
       URL = {https://doi.org/10.1007/s12220-025-02169-z},
}

@book{VaropoulosSaloffCosteCoulhon92,
  author    = {Varopoulos, N. Th. and Saloff-Coste, L. and Coulhon, T.},
  title     = {Analysis and Geometry on Groups},
  series    = {Cambridge Tracts in Mathematics},
  volume    = {100},
  publisher = {Cambridge University Press},
  address   = {Cambridge},
  year      = {1992}
}

@article {CJKS,
    AUTHOR = {Coulhon, Thierry and Jiang, Renjin and Koskela, Pekka and
              Sikora, Adam},
     TITLE = {Gradient estimates for heat kernels and harmonic functions},
   JOURNAL = {J. Funct. Anal.},
  FJOURNAL = {Journal of Functional Analysis},
    VOLUME = {278},
      YEAR = {2020},
    NUMBER = {8},
     PAGES = {108398, 67},
      ISSN = {0022-1236,1096-0783},
   MRCLASS = {53C23 (31C05 31C25 31E05 35K08 43A85 58J05 58J35)},
  MRNUMBER = {4056992},
MRREVIEWER = {Nelia\ Charalambous},
       DOI = {10.1016/j.jfa.2019.108398},
       URL = {https://doi.org/10.1016/j.jfa.2019.108398},
}

@article {Jiang21,
    AUTHOR = {Jiang, Renjin},
     TITLE = {Riesz transform via heat kernel and harmonic functions on
              non-compact manifolds},
   JOURNAL = {Adv. Math.},
  FJOURNAL = {Advances in Mathematics},
    VOLUME = {377},
      YEAR = {2021},
     PAGES = {Paper No. 107464, 50},
      ISSN = {0001-8708,1090-2082},
   MRCLASS = {58J35 (35B65 35K05 42B20 58J05)},
  MRNUMBER = {4186007},
MRREVIEWER = {Li\ Chen},
       DOI = {10.1016/j.aim.2020.107464},
       URL = {https://doi.org/10.1016/j.aim.2020.107464},
}

@article {JerisonKenig95,
    AUTHOR = {Jerison, David and Kenig, Carlos E.},
     TITLE = {The inhomogeneous {D}irichlet problem in {L}ipschitz domains},
   JOURNAL = {J. Funct. Anal.},
  FJOURNAL = {Journal of Functional Analysis},
    VOLUME = {130},
      YEAR = {1995},
    NUMBER = {1},
     PAGES = {161--219},
      ISSN = {0022-1236,1096-0783},
   MRCLASS = {35J25 (46E35)},
  MRNUMBER = {1331981},
MRREVIEWER = {H.\ Triebel},
       DOI = {10.1006/jfan.1995.1067},
       URL = {https://doi.org/10.1006/jfan.1995.1067},
}

@article {Zanger2000,
    AUTHOR = {Zanger, Daniel Z.},
     TITLE = {The inhomogeneous {N}eumann problem in {L}ipschitz domains},
   JOURNAL = {Comm. Partial Differential Equations},
  FJOURNAL = {Communications in Partial Differential Equations},
    VOLUME = {25},
      YEAR = {2000},
    NUMBER = {9-10},
     PAGES = {1771--1808},
      ISSN = {0360-5302,1532-4133},
   MRCLASS = {35J25},
  MRNUMBER = {1778780},
MRREVIEWER = {H.\ Triebel},
       DOI = {10.1080/03605302.2000.10824220},
       URL = {https://doi.org/10.1080/03605302.2000.10824220},
}

@article {HePhD,
    AUTHOR = {He, Dangyang},
     TITLE = {Analysis of the {L}aplacian on a class of nondoubling
              connected sums and manifolds with quadratically decaying
              {R}icci curvature},
   JOURNAL = {Bull. Aust. Math. Soc.},
  FJOURNAL = {Bulletin of the Australian Mathematical Society},
    VOLUME = {113},
      YEAR = {2026},
    NUMBER = {1},
     PAGES = {173--174},
      ISSN = {0004-9727,1755-1633},
   MRCLASS = {42B15 (42B20 46E35 49Q20 58J35)},
  MRNUMBER = {5021248},
       DOI = {10.1017/s000497272510049x},
       URL = {https://doi.org/10.1017/s000497272510049x},
}

\end{document}